\documentclass[12pt]{amsart}
\usepackage[margin=1in]{geometry}
\usepackage{amsfonts,amssymb,stmaryrd,amscd,amsmath,latexsym,amsbsy,amsthm}
\usepackage{hyperref}
\usepackage{mathtools,nccmath}
\usepackage{enumitem}
\usepackage{multicol}

\newtheorem{theorem}{Theorem}[section]
\newtheorem{lemma}[theorem]{Lemma}
\newtheorem{proposition}[theorem]{Proposition}
\newtheorem{corollary}[theorem]{Corollary}
\theoremstyle{definition}
\newtheorem{definition}[theorem]{Definition}
\newtheorem{example}[theorem]{Example}

\newtheorem{conjecture}[theorem]{Conjecture}
\newtheorem{remark}[theorem]{Remark}

\newcommand{\bla}{{\mbox{\boldmath$\lambda$}}}
\newcommand{\bLa}{{\mbox{\boldmath$\Lambda$}}}

\newcommand{\Hom}{\text{Hom}}

\newcommand{\g}{\mathfrak{g}}

\newcommand{\h}{\mathfrak{h}}
\newcommand{\n}{\mathfrak{n}}
\renewcommand{\b}{\mathfrak{b}}

\newcommand{\la}{\lambda}

\newcommand{\bs}{\backslash}

\newcommand{\C}{{\mathbb C}}

\newcommand{\ben}{\begin{enumerate}}
\newcommand{\een}{\end{enumerate}}

\newcommand{\nc}{\newcommand}
\nc{\on}{\operatorname}
\nc{\wh}{\widehat}
\nc{\wt}{\widetilde}
\nc{\sw}{{\mathfrak s}{\mathfrak l}}
\nc{\ghat}{\wh{\g}}
\nc{\hhat}{\wh{\h}}
\nc{\mc}{\mathcal}
\nc{\Bun}{\on{Bun}}
\nc{\ol}{\overline}
\nc{\OO}{\mathcal O}
\nc{\pone}{{\mathbb P}^1}
\nc{\pa}{\partial}
\nc{\Pic}{\on{Pic}}
\nc{\ga}{\gamma}
\nc{\orr}{\underline}
\nc{\mbb}{\mathbb}
\nc{\mbf}{\mathbf}
\nc{\V}{{\mc V}}

\newcommand{\norm}[1]{\left\lVert#1\right\rVert}
\newcommand{\bmu}{{\mbox{\boldmath$\mu$}}}
\newcommand{\bbe}{{\mbox{\boldmath$\beta$}}}

\theoremstyle{plain}

\newtheorem*{sol}{Solution}

\theoremstyle{definition}

\theoremstyle{remark}

\usepackage{color}

\newcommand{\solu}[1]{\begin{sol}{\bf  (\ref{#1})}}

\newcommand{\mcV}{\mathcal V}

\newcommand{\mcE}{\mathcal E}

\nc{\Bunc}{\on{Bun}^{\circ}}

\def\g{\mathfrak{g}}
\def\R{\mathbb{R}}

\def\D{\mathcal{D}}

\def\h{\mathfrak{h}}

\def\Z{\mathbb{Z}}

\def\M{\mathcal{M}}

\def\Hom{\mathrm{Hom}}

\def\n{\mathfrak{n}}

\def\b{\mathfrak{b}}

\def\Fq{\mathbb{F}_q}
\def\LG{G^\vee}
\def\lg{\g^\vee}
\def\lb{{\mathfrak b}^\vee}

\nc{\ppart}{(\!(t)\!)}
\nc{\zpart}{(\!(z)\!)}
\nc{\hl}{h_{\ell}}
\nc{\hr}{h_{r}}

\nc{\Op}{\on{Op}}

\nc{\AD}{{\mathbb A}}
\nc{\al}{\alpha}

\usepackage{xcolor}
\usepackage{amscd}
\nc{\mb}{\mathbf} %DIF > 
\numberwithin{equation}{section}

\begin{document}

\title[A general framework for the analytic Langlands correspondence]{A general framework for the analytic Langlands correspondence}

\author{Pavel Etingof}

\address{Department of Mathematics, MIT, Cambridge, MA 02139, USA}

\author{Edward Frenkel}

\address{Department of Mathematics, University of California,
  Berkeley, CA 94720, USA}

\author{David Kazhdan}

\address{Einstein Institute of Mathematics, Edmond J. Safra Campus,
  Givaat Ram The Hebrew University of Jerusalem, Jerusalem, 91904,
  Israel}

\maketitle

\centerline{\bf  To Corrado De Concini with admiration} 

\begin{abstract} We discuss a general framework for the analytic
  Langlands correspondence over an arbitrary local field $F$
  introduced and studied in our works \cite{EFK1,EFK2,EFK3}, in
  particular including non-split and twisted settings. Then we
  specialize to the archimedean cases ($F=\Bbb C$ and $F=\Bbb R$) and
  give a (mostly conjectural) description of the spectrum of the Hecke
  operators in various cases in terms of opers satisfying suitable
  reality conditions, as predicted in part in \cite{EFK2,EFK3} and
  \cite{GW}. We also describe an analogue of the Langlands
  functoriality principle in
  the analytic Langlands correspondence over $\C$ and show that it is
  compatible with the results and conjectures of \cite{EFK2}.
  Finally, we apply the tools of the analytic Langlands correspondence
  over archimedean fields in genus zero to the Gaudin model and its
  generalizations, as well as their $q$-deformations.
\end{abstract} 

\tableofcontents

\section{Introduction} 

\subsection{Overview} 
Let $X$ be a smooth irreducible projective curve of genus ${\rm g}>1$ and $G$ a reductive algebraic group, both defined over a local field $F$. Let ${\rm Bun}^\circ_G(X)$ be the variety of regularly stable principal $G$-bundles on $X$ (see e.g. \cite{EFK1}, Section 1.1). In \cite{EFK1,EFK2,EFK3}, motivated in part by the works \cite{BK1,Ko,La,Te}, we proposed the {\bf  analytic Langlands correspondence}, which is the study of the spectrum of {\bf  Hecke operators} acting on the space of complex-valued half-densities on ${\rm Bun}^\circ_G(X)(F)$. Justifying its name, this correspondence is a natural analytic analog 
of two previously known settings of Langlands correspondence -- {\bf arithmetic} (for curves over a finite field) and {\bf geometric} (involving $\mathcal D$-modules on complex curves and on moduli stacks of $G$-bundles on curves), to both of which it is actually intimately related. We also proposed a ramified generalization of the analytic Langlands correspondence for $G$-bundles with level structure at an $F$-rational effective divisor $D\subset X$ (which allows one to also consider ${\rm g}=0,1$).

However, previously for simplicity we focused  on the basic case when $G$ and $D$ are split, and in fact mostly assumed that $F=\Bbb C$. Yet the natural generality 
of the (arithmetic) Langlands correspondence is that of a flat reductive group scheme 
$\mathcal G$ over $X$, for example, one defined by an action of the \'etale fundamental group $\pi_1^{\rm et}(X)$ on $G$. Roughly speaking, one of the main goals of this paper is to discuss the (ramified) analytic Langlands correspondence in this more general setting, when $F$ is arbitrary and $G$, $D$ are not necessarily split; in particular, this means that we need focus on Galois-theoretic aspects of the theory. We also allow twists by $Z(G)$-gerbes on $X$ and, in the ramified case, by unitary representations of the group of changes of the level structure. 
A detailed discussion of the general framework of the analytic Langlands correspondence with a focus on these additional features is the subject of the first half of the paper.

The second half of the paper is dedicated to the archimedean cases, $F=\Bbb C$ and $F=\Bbb R$. In these cases, as shown in \cite{EFK2}, the Hecke operators $H_{x,\la}$ commute with the {\bf quantum Hitchin Hamiltonians} and also satisfy a certain differential equation with respect to $x\in X$ involving these Hamiltonians, called the {\bf universal oper equation}. As a result, the joint spectrum of the Hecke and quantum Hitchin Hamiltonians is (conjecturally) labeled by $G^\vee$-{\bf opers} $L$ on $X$ satisfying a certain topological condition called a {\bf  reality condition}. For $F=\Bbb C$, as explained in \cite{EFK2}, this is the condition that the monodromy of $L$ can be conjugated into an inner (conjecturally, split) real form $G^\vee_\Bbb R$ of $G^\vee$. On the other hand, for $F=\Bbb R$ the theory depends on several additional pieces of data (an antiholomorphic involution of $X$, an inner class of $G$, a form $G^\sigma$ of $G$ in this class attached to each oval of $X(\Bbb R)$, etc.), and 
the exact form of the reality condition depends on these details. We work out this condition 
in several examples for $G=GL_1$ and $PGL_2$, generalizing \cite{EFK3}, Subsection 4.7 and 
\cite{GW}, Section 6. 

Finally, we explain how the {\bf generalized Bethe Ansatz method for the Gaudin model} 
can be viewed (in several ways) as an instance of the tamely ramified 
analytic Langlands correspondence in genus $0$ over $\Bbb R$ and $\Bbb C$. 
Interestingly, the role of Hecke operators in this setting is played by 
{\bf Baxter's Q-operator of the Gaudin model}, the $q\to 1$
limit of the Q-operator of the XXZ quantum spin chain introduced by R. Baxter. 
Motivated by this, we discuss a $q$-deformation of the archimedean analytic 
Langlands correspondence in genus $0$. 

\subsection{Summary of the main results} 
To summarize, in this paper we accomplish the following. 

1. We formulate the problem of (ramified) analytic Langlands correspondence in a general setting when the group $G$ and the ramification divisor $D\subset X$ are not necessarily split and ramification points carry unitary representations of the groups of changes of level structure, in presence of possible twists by a $Z$-gerbe on $X$ and an action of $\pi_1(X)$ 
on the root datum of $G$. We state the conjecture on compactness of Hecke operators, which leads to discreteness of their spectrum. 

2. For $G=PGL_2$ and genus $0$ with split ramification divisor, in the tamely ramified case when ramification points carry principal series representations of $G(F)$, we compute explicitly the Hecke operators $H_x$ (which are known to be compact in this case) and find their asymptotics near ramification points. This gives the asymptotics of eigenvalues of $H_x$. 

3. In the cases $F=\Bbb C$ and $F=\Bbb R$, we conjecturally describe the spectrum 
of Hecke operators in the setting of (1) in terms of opers with
monodromy representation satisfying suitable ``reality conditions". We
formulate reality conditions in various special cases, and show that
spectral opers must satisfy these conditions. In particular, we work
out reality conditions for $G=GL_1$, $F=\Bbb R$, and also $G=PGL_2$,
$F=\Bbb R$ in the ramified and tamely ramified cases. We describe
behavior of eigenvalues of $H_x$ near real ovals of $X$ (when they are
present) in various situations.

4. In the case $F=\Bbb C$, we discuss in detail in Section \ref{prin}
the Hecke operators corresponding to the {\em principal weights} of
$G^\vee$ (such that the corresponding irreducible representation of
$\g^\vee$ remains irreducible under a principal $\sw_2$ subalgebra),
following the approach of \cite{EFK2}, Section 5. We then consider the
general case. In Subsection \ref{irred}, we prove that for a generic
$G^\vee$-oper $\chi$ on a curve of genus ${\rm g} > 1$, the Zariski
closure $M_\chi$ of the image of its monodromy representation is equal
to $G^\vee$. We also show if $G^\vee$ is connected simple group of
adjoint type, then for any $G^\vee$-oper $\chi$, the group $M_\chi$ is
a connected simple subgroup of $G^\vee$ that contains a principal
$PGL_2$ subgroup of $G^\vee$. This allows us to elucidate the
conjectural formula for the eigenvalues of the Hecke operators
presented in \cite{EFK2}, Conjecture 5.1 (see Subsection
\ref{genwt}). We also use these results in Subsection \ref{funct} to
describe an analogue of the Langlands functoriality principle in the
analytic Langlands over $\C$
correspondence and show that it is compatible with the results and
conjectures of \cite{EFK2}.

5. We describe several settings of the Gaudin model in terms of the
analytic Langlands correspondence, enabling us to describe the spectrum
of the commuting Gaudin Hamiltonians. In particular, we
reinterpret the known description of the spectrum of the Gaudin
Hamiltonians in the case of the tensor product of finite-dimensional
representations in terms of monodromy-free opers \cite{F1,R} as a
special case of real analytic Langlands correspondence. Namely, we use
our results in the ``quaternionic'' real case discussed in Subsection
\ref{quatov} to obtain this description of the spectrum. We also
explain the connection to the Bethe Ansatz method of diagonalization
of these Hamiltonians.

Our interpretation of the Gaudin model in terms of the analytic
Langlands correspondence allows us to extend it to a more general
setting in which the space of states is infinite-dimensional and the
traditional Bethe Ansatz methods do not apply; for example, the tensor
product of unitary principal series representations.

In all of these cases, the key new element is the
existence of the Hecke operators commuting with the Gaudin
Hamiltonians and the fact that they satisfy differential equations
(the universal oper equations), which can be used to describe the
analytic properties of the $G^\vee$-opers encoding the possible joint
eigenvalues of the Gaudin Hamiltonians.  We show that the Hecke
operators in this setting are closely related to the analogue of the
Baxter $Q$-operator in the Gaudin model. We use this relation to
discuss a possible $q$-deformation of the analytic Langlands
correspondence, where the role of Hecke operators is played by
Baxter's $Q$-operators of the XXZ model (and its generalization from
$SL_2$ to a general simple complex Lie group).

We note that some results on the spectra of the Gaudin Hamiltonians
for $SL_2$ in the real case were obtained in \cite{NRS} from the point
of view of $N=2$ SUSY 4d gauge theory (see also \cite{JLN}). It would
be interesting to see if there is a connection between their results
and ours.

\subsection{Structure of the paper} The structure of the paper is as follows. 

In Section 2 we give an informal description of the general framework for the analytic Langlands correspondence over an arbitrary local field $F$. We start with reviewing the theory of forms of reductive groups, especially over non-archimedean local fields (the Kneser-Bruhat-Tits theory). Then we explain that in the unramified case the appropriate moduli space of $F$-rational $G$-bundles is determined by an {\bf  inner class} $C(s)$ of $G$ over $F$, $s\in H^1(F,{\rm Out}G)$. We also explain that every such bundle $P$ and a geometric point $x$ of $X$ defines an $F$-{\bf  form} $G^\sigma$ of $G$ over the field $E_x$ of definition of $x$ in $C(s)$. Thus in the tamely ramified case ($\ell=1$) with ramification divisor $D\subset X$, the input data of the theory is a choice, for each $t\in D$, 
of a form $G^{\sigma}$ of $G$ over $E_{t}$ in $C(s)$ and a unitary representation $V$ of $G^{\sigma}(E_{t})$. For such input data, we define Hecke operators on the $L^2$ space of the moduli space of $F$-rational $G$-bundles, pose the spectral problem for them, consider various examples and present several conjectures that we've proved in some interesting special cases. 

In Section 3 we discuss the case $F=\Bbb C$. We start by reviewing the
results and conjectures of \cite{EFK1,EFK2,EFK3} on parametrizing the
spectrum $\Sigma$ of Hecke and quantum Hitchin Hamiltonians by real
opers on the Riemann surface $X(\Bbb C)$. Next, we consider the
differential equations satisfied by the Hecke operators $H_\la$ in
various cases: $G=PGL_2$ and $X=\pone$ with parabolic structures at
finitely many points in Subsection \ref{Differential equations for
  Hecke operators} (recalling and generalizing the results of
\cite{EFK3}); $G=PGL_n$, $X$ of genus ${\rm g} > 1$, and
$\la=\omega_1$ in Subsection \ref{diffPGLn} (recalling the results of
\cite{EFK2}); a generalization of the latter case to an arbitrary
principal $\la$ in Subsection \ref{prin}; and the general case in
Subsections \ref{gencase} and \ref{genwt}. In Subsection \ref{irred}
we describe the Zariski closures of the monodromy representations of
$G^\vee$-opers in the case that $G^\vee$ is a connected simple group
of adjoint type and ${\rm g} > 1$. In particular, we show that the
monodromy of a generic $G^\vee$-oper is dense in $G^\vee$. We use this
in Subsection \ref{genwt} to elucidate some results of \cite{EFK2} and
in Subsection \ref{funct} to describe an analogue of the Langlands
functoriality principle in
the analytic Langlands correspondence. In Subsection \ref{twisted} we
discuss the twists by $Z(G)$-gerbes and by ${\rm Aut}G$-torsors on
$X$. We also consider the twists by unitary representations at
ramification points for $G=PGL_2$. For all these twists, we describe
(conjecturally) the set of opers parametrizing $\Sigma$.

In Section 4, we consider the case $F=\Bbb R$. In this case the curve $X$ is a Riemann surface with an antiholomorphic involution $\tau$. We first assume that $\tau$ has no fixed points on $X$ (i.e., $X(\Bbb R)=\emptyset$) and review the conjectures from \cite{GW}, Section 6 
on the opers that are expected to label the eigenspaces of Hecke operators. We also show following \cite{W} that these conjectures hold for $G=GL_1$. Then we proceed to the general case, when $X(\Bbb R)$ may be nonempty, discussed in \cite{GW}, Subsection 6.3. This case is more complicated since it involves boundary conditions for oper solutions on the ovals of $X(\Bbb R)$. We describe what happens for $G=GL_1$, and then propose a conjectural reality condition for spectral opers for $G=SL_2$. 
In this case, we have two forms of $G$ (both inner) -- the split form and the compact form, giving 
rise to two types of ovals -- real and quaternionic, respectively. We propose the boundary conditions for spectral opers on both real and quaternionic ovals. We also explain what happens 
in presence of ramification points, and in the case $X=\Bbb P^1$ with the usual real structure and all ramification points real, we recover the description of spectral opers from \cite{EFK3}, Subsection 4.7 (``balanced" opers).   

Finally, in Section 5 we interpret the Gaudin model and
  its generalization in terms of the analytic Langlands
  correspondence. In particular, for $G=SL_2$, we derive the
  description of the spectrum of the Gaudin Hamiltonians in terms of
  monodromy-free $PGL_2$-opers \cite{F1,R} from a special case of the
  analytic Langlands correspondence over $\Bbb R$ (namely, for the
  compact form of $SL_2(\C)$). We give a description of the spectrum
of the Gaudin Hamiltonians on the tensor product of representations of
the unitary principal series of $SL_2({\mathbb R})$ in terms of
balanced $PGL_2$-opers. We also discuss a $q$-deformation of the
analytic Langlands correspondence in genus $0$ and its connection with
Bethe Ansatz method for the XXZ model (a $q$-deformation of the Gaudin
model).

\subsection{Acknowledgments} 
It is our pleasure to dedicate this paper to Corrado De Concini, who 
has been an inspiration to us mathematically and personally for many decades.  

The material of Section 4 arose from our discussions with Davide
Gaiotto and Edward Witten in the Spring of 2021. We are very grateful
to them for these discussions, which motivated us to write this
section. We thank Dima Arinkin for providing a proof of
Theorem \ref{zar}. We also thank Nikita Nekrasov for valuable
discussions, and Ilya Dumanskii for communication regarding the
material of Subsection \ref{pgl2}.

P. E.'s work was partially supported by the NSF grant DMS-2001318. 
The project has received funding from ERC under grant agreement No. 669655.

\section{Analytic Langlands correspondence over a general local field}
\subsection{Varieties over arbitrary fields} \label{formvar}

Let $L$ be a separably closed field and $X$ be an algebraic variety defined over 
$L$. Since $X$ is reduced, $X(L)=X(\overline L)$, where $\overline L\supset L$ is an algebraic closure of $L$; in particular,  $X(L)\ne \emptyset$. 
For $\gamma\in {\rm Aut}L$, let ${}^\gamma X$  be the twist of $X$ by $\gamma$.
Thus $X(L)={}^\gamma X(L)$ and 
the structure sheaf of ${}^\gamma X$ is obtained from the one of $X$ 
by twisting the scalar multiplication by $\gamma$. 

Now let $F$ be any field. Let $F_{\rm sep}$ be a separable closure of $F$ 
and for a field extension $F\subset E\subset F_{\rm sep}$, 
$\Gamma_E:={\rm Gal}(F_{\rm sep}/E)$ be the absolute Galois group of $E$. 

Let $X$ be a variety defined over $F$. Let $\tau(\gamma): {}^\gamma X\to X$, $\gamma\in \Gamma_F$, be the collection of isomorphisms defining 
the $F$-structure on $X$. This defines an action 
$$
\gamma\mapsto \tau(\gamma): X(F_{\rm sep})\to 
X(F_{\rm sep})
$$ 
of $\Gamma_F$ on $X(F_{\rm sep})$, and 
the set $X(E)$ of $E$-points of $X$ is the fixed point set of $\Gamma_E\subset \Gamma_F$ on $X(F_{\rm sep})$. 

Let $x\in X(F_{\rm sep})$ be a point with stabilizer $\Gamma_x\subset \Gamma_F$. We will call the field $E:=F_{\rm sep}^{\Gamma_x}$ the {\bf  field of definition of $x$}; thus $\Gamma_x=\Gamma_E$.\footnote{Note that the field of definition $E$ of a point $x$ is not just an abstract field extension of $F$, but rather a subfield of $F_{\rm sep}$ containing $F$; so if fields of definition of two points are not Galois then they may be distinct but nevertheless isomorphic.} Let $X_E\subset X(F_{\rm sep})$ be the subset of points with field of definition $E$. Then $X(E)$ is the disjoint union of $X_K$ over all $F\subset K\subset E$. 

\begin{example} Let $F=\Bbb R$, then $F_{\rm sep}=\overline F=\Bbb C$ and
$\Gamma_F=\Bbb Z/2$. An  $\Bbb R$-structure on $X$ 
is given by an antiholomorphic involution $\tau: X(\Bbb C)\to X(\Bbb C)$. 
Furthermore, the real locus $X_{\Bbb R}=X(\Bbb R)=X(\Bbb C)^\tau$ 
is just the set of fixed points of $\tau$ on $X(\Bbb C)$, and $X_{\Bbb C}=X(\Bbb C)\setminus X(\Bbb R)$. 
\end{example} 

\subsection{Forms of reductive groups} \label{formsredgps}
Let $G$ be a split connected reductive algebraic group over $\Bbb Z$ corresponding to a 
polarized root datum $\Delta_G$. In particular, this means that we fix a positive Borel subgroup $B\subset G$ 
and a maximal torus $T\subset B$. Let $G_{\rm ad}$ be the corresponding adjoint group.  For any field $E$ let ${\rm Aut}G(E):={\rm Aut}\Delta_G\ltimes G_{\rm ad}(E)$; here ${\rm Aut}\Delta_G={\rm Out}G$ is the group of outer automorphisms of $G$.\footnote{The group ${\rm Aut}\Delta_G$ may be infinite, for example if $G$ is an $n$-dimensional torus then ${\rm Aut}\Delta_G=GL_n(\Bbb Z)$. Thus ${\rm Aut}G$ is not an algebraic group, in general. More specifically, it is an algebraic group iff the center of $G$ has dimension $\le 1$ (e.g., for $G=GL_n$ or $G$ semisimple). But this is not important for our considerations.} This group acts on $G(K)$ for any field extension $K$ of $E$. 

Let $F$ be a field whose characteristic is either zero or coprime to the determinant of the Cartan matrix of $G$.\footnote{This assumption is not essential and is made to simplify the exposition.} 

The classification of forms of $G$ over $F$ is as described in Subsection \ref{formvar}, except that unlike the case of general varieties, we already have a distinguished $F$-form of $G$ 
(the {\bf  split form}) defining an action $g\mapsto \gamma(g)$ of $\Gamma_F$ on $G(F_{\rm sep})$, so we can describe all $F$-forms of $G$ by counting from this form, in terms of Galois cohomology. Namely, forms of $G$ over $F$ are parametrized by the (continuous) Galois cohomology 
$H^1(\Gamma_F, {\rm Aut}G(F_{\rm sep}))$ (\cite{S}). Specifically, let 
$\theta: \Gamma_F\to {\rm Aut}G(F_{\rm sep})$
be a 1-cocycle, i.e., 
$$
\theta(\gamma_1\gamma_2)=\theta(\gamma_1)\circ \gamma_1(\theta(\gamma_2)).
$$
Then we have an action $\sigma=\sigma_\theta$ of $\Gamma_F$ on $G(F_{\rm sep})$ 
given by 
$$
\sigma(\gamma)=\theta(\gamma)\circ \gamma,
$$
and conversely, an action $\sigma$ of $\Gamma_F$ such that 
for $\gamma\in \Gamma_F$, $\theta_\sigma(\gamma):=\sigma(\gamma)\circ \gamma^{-1}\in {\rm Aut}G(F_{\rm sep})$ gives rise to a 1-cocycle $\theta=\theta_\sigma$. 
The form $G^\sigma$ of $G$ corresponding to $\sigma$ 
(or $\theta$) is defined by its functor of points 
$$
G^\sigma(A):=G(A\otimes_F F_{\rm sep})^{\Gamma_F}
$$
for any commutative $F$-algebra $A$. In particular, for any field extension $F\subset E \subset F_{\rm sep}$, $G^\sigma(E)$ is the subgroup of fixed points of $\sigma(\Gamma_E)$. In other words, $G^\sigma(E)$ is the group of elements of $G(F_{\rm sep})$ satisfying the equation 
$$
\gamma(g)=\theta(\gamma)^{-1}(g)
$$
for all $\gamma\in \Gamma_E$. 

For example, if $\sigma=1$ then $G^\sigma=G_{\rm spl}$, the split form of $G$. 

Moreover, if an element $a\in {\rm Aut}G(F_{\rm sep})$ transforms a 1-cocycle $\theta_{\sigma_1}$ into a 1-cocycle $\theta_{\sigma_2}$ 
then it canonically defines an isomorphism $\chi(a): G^{\sigma_1}\to G^{\sigma_2}$ such that for any $b$ transforming $\theta_{\sigma_2}$ into $\theta_{\sigma_3}$ we have $\chi(ba)=\chi(b)\circ \chi(a)$. 
In other words, denoting by $\bold A=\bold A(F,G)$ the groupoid whose objects are 
$1$-cocycles $\theta:\Gamma_F\to {\rm Aut}G(F_{\rm sep})$ and morphisms from $\theta$ to $\theta'$ are 
elements $a\in {\rm Aut}G(F_{\rm sep})$ such that 
$\theta'(\gamma)=a\theta(\gamma)\gamma(a)^{-1}$, we obtain a functor $\theta_\sigma\mapsto G^{\sigma}$ 
from $\bold A$ to the category of algebraic $F$-groups. 
In particular, $G^\sigma$ depends only on the cohomology class $[\theta_\sigma]$ up to isomorphism. 

Since the sequence
\begin{equation}\label{groupseq}
1\to G_{\rm ad}(F_{\rm sep})\to {\rm Aut}G(F_{\rm sep})\to {\rm Aut}\Delta_G\to 1
\end{equation} 
splits, it
defines a short exact sequence of pointed sets 
\begin{equation}\label{h1seq}
1\to H^1(\Gamma_F,G_{\rm ad}(F_{\rm sep}))\to H^1(\Gamma_F,{\rm Aut}G(F_{\rm sep}))\to H^1(\Gamma_F,{\rm Aut}\Delta_G)\to 1,
\end{equation} 
where $H^1(\Gamma_F,{\rm Aut}\Delta_G)={\rm Hom}(\Gamma_F,{\rm Aut}\Delta_G)$/conjugation (note that $\Gamma_F$ acts trivially on $\Delta_G$ since we start with the split form). In other words, the second map in \eqref{h1seq} is injective and its image coincides with the kernel (i.e., the preimage of $1$) of the third map, which is surjective. 

Recall that a form $G^\sigma$ is called {\bf  inner} if $[\theta_\sigma]\in  H^1(\Gamma_F,G_{\rm ad}(F_{\rm sep}))$, i.e., if it projects to $1\in H^1(\Gamma_F,{\rm Aut}\Delta_G)$. The {\bf  inner class} of $G^\sigma$ is the collection 
$C(\sigma)$ of all forms $G^\eta$ of $G$ such that $[\theta_\sigma]$ and $[\theta_\eta]$ map to the same element of $H^1(\Gamma_F,{\rm Aut} \Delta_G)$. Thus inner classes are labeled by conjugacy classes $[s]$ of homomorphisms 
$s: \Gamma_F\to {\rm Aut}\Delta_G$. 
For example, $C(1)$ (the inner class of the split form) consists of all the inner forms of 
$G$.  

Note that since the sequence \eqref{groupseq} is 
canonically split, so is the sequence \eqref{h1seq}. Thus for every $s\in {\rm Hom}(\Gamma_F,{\rm Aut}\Delta_G)$ the inner class $C(s)$ has a canonical representative called the {\bf  quasi-split form} of $G$ in $C(s)$ and denoted by $G^s$; this is the only form of $G$ in $C(s)$ which has an $F$-rational Borel subgroup. For example, the quasi-split inner form is the split form. By \eqref{h1seq}, this implies that each inner class $C(s)$ can be canonically identified with $H^1(\Gamma_F,G^s_{\rm ad}(F_{\rm sep}))$ (with the action 
of $\Gamma_F$ defined by $s$). 

\subsection{Forms of reductive groups over a local field} 
Now let $F$ be a non-archimedean local field. We recall the theory of forms of 
reductive groups over $F$. We start with two classical theorems. 

\begin{theorem} \label{aux} (Kneser, Bruhat-Tits, \cite{BT}, 4.7) 
For a simply connected semisimple group $G$ over $F$, one has $H^1(\Gamma_F,G(F_{\rm sep}))=1$. 
\end{theorem} 

\begin{theorem}\label{tdu} (Tate duality for Galois cohomology, \cite{S}, II.5.2, Theorem 2). 
Let $A$ be a finite $\Gamma_F$-module of order prime to ${\rm char}(F)$ and $A^*:={\rm Hom}(A,\Bbb G_m)$. Then for $0\le i\le 2$, $H^i(\Gamma_F,A)$ is finite and 
we have a canonical isomorphism 
$$
H^{i}(\Gamma_F,A)\cong \Hom(H^{2-i}(\Gamma_F,A^*),
\Bbb Q/\Bbb Z). 
$$
\end{theorem} 

In particular, taking $i=2$, we get 
$H^2(\Gamma_F,A)\cong \Hom((A^*)^{\Gamma_F},\Bbb Q/\Bbb Z)=A(-1)_{\Gamma_F}$, the coinvariants of $\Gamma_F$ in the negative Tate twist $A(-1)$ of $A$. 

\begin{corollary}\label{coro} 
Let $G^s$ be a simply connected quasi-split semisimple group over $F$,
$G^s_{\rm ad}$ the corresponding adjoint group, and $Z^s(F_{\rm sep})$ the center 
of $G^s(F_{\rm sep})$ regarded as a $\Gamma_F$-module (with action defined by $s$). Then there is a natural inclusion 
$$
H^1(\Gamma_F,G_{\rm ad}^s(F_{\rm sep}))\hookrightarrow Z^s(F_{\rm sep})(-1)_{\Gamma_F}.
$$ 
\end{corollary} 

\begin{proof} By Theorem \ref{aux} and the ``long" exact sequence of Galois cohomology, 
we have an embedding $\xi: H^1(\Gamma_F,G_{\rm ad}^s(F_{\rm sep}))\hookrightarrow H^2(\Gamma_F,Z^s(F_{\rm sep}))$, so the result follows from Theorem \ref{tdu}, using that $p$ does not divide the determinant of the Cartan matrix of $G$. 
\end{proof} 

In fact, there is an even stronger result: 

\begin{corollary} The map of Corollary \ref{coro} is an isomorphism. 
\end{corollary} 

\begin{proof} It is known (\cite{K},\cite{T}) that the map of Corollary \ref{coro} is surjective, so the result follows. 
\end{proof} 

Thus we obtain the following corollary. Let $G_{\rm sc}^s$ be the universal cover of $G_{\rm ad}^s$. 

\begin{corollary}\label{classifi} Let $G$ be a split connected reductive group and 
$s\in {\rm Hom}(\Gamma_F,{\rm Aut}\Delta_G)$. Then the inner class $C(s)$ is naturally identified with the group $Z^s(F_{\rm sep})(-1)_{\Gamma_F}$, where $Z^s$ is the center of 
$G^s_{\rm sc}$.
\end{corollary}   

 Thus we get 

\begin{proposition}\label{findis} Over any local field, every inner class of a split connected reductive group $G$ is finite.\footnote{We have shown this when ${\rm char}(F)$ does not divide the determinant of the Cartan matrix of $F$, but this assumption is, in fact, unnecessary.} 
\end{proposition} 

\begin{proof} Recall that if $F=\Bbb R$ then we have the well known Cartan classification of forms of $G$ over $F$. This implies the proposition in the archimedean case. In the non-archimedean case the proposition follows from Corollary \ref{classifi}.  
\end{proof}

\begin{example} Let $G=SL_n$ and ${\rm char}(F)$ be coprime to $n$. 

(1) {\it  The split inner class.} In this case, $Z^s=\mu_n$. Thus by Corollary \ref{classifi}, forms in this class are parametrized by $\Bbb Z/n$. Namely, we identify $\Bbb Z/n$ with $\frac{1}{n}\Bbb Z/\Bbb Z\subset \Bbb Q/\Bbb Z={\rm Br}(F)$. Thus every $m\in \Bbb Z/n$ (represented by an integer in $[1,n]$) gives rise to a central division algebra $D_m$ over $F$ of dimension 
$(n/m)^2$, and the corresponding form $G^\sigma$ is $SL_{n/m,D_m}$ (i.e., 
$G^\sigma(E)=SL_{n/m}(E\otimes_F D_m)$ for a field extension $E\supset F$).

(2) {\it  The non-split inner class $C_L$ attached to a separable quadratic field extension 
$L$ of $F$ ($n\ge 3$).} The quadratic extension defines a character $\chi_L: \Gamma_F\to \pm 1$, which gives rise to an action of $\Gamma_F$ on $\Bbb Z$; we denote this module 
by $\Bbb Z_L$. Then $Z^s=\mu_n\otimes_{\Bbb Z} \Bbb Z_L$, so the inner class $C_L$ is parametrized by $(\Bbb Z_L/n)_{\Gamma_F}$, which is trivial if $n$ is odd and $\Bbb Z/2$ if $n$ is even. Thus we should expect the quasi-split form and also an additional form for even $n$. And indeed, this is the case: these forms are the corresponding {\bf  special unitary groups}. 
Namely, recall that if $N: L^\times \to F^\times$ is the norm map then 
$|F^\times/N(L^\times)|=2$. Thus we have two equivalence classes of nondegenerate Hermitian forms on $L^n$ up to isomorphism -- $B_+$ whose determinant is a norm and $B_-$ whose determinant is not. So we have the corresponding special unitary groups 
$SU_{n,L}^+$ and $SU_{n,L}^-$ (namely, $SU^\pm_{n,L}(E)=SU^\pm_n(E\otimes_F L)$ for a field extension $E\supset F$). However, if $n$ is odd then for any non-norm $a\in F^\times$, the forms $aB_-$ and $B_+$ are equivalent, so these two groups are isomorphic. 

We note that $SU^+_{2,L}=SL_{1,D_2}$ and $SU^-_{2,L}=SL_2$ (for any $L$).  
\end{example} 

\begin{remark} If ${\rm char}(F)=0$ then
$F$ has finitely many extensions 
of every fixed degree. Thus for any finite group $\Gamma$, there are finitely many homomorphisms 
$\Gamma_F\to \Gamma$. Also by a theorem of Jordan and Zassenhaus, for each $n$ there are finitely many finite subgroups of $GL_n(\Bbb Z)$ up to isomorphism, so 
the number of homomorphisms $\Gamma_F\to GL_n(\Bbb Z)$ 
(i.e., of $F$-forms of the $n$-dimensional torus) is finite. It follows that the number of inner classes over $F$ of any split connected reductive group is finite as well. 

This is, however, false in positive characteristic $p>0$, as in this case 
$|{\rm Hom}(\Gamma_F,\Bbb Z/p)|=\infty$. For example, for $p=2$ 
there are infinitely many separable quadratic extensions of $F$, so there are infinitely many inner classes of $SL_n$ for $n\ge 3$. But this is not going to matter for us here. 
\end{remark} 

Finally, we have 

\begin{proposition}\label{finii}
Let $T$ be an abelian reductive group over $F$. 
Then 

(i) there exists a finite subgroup $R\subset H^1(\Gamma_F,T(F_{\rm sep})) $ such that 
$H^1(\Gamma_F,T(F_{\rm sep}))/R$ embeds into $H^1(\Gamma_F,(T/T_0)(F_{\rm sep}))$, where 
$T_0$ is the identity component of $T$;

(ii) If ${\rm char}(F)$ is coprime to the order of $T/T_0$ then 
the group $H^1(\Gamma_F,T(F_{\rm sep}))$ 
is finite. 
\end{proposition} 

\begin{proof} By Theorem \ref{tdu}, (ii) follows from (i), so it remains to prove (i). 

We have an exact sequence 
$$
H^1(\Gamma_F,T_0(F_{\rm sep}))\to H^1(\Gamma_F,T(F_{\rm sep}))\to H^1(\Gamma_F,(T/T_0)(F_{\rm sep})).
$$
Thus it suffices to prove (i) when $T=T_0$ is connected (a torus), i.e., to show that in this case $H^1(\Gamma_F,T(F_{\rm sep}))$ is finite. Let $\Gamma_E\subset \Gamma_F$ be the kernel of the action of $\Gamma_F$ on $T$. 
Then we have the inflation-restriction exact sequence 
$$
0\to H^1(\Gamma_F/\Gamma_E,T(E))\to H^1(\Gamma_F,T(F_{\rm sep}))\to H^1(\Gamma_E,T(F_{\rm sep}))^{\Gamma_F/\Gamma_E}.
$$
By Hilbert Theorem 90, the last term vanishes, so $H^1(\Gamma_F,T(F_{\rm sep}))\cong 
H^1(\Gamma_F/\Gamma_E,T(E))$. 
But the group $H^1(\Gamma_F/\Gamma_E,T(E))$
has exponent dividing $N=[E:F]$, so it is finite. This implies the result. 
\end{proof} 

\subsection{Principal $G$-bundles}  \label{princibun}

Let $F,F_{\rm sep},\Gamma_F$ be as above, $X$ be an algebraic $F$-variety, and $G$ an algebraic $F$-group. We denote by $Bun _G(X)$ the $F$-stack of principal $G$-bundles on $X$. Then $ \Gamma_F $ acts on the set $Bun _G(X)(F_{\rm sep})$.
If a $G$-bundle  $P \in Bun _G(X)( F_{\rm sep}) $ is defined by transition functions 
$g_{ij}: U_i\cap U_j\to G$ where $\{ U_i\}$ is an open cover of $X$, and $\gamma \in \Gamma_F $, then we can define the $G$-bundle  $P^\gamma $ by the transition functions $g_{ij}^\gamma : U_i ^\gamma \cap U_j ^\gamma\to G$. It is clear that this definition is independent on choices and gives rise to an action of $\Gamma_F$ on ${\rm Bun}_G(X)(F_{\rm sep})$. 

Now let us write this definition slightly more explicitly. Let $\theta: \Gamma_F\to {\rm Aut}G(F_{\rm sep})$ be a 1-cocycle, $\sigma(\gamma):=\theta(\gamma)\circ \gamma$ be the corresponding action of $\Gamma_F$, and $G^\sigma$ be the form of $G$ over $F$ defined by $\sigma$. Also let $\tau$ be the natural action of $ \Gamma_F$ on $X ( F_{\rm sep})$ (the $F$-structure on $X$). Then given a principal $G^\sigma$-bundle $P$ on $X$ defined over $F_{\rm sep}$, we have the $G^\sigma$-bundle $P^\gamma=(\sigma,\tau)(\gamma)P$ with transition functions $g_{ij}^\gamma=g^{(\sigma,\tau,\gamma)}_{ij}$ given by 
$$
g^{(\sigma,\tau,\gamma)}_{ij}=\sigma(\gamma)\circ g_{ij}\circ \tau(\gamma)^{-1}.
$$
Let ${\rm Bun}_{G,\sigma}(X,\tau):=Bun_G(X)(F_{\rm sep})^{\Gamma_F}$ be the set of fixed points of this action of $\Gamma_F$.

Now suppose that $G$ is a split connected reductive group over $\Bbb Z$. 
If $h: \Gamma_F\to G_{\rm ad}^\sigma(F_{\rm sep})$ is a 1-cocycle then 
$$
g^{(h\sigma,\tau,\gamma)}_{ij}=h(\gamma)\circ g^{(\sigma,\tau,\gamma)}_{ij}.
$$
Hence $h(\gamma)$ defines a canonical isomorphism
$(h\sigma,\tau)(\gamma)P\cong (\sigma,\tau)(\gamma)P$. 
In other words, $(\sigma,\tau)(\gamma)$ depends only on the inner class $C(\sigma)$ of $\sigma$, i.e. on $s\in {\rm Hom}(\Gamma_F,{\rm Aut}\Delta_G)$ such that $C(\sigma)=C(s)$. Thus $(\sigma,\tau)(\gamma)= (s,\tau)(\gamma)$ for all $\gamma$ and ${\rm Bun}_{G,\sigma}(X,\tau)={\rm Bun}_{G,s}(X,\tau)$. 

\subsection{Moduli of $G$-bundles on a smooth projective curve}\label{smpro}

Let $X$ be a smooth irreducible projective curve of genus ${\rm g}\ge 2$ over $F$. In this case we have a notion of a {\bf  regularly stable} principal $G$-bundle on $X$, which is a stable bundle whose group of automorphisms  is the minimal possible, i.e., reduces to the center $Z$ of $G$. 

The set ${\rm Bun}^\circ_G(X)(F_{\rm sep})\subset {\rm Bun}_G(X)(F_{\rm sep})$ of regularly stable bundles is the set of $F_{\rm sep}$-points of a smooth algebraic variety ${\rm Bun}^\circ_G(X)$ of dimension $({\rm g-1})\dim G$ defined over $F$ (this is, in fact, the underlying variety of a stack which is the quotient of a variety by the trivial action of $Z$). Moreover, every pair $(s,\tau)$ 
defines a form ${\rm Bun}^\circ_G(X)_{s,\tau}$ of this variety.  

Let 
$$
{\rm Bun}^\circ_{G,s}(X,\tau):={\rm Bun}^\circ_G(X)_{s,\tau}(F)\subset {\rm Bun}_{G,s}(X,\tau)
$$
be the subset of isomorphism classes of regularly stable bundles.

Define a {\bf  pseudo}-$F$-{\bf  structure} on 
$P\in {\rm Bun}_{G,s}(X,\tau)$   
to be a collection of isomorphisms
$$
A(\gamma): (s,\tau)(\gamma)P\to P,\ \gamma\in \Gamma_F.
$$ 
Such a data defines a 2-cocycle $a=a_A$ 
on $\Gamma_F$ with coefficients in $Z^s(F_{\rm sep})$ such that 
\begin{equation}\label{condi1}
A(\gamma_1\gamma_2)=A(\gamma_1)\circ (s,\tau)(\gamma_1)(A(\gamma_2))\circ a(\gamma_1,\gamma_2). 
\end{equation} 
We will say that a pseudo-$F$-structure $A$ is an $F$-{\bf  structure}
if $a_A=1$. 

Any two pseudo-$F$-structures
$A,A'$ on $P$ differ by a 
1-cochain $c: \Gamma_F\to Z^s(F_{\rm sep})$, 
and $a_A/a_{A'}=dc$. Thus the class $[a_A]\in H^2(\Gamma_F,Z^s(F_{\rm sep}))$ does not depend on $A$ 
and only depends on $P$, so we'll denote it by $\alpha_P$. So we obtain 

\begin{lemma} 
We have a decomposition
$$
{\rm Bun}^\circ_{G,s}(X,\tau)=
\sqcup_{\alpha\in H^2(\Gamma_F,Z^s(F_{\rm sep}))}{\rm Bun}^\circ_{G,s,\alpha}(X,\tau),
$$ 
where ${\rm Bun}^\circ_{G,s,\alpha}(X,\tau)$ is the subset of $P$ with $\alpha_P=\alpha$. 
\end{lemma} 

In fact, it is more natural to consider a Galois covering of 
the set ${\rm Bun}^\circ_{G,s,\alpha}(X,\tau)$
which keeps track of the isomorphism $A$. 
To this end, fix a 2-cocycle $a$ representing $\alpha$. 
Then for a principal bundle 
$P\in {\rm Bun}^\circ_{G,s,\alpha}(X,\tau)$
a solution $A$ of \eqref{condi1} is unique 
up to multiplication by a 1-cocycle $c: 
\Gamma_F\to Z^s(F_{\rm sep})$. On the other hand, 
if $c$ is a coboundary then $A$ and $c A$ 
are equivalent by an element of $Z(F_{\rm sep})$. 
Thus the set of solutions $A$ of \eqref{condi1} up to isomorphism
is a torsor over $H^1(\Gamma_F,Z^s(F_{\rm sep}))$. 
In other words, the set ${\rm Bun}^\circ_{G,s,a}(X,\tau)$ of 
isomorphism classes of pseudo-$F$-structures $(P,A)$ satisfying \eqref{condi1} with $a_A=a$ is a $H^1(\Gamma_F,Z^s(F_{\rm sep}))$-torsor over ${\rm Bun}^\circ_{G,s,\alpha}(X,\tau)$. 

Furthermore, if $[a]=[a']=\alpha$ and $a/a'=dc$ 
then multiplication by the 1-cochain $c$ defines a bijection 
$\nu_c: {\rm Bun}^\circ_{G,s,a}(X,\tau)\cong {\rm Bun}^\circ_{G,s,a'}(X,\tau)$ which depends on $c$ only up to coboundaries, and  
if $c$ is a 1-cocycle (so $a=a'$), this recovers the action of $H^1(\Gamma_F,Z^s(F_{\rm sep}))$ on the fibers of the projection 
$$
{\rm Bun}^\circ_{G,s,a}(X,\tau)\to {\rm Bun}^\circ_{G,s,\alpha}(X,\tau).
$$ 

In other words, for any $\alpha$ we may consider the groupoid $\bold H_\alpha$ 
whose objects are 2-cocycles $a$ with $[a]=\alpha$ and morphisms 
are ${\rm Hom}(a,a')=\lbrace c: a/a'=dc\rbrace$ with composition defined by addition. Then canonically we have an $\bold H_\alpha$-set ${\rm Bun}^\circ_{G,s,\alpha}(X,\tau)$ (a functor $\bold H_\alpha\to {\bf  Sets}$), which is the collection of sets ${\rm Bun}^\circ_{G,s,a}(X,\tau),[a]=\alpha$ with an action of the groupoid $\bold H_\alpha$.

\subsection{The form of $G$ attached to a principal bundle and a point}\label{formpt}
 
Now let $x\in X(F_{\rm sep})$ be a point with  field of definition $E$ and stabilizer $\Gamma_x=\Gamma_E\subset \Gamma_F$. Let $P\in {\rm Bun}^\circ_{G,s}(X,\tau)$
and $P_x$ be the fiber of $P$ at $x$. Fix a pseudo-$F$-structure 
$A$ on $P$. Then for $\gamma\in \Gamma_E$
we have an isomorphism $A(\gamma): P_x\to P_x$. By \eqref{condi1}, if we identify $P_x$ with $G$, we obtain a 1-cocycle $\theta: \Gamma_E\to G_{\rm ad}^s(F_{\rm sep})$; indeed, the $a_A$-factor  goes away upon projection to the adjoint group. Moreover, this cocycle is independent on the choice of $A$, since two different choices differ by a central element of $G$ which maps to $1$ in $G_{\rm ad}$. 

When we change the identification $P_x\cong G$ by $g\in G(F_{\rm sep})$, the cocycle $\theta$ changes by the coboundary $dg$, so we obtain a well defined cohomology class $[\theta]\in H^1(\Gamma_E,G_{\rm ad}^s(F_{\rm sep}))$. This class defines an action $\sigma=\sigma_\theta$ 
of $\Gamma_E$ on $G(F_{\rm sep})$.
Thus we get 

\begin{proposition}\label{assign} The above procedure assigns to a bundle $P\in {\rm Bun}_{G,s}^\circ(X,\tau)$ and a point $x\in X(F_{\rm sep})$ with field of definition $E$, an $E$-form $G^\sigma$ of $G$ in the inner class $C_E(s)$.
\end{proposition}

Moreover, the connecting homomorphism 
$$
\xi: H^1(\Gamma_E,G_{\rm ad}^s(F_{\rm sep}))\to H^2(\Gamma_E,Z^s(F_{\rm sep}))
$$ 
maps $[\theta]$ to $\alpha_P|_{\Gamma_E}$. It follows that if $X(E)\ne \emptyset$ then $\alpha_P|_{\Gamma_E}$ comes from an element of $H^2(\Gamma_E,Z^s_0(F_{\rm sep}))$, where $Z_0:=Z\cap [G,G]$. 

\begin{example} Let $G=\Bbb G_m$, $s=1$ (the split 1-dimensional torus). 
Then 
$$
H^2(\Gamma_F,Z^s(F_{\rm sep}))=H^2(\Gamma_F,F_{\rm sep}^\times)={\rm Br}(F),
$$ 
the Brauer group of $F$. However, if $X(E)\ne \emptyset$ and $\alpha\in {\rm Br}(F)$ with ${\rm Bun}^\circ_{G,s,\alpha}(X,\tau)\ne \emptyset$ then the image of $\alpha$ in ${\rm Br}(E)$ is trivial, i.e., 
the central division $F$-algebra $D_\alpha$ splits over $E$. 
Thus, assuming that $E$ is a Galois extension of $F$, 
we get that $\alpha$ belongs to ${\rm Br}(E/F):=H^2({\rm Gal}(E/F),E_{\rm sep}^\times)$, the relative Brauer group which is a finite subgroup of ${\rm Br}(F)$.

Thus all components 
${\rm Bun}^\circ_{G,s,\alpha}(X,\tau)$ are empty except finitely many. 
It is not hard to show that the same is true in general. 
\end{example} 

\subsection{Principal bundles on curves over a local field}\label{pblocfie}

Now let $F$ be a local field. Then ${\rm Bun}^\circ_{G,s,\alpha}(X,\tau)$  is an analytic $F$-manifold of dimension $({\rm g}-1)\dim G$. 
Thus by Proposition \ref{finii}, for any $a$ with $[a]=\alpha$, 
${\rm Bun}^\circ_{G,s,a}(X,\tau)$ is also an analytic manifold of this dimension (as it is a finite covering of 
${\rm Bun}^\circ_{G,s,\alpha}(X,\tau)$). 
Note that these manifolds are non-empty for $\alpha=1$ even if $X(F)=\emptyset$ (although they might be empty for some $\alpha\ne 1$). 

Let $E/F$ be  a finite extension and 
 $X_E\subset X(F_{\rm sep})$ be the subset of points with field of definition equal to  $E$. 
Then $X_E$ is an open subset of $X(E)$, hence a 1-dimensional analytic $E$-manifold. 
 
\begin{corollary}\label{uniq} If $F$ is non-archimedean and $P\in {\rm Bun}^\circ_{G,s,\alpha}(X,\tau)$ then the $E$-form $G^\sigma$ of $G$ in the class $C_E(s)$ attached to $P$ and $x\in X_E$ 
in Subsection \ref{formpt} is independent on $x$. 
\end{corollary} 

\begin{proof} By Theorem \ref{aux} the map $\xi$ is injective.
Thus there exists a unique $\theta$ up to coboundaries 
such that $\xi([\theta])=\alpha$, and we have $\sigma=\sigma_\theta$.
\end{proof} 

However, for $F=\Bbb R$, Corollary \ref{uniq} is not true (even for $G=SL_2$).
In this case $X(\Bbb R)=X_{\Bbb R}$ is a union of ovals, and 
the form $G^\sigma$ attached to a given bundle $P$ and $x\in X(\Bbb R)$ 
is only {\bf locally} constant in $x$, i.e., may depend on the oval 
to which $x$ belongs. This leads to interesting topological phenomena 
described in Subsection \ref{requatco} below. 

\subsection{Hecke operators} \label{heop}
As before, let $F$ be a local field and $a$ a 2-cocycle such that $[a]=\alpha$. Consider the Hilbert space 
$$
\mathcal H(s,\tau,a):=L^2({\rm Bun}^\circ_{G,s,a}(X,\tau))
$$ 
of square-integrable half-densities on the analytic $F$-manifold ${\rm Bun}^\circ_{G,s,a}(X,\tau)$. The collection of Hilbert spaces $\mathcal H(s,\tau,a)$, $[a]=\alpha$ is an $\bold H_\alpha$-Hilbert space (unitary representation of the groupoid $\bold H_\alpha$) which we will denote
by $\mathcal H(s,\tau,\alpha)$.  
We have a decomposition 
$$
\mathcal H(s,\tau,\alpha)=\oplus_{\chi}\mathcal H(s,\tau,\alpha,\chi),
$$
where $\mathcal H(s,\tau,\alpha,\chi)$ is the isotypic component of 
the character $\chi: H^1(\Gamma_F,Z^s(F_{\rm sep}))\to \Bbb C^\times$. 
Note that $\mathcal H(s,\tau,\alpha,\chi)$ is a well defined Hilbert space up to 
scaling by a phase factor. Namely, it is canonically isomorphic up to a phase factor to the space $L^2({\rm Bun}_{G,s,\alpha}^\circ(X,\tau),\mathcal L_\chi)$ of half-densities with values in $\mathcal L_\chi$, where $\mathcal L_\chi$ is the complex line bundle 
over ${\rm Bun}_{G,s,\alpha}^\circ(X,\tau)$ associated 
to the principal bundle ${\rm Bun}_{G,s,a}^\circ(X,\tau)\to {\rm Bun}_{G,s,\alpha}^\circ(X,\tau)$ via the character $\chi$ (namely, the line bundle $\mathcal L_\chi$ 
is independent of $a$ up to scaling by a phase factor). 

We would like  to define the action of commuting {\bf  Hecke operators} on $\mathcal H(s,\tau,\alpha) $ and to find 
the joint spectral decomposition of the algebra generated by these operators, see \cite{EFK2}. In more down-to-earth terms, 
these should be operators on the Hilbert space $\mathcal H(s,\tau,a)$ for any fixed choice of the representative $a$ of $\alpha$ which commute with the action of $H^1(\Gamma_F,Z^s(F_{\rm sep}))$, i.e., preserve 
the spaces $\mathcal H(s,\tau,\alpha,\chi)$. We view the eigenfunctions of these operators as the automorphic forms in the setting of the analytic Langlands correspondence.

The Hecke operators are defined as follows. Let $\Lambda$ be the coweight lattice of $G$ and $\lambda\in \Lambda^+$ be a dominant coweight. Then we define a variety ${\mathcal Z}_\lambda$ called the (regularly stable) {\bf  Hecke correspondence} equipped with a map 
$$
(p_1,p_2,q): {\mathcal Z}_\lambda\to {\rm Bun}_{G}^\circ(X)^2\times X,
$$  
see \cite{BD}. Namely, ${\mathcal Z}_\lambda$ consists of triples $T=(P_1,P_2,x)$ where $P_1,P_2\in {\rm Bun}_{G}^\circ(X)$ are principal $G$-bundles 
on $X$ which are identified outside $x$ and are in relative position $\lambda$ at $x$ (see \cite{EFK2}, p.4), and $p_i(T)=P_i$, $q(T)=x$. For $x\in X(F)$ let ${\mathcal Z}_{\lambda,s,\tau,x}$ be the set of pairs $(P_1,P_2)$ such that $(P_1,P_2,x)\in {\mathcal Z}_\lambda(F_{\rm sep})$ and $P_1,P_2$ are $F$-rational with respect to $(s,\tau)$, i.e., belong to 
${\rm Bun}_{G,s}(X,\tau)$.\footnote{For ${\mathcal Z}_{\lambda,s,\tau,x}$ to be non-empty, $\lambda$ must be invariant under the action of $\Gamma_F$ on $\Delta_G$ via $s$.} It is easy to see that if $(P_1,P_2)\in {\mathcal Z}_{\lambda,s,\tau,x}$ then $\alpha_{P_1}=\alpha_{P_2}$. 
Thus we may ``define" the Hecke operator 
$H_{x,\lambda}$ on $\mathcal H(s,\tau,\alpha)$ by the formula 
$$
(H_{x,\lambda}\psi)(P)=\int_{{\mathcal Z}_{\lambda,s,\tau,x}(P)}\psi(Q)\norm{dQ},
$$
where ${\mathcal Z}_{\lambda,s,\tau,x}(P):=\lbrace{Q: (P,Q)\in {\mathcal Z}_{\lambda,s,\tau,x}\rbrace}$ and $dQ$ is an appropriate canonically defined algebraic volume element introduced by Beilinson and Drinfeld in \cite{BD}, see also \cite{EFK2}, Theorem 1.1. If $(\lambda,\rho_G)\in \Bbb Z+\frac{1}{2}$, where $\rho_G$ is the half-sum of positive roots of $G$, then $dQ$ depends on a choice of a spin structure on $X$, but in any case $\norm{dQ}$ is independent of choices. Thus the domain of integration ${\mathcal Z}_{\lambda,s,\tau,x}(P)$ is (non-canonically) isomorphic to the set ${\rm Gr}^\lambda_{G,\sigma}(F)$ of fixed points of the cell ${\rm Gr}^\lambda_G$ in the affine Grassmannian ${\rm Gr}_G$ over $F_{\rm sep}$ under the $\Gamma_F$-action corresponding to the $F$-form $G^\sigma$ of $G$ defined by $(P,x)$ (\cite{EFK2}, Introduction and Section 5). The element $dQ$ is a $-(\lambda,\rho_G)$-form on $X$, thus the family of operators $\lbrace H_{x,\lambda},x\in X(F)\rbrace$ is an operator-valued $-(\lambda,\rho_G)$-density on $X(F)$. 

The word ``define" is in quotation marks because we can prove the convergence of these integrals only in some special cases. However, we expect that if 
$\psi$ is a smooth function (locally constant in the non-archimedean case) compactly supported in the locus of sufficiently generic bundles then the integral $H_{x,\lambda}\psi$ is convergent
and defines a function whose restriction to the open dense subset of {\bf  very stable}\footnote{A $G$-bundle $P$ on $X$ is said to be very stable if it does not admit a nonzero nilpotent Higgs field, i.e.
a section of $\Omega^1(X,{\rm ad}P)$ taking values in the nilpotent cone of ${\rm Lie}G$. It is known that 
very stable bundles are stable and form a dense open set in the variety of stable bundles, see \cite{Z} and references therein.} bundles is smooth. This is indeed easy to see when the genus ${\rm g}$ is big with respect to $\lambda$, so that the codimension of the unstable locus in the moduli stack of $G$-bundles is bigger than $\dim{\rm Gr}_G^\lambda=2(\lambda,\rho_G)$. 

This makes the operator $H_{x,\lambda}$ at least densely defined on $\mathcal H(s,\tau,\alpha)$, although it is not obvious that it lands in $\mathcal H(s,\tau,\alpha)$. 
We have conjectured in \cite{EFK2,EFK3} (see also \cite{BK2}) that in fact it does, and moreover it extends to a bounded operator which is norm-continuous in $x$ (see Conjecture \ref{spedeccon} below).

If so, then one can show that the operators $H_{x,\lambda}$ are normal (with $H_{x,\lambda}^\dagger=H_{x,\lambda^*}$), commute for different $x$ and $\lambda$ (as Hecke modifications at different points are done independently), and 
$$
H_{x,\lambda_1}H_{x,\lambda_2}=H_{x,\lambda_1+\lambda_2}.
$$  

\subsection{Hecke operators for effective divisors with coefficients in the coweight lattice}
It is possible that $X(F)=\emptyset$, then there are no $x\in X(F)$, so we cannot define the operators $H_{x,\lambda}$. But we can make a generalization.\footnote{In the case of $X=\Bbb P^1$ with ramification points and $G=PGL_2$, this generalization is discussed in \cite{EFK3}, Subsection 3.12.} Namely, let $M(X,G)$ be the set of  
$\Gamma_F$-invariant (for the action defined by $s$) functions $\mu: X(F_{\rm sep})\to \Lambda^+$ with finite support (i.e., effective divisors with coefficients in $\Lambda$). The set $M(X,G)$ is a commutative monoid graded by $\Lambda^+$: the degree $|\mu|$ is the sum of all values of $\mu$. Let $M(X,G)[\lambda]$ be the part of $M(X,G)$ of degree $\lambda\in \Lambda^+$. 

For $\mu\in M(X,G)$, let ${\mathcal Z}_{\mu,s,\tau}$ be the set of pairs $(P_1,P_2)$ of $F$-rational bundles which are identified outside ${\rm supp}\mu$ and are in relative position $\mu(x)$ at each point $x\in {\rm supp}\mu$. Then we can define the Hecke operator 
$$
(H_\mu\psi)(P)=\int_{{\mathcal Z}_{\mu,s,\tau}(P)}\psi(Q)\norm{dQ},
$$
where ${\mathcal Z}_{\mu,s,\tau}(P):=\lbrace{Q: (P,Q)\in {\mathcal Z}_{\mu,s,\tau}\rbrace}$. 
For example, if $x\in X(F)$ and $\mu=\lambda \delta_x$ then 
$H_\mu=H_{x,\lambda}$. 

The set $M(X,G)[\lambda]$ is the direct limit of subsets 
$$
M(X,G,E)[\lambda]=\lbrace \mu\in M(X,G)[\lambda]|{\rm supp}\mu\subset X(E)\rbrace
$$ 
over finite Galois extensions $F\subset E\subset F_{\rm sep}$. Note that $M(X,G,E)[\lambda]$ has a natural topology induced by the topology of $F$, so we have the topology of direct limit 
on $M(X,G)[\lambda]$ for each $\lambda$, hence on $M(X,G)$. We expect that the operator $H_\mu$ is norm-continuous with respect to $\mu$ in this topology. 

Moreover, let $x\in X_E$ for a finite extension $F\subset E\subset F_{\rm sep}$ and 
$\mu_x:=\sum_{\gamma\in \Gamma_F/\Gamma_E}\gamma(\lambda\delta_x)$ 
for some $\Gamma_E$-invariant coweight $\lambda\in \Lambda^+$. 
Then $H_{\mu_x}$ is an operator-valued $-(\lambda,\rho_G)$-density on the
$1$-dimensional analytic $E$-manifold $X_E$. 

These operators are expected to have similar properties to $H_{x,\lambda}$ for $x\in X(F)$ (bounded, normal with $H_\mu^\dagger=H_{\mu^*}$, $H_{\mu_1}H_{\mu_2}=H_{\mu_1+\mu_2})$, except that they always exist, as $X(F_{\rm sep})\ne \emptyset$. 

\begin{remark}\label{moregen} Let $C$ be a finite central subgroup of $G$ defined over $F$
(for example, we can take $C=Z$ if $G$ is semisimple). Then in genus zero the Hecke operators 
$H_{\mu}$ make sense more generally, when $\mu\in M(X,G/C)[\la]$ with $\lambda\in \Lambda_+$. 
\end{remark} 

\subsection{The spectral decomposition}\label{spedec}
As we have mentioned, the main problem of the analytic Langlands correspondence is to describe the joint spectral decomposition of the Hecke operators $H_\mu$. 

The following conjecture was made in \cite{EFK2} (see Conjecture 1.2 and Corollary 1.3).

\begin{conjecture}\label{spedeccon} If $G$ is semisimple and $\mu$ is nonzero on every simple factor then the operator $H_\mu$ is compact, and the intersection 
of the kernels of $H_\mu$ (for various $\mu$) is zero. 
Thus the spectrum of $\lbrace H_\mu\rbrace$ is discrete, i.e., we have an orthogonal decomposition 
$$
\mathcal H(s,\tau,\alpha)=\oplus_k \mathcal H(s,\tau,\alpha)_k,
$$
where $\mathcal H(s,\tau,\alpha)_k$ are finite dimensional joint eigenspaces. 
\end{conjecture} 

\subsection{The abelian case} 

Consider the abelian case, i.e., $G=\Bbb G_m^n$ is a torus. 
Then an inner class (or, equivalently, an $F$-form) of $G$ 
is defined by a homomorphism with finite image $s: \Gamma_F\to GL_n(\Bbb Z)$, and all bundles are stable with the automorphism group $G=Z$. Thus ${\rm Bun}_{G,s}^\circ(X,\tau)$ is the group ${\rm Pic}(X)^n(F_{\rm sep})^{\Gamma_F}$ where $\Gamma_F$ acts via $(s,\tau)$. We have a homomorphism 
$$
\varphi: {\rm Pic}(X)^n(F_{\rm sep})^{\Gamma_F}\to H^2(\Gamma_F,Z^s(F_{\rm sep})), 
$$
so we may consider the fibers 
$$
{\rm Pic}(X)^n(F_{\rm sep})^{\Gamma_F}_\alpha:=\varphi^{-1}(\alpha),
$$
and 
$$
\mathcal H(s,\tau,\alpha,\chi)=L^2({\rm Pic}(X)^n(F_{\rm sep})_\alpha^{\Gamma_F},\mathcal L_\chi).
$$ 
The Hecke operators act by translations. 
Thus the spectral decomposition in this case is just the decomposition 
of $L^2({\rm Pic}(X)^n(F_{\rm sep})_\alpha^{\Gamma_F},\mathcal L_\chi)$ into the characters of some finite index subgroup of ${\rm Pic}(X)^n(F_{\rm sep})_0^{\Gamma_F}={\rm Ker}\varphi$. 

The group ${\rm Pic}(X)^n(F_{\rm sep})^{\Gamma_F}$ is isomorphic (albeit non-canonically) to the product of the compact group  ${\rm Pic}_0(X)^n(F_{\rm sep})^{\Gamma_F}$ and a lattice. So the problem boils down to finding 
the character group of ${\rm Pic}_0(X)^n(F_{\rm sep})^{\Gamma_F}$. 
In the archimedean case, this reduces to classical Fourier analysis, so 
let us consider the non-archimedean case. 

\begin{example} Suppose $n=1,s=1$ and ${\rm char}(F)=0$, so that $F$ is an extension of $\Bbb Q_p$ of some degree $m$. Then we need to describe the character group of $J(F)$ where $J:={\rm Pic}_0(X)$. Let $\mathcal O_F$ be the ring of integers in $F$, $\mathfrak m\subset \mathcal O_F$ the maximal ideal, 
and $\Bbb F_q=\mathcal O_F/\mathfrak m$ the residue field of characteristic $p$. Suppose that $X$ has a good reduction $X_0$ over $\Bbb F_q$. In this case by Hensel's lemma, $J(F)=J(\mathcal O_F)$ and thus we have a short exact sequence 
$$
0\to J(\mathfrak m)\to J(F)\to J(\Bbb F_q)\to 0,
$$ 
which implies that we have a short exact sequence 
of character groups
$$
0\to J(\Bbb F_q)^\vee\to J(F)^\vee\to J(\mathfrak m)^\vee\to 0. 
$$
The group $J(\mathfrak m)$ is the formal Lie group 
attached to $J$ over $\mathcal O_F$, which is a finitely generated 
$\Bbb Z_p$-module of rank $m{\rm g}$. So it has a (non-canonical) decomposition
$$
J(\mathfrak m)\cong {\rm Tors}(J(\mathfrak m))\oplus \Bbb Z_p^{m{\rm g}},
$$
where ${\rm Tors}(J(\mathfrak m))$ is a finite abelian $p$-group whose exact structure depends on the arithmetic of $F$ and $X$. Thus
$$
J(\mathfrak m)^\vee\cong {\rm Tors}(J(\mathfrak m))^\vee\oplus 
(\Bbb Q_p/\Bbb Z_p)^{m{\rm g}}.
$$
On the other hand, the group $J(\Bbb F_q)^\vee$ can be described 
as usual through global class field theory of the function field 
$\Bbb F_q(X_0)$. 
\end{example} 

\begin{remark} 
In general, recall that $G$ is isogenous to a product of a torus and a semisimple group. Thus the problem of computing the spectrum of Hecke operators for a general $G$ effectively reduces to the case of semisimple $G$, once the abelian case has been understood.
\end{remark} 

\subsection{The archimedean case} \label{archi} 

Consider now the case when $F=\Bbb R$ or $F=\Bbb C$. In this case 
we have the {\bf quantum Hitchin system} of Beilinson and Drinfeld (\cite{BD}). This is a commutative algebra $\mathcal D$ consisting of global differential operators acting on a square root of the canonical bundle on ${\rm Bun}_G(X)$, which is naturally isomorphic to the algebra of polynomial functions on the affine space ${\rm Op}_{\g^\vee}$ of {\bf  Beilinson-Drinfeld opers} for the Langlands dual Lie algebra $\g^\vee$ of $\g={\rm Lie}G$. The algebra $\mathcal D$ acts by unbounded (i.e., densely defined) operators on $\mathcal H(s,\tau,\alpha)$, namely, on the subspace of smooth half-densities with compact support. Moreover, for $F=\Bbb C$ there is a similar action of the complex conjugate algebra $\overline{\mathcal D}$ which commutes with $\mathcal D$, so we get an action of $\mathcal A:=\mathcal D\otimes \overline{\mathcal D}$. 

Furthermore, these algebras commute in an algebraic sense with Hecke operators, i.e., the Schwartz kernels of the Hecke operators satisfy appropriate differential equations, see \cite{EFK2}.\footnote{More precisely, these differential equations are proved in \cite{EFK2} for $F=\Bbb C$, but the same argument applies to $F=\Bbb R$, except that we get just a single equation rather than two complex conjugate equations.}  It is moreover expected that they commute in a stronger, analytic sense. Namely, we conjecture that the operators from $\mathcal D$ (and in the complex case, $\overline {\mathcal D}$) have canonical normal extensions which strongly commute with each other and with Hecke operators, thereby having a common spectral decomposition.  

\subsection{The Schwartz space} 
Let $G$ be a semisimple group. 

1. Assume first that $F$ is a non-archimedean local field. 
Let $\phi$ be a locally constant $\mathcal L_\chi$-valued half-density on 
${\rm Bun}_{G,s,\alpha}^\circ(X,\tau)$ supported at sufficiently generic bundles. 
We conjecture that the space $\mathcal S(\phi)$ generated by $\phi$ under the action of Hecke operators is finite dimensional. Then $\mathcal S(\phi)$ a direct sum of eigenspaces of Hecke operators.\footnote{Since 
the action of Hecke operators on these finite dimensional spaces is given by matrices with algebraic entries, 
this would imply that all eigenvalues of Hecke operators are algebraic numbers, as conjectured by Kontsevich in \cite{Ko}, p.19.}
So we may consider the {\bf  Schwartz space} $\mathcal S(s,\tau,\alpha,\chi)=\sum_\phi \mathcal S(\phi)$. 
Equivalently, $\mathcal S(s,\tau,\alpha,\chi)$ is the space of smooth (=finite) vectors in $\mathcal H(s,\tau,\alpha,\chi)$, or the algebraic direct sum of joint eigenspaces of Hecke operators. 

2. Now assume that $F$ is archimedean. For $\lambda\in (\Lambda^+)^{\Gamma_F}$, 
let $\mathcal H_\lambda\subset \mathcal H$ be the (algebraic) sum of the images 
of $H_\mu$ with $|\mu|=\lambda$. Then we define the {\bf  Schwartz space} $\mathcal S(s,\tau,\alpha,\chi)\subset \mathcal H(s,\tau,\alpha,\chi)$ to be the intersection of $\mathcal H_\lambda$ over all $\lambda$; thus unlike the non-archimedean case, the Schwartz space is no longer countably dimensional, but it has a natural Fr\'echet topology defined by a collection of seminorms in which it is complete. It is clear that all eigenfunctions of Hecke operators belong to 
$\mathcal S(s,\tau,\alpha,\chi)$. 

Moreover, we expect that the Schwartz space $\mathcal S(s,\tau,\alpha)$ is the intersection of domains of the operators from $\mathcal D$. In other words, it is the space of square integrable half-densities whose Fourier coefficients $c_\Lambda$ with respect to an orthonormal eigenbasis $\psi_\Lambda$ of the algebra $\mathcal D$ decay rapidly (faster than any polynomial) as a function of the eigenvalues $\Lambda_i$ of generators $D_i\in \mathcal D$.

\begin{remark} Another, more geometric definition of the Schwartz space based on the geometry of the stack of $G$-bundles on $X$, which is conjecturally equivalent to the above, is proposed in the non-archimedean case in \cite{BK2}, see also \cite{BKP1}, \cite{BKP2}. A similar definition is expected to exist in the archimedean case. 
\end{remark} 

\subsection{The ramified case}\label{ram1}

The above picture has a generalization to the ramified case, i.e., the case of bundles with level structure along an $F$-rational effective divisor on $X$. This is, in particular, required to extend the analytic Langlands correspondence to curves of genus 0 and 1 (as such curves admit stable bundles only in the ramified setting). Before considering  the ramified case we remind the following general construction.

Let $F$ be a local field, $Y$ be an analytic $F$-manifold, and $\bold G$ an analytic $F$-group acting properly and freely on $Y$. Let $\pi$ be a unitary representation of $\bold G$. Let $V_\pi:=(Y\times \pi)/\bold G$ be the associated Hilbert space bundle over the $F$-manifold $Y/\bold G$ (the action of $\bold G$ is given by $g(y,v)=(gy,gv)$).
So we have the projection $V_\pi\to Y/\bold G$ with fiber $\pi$. 
Let $\psi$ be a measurable half-density on $Y/\bold G$ with values in $V_\pi$. 
Then $|\psi|^2$ is a density on $Y/\bold G$. Let $\mathcal H(Y,\bold G,\pi)$ 
be the Hilbert space of all $\psi$ with $\int_{Y/\bold G}|\psi|^2<\infty$. 

More generally, if $\mathcal L$ is a $\bold G$-equivariant hermitian line bundle on $Y$ (i.e., a hermitian line bundle on $Y/\bold G$), the we can 
define the Hilbert space $\mathcal H(Y,\bold G,\pi,\mathcal L)$ 
of $\mathcal L$-valued sections $\psi$ with $\int_{Y/\bold G}|\psi|^2<\infty$. 

If $\bold P\subset \bold G$ is a closed subgroup and $\pi$ a unitary representation of $\bold P$ then we have a canonical isomorphism 
$$
\mathcal H(Y,\bold P,\pi,\mathcal L)\cong \mathcal H(Y,\bold G,{\rm ind}_{\bold P}^{\bold G}\pi,\mathcal L),
$$ 
where ind denotes unitary induction. Moreover, if $\pi$ is a representation  of a closed subgroup $\bold P\subset \bold N\subset N_{\bold G}(\bold P)$ then 
this space carries a unitary action of $\bold N/\bold P$. 

We are now ready to discuss the ramified case. 
The most general setting is as follows. 

Let $t_1,...,t_N$ be distinct points in $X(F_{\rm sep})$ and $\zeta_i$ be local parameters on $X$ near $t_i$. Let $\ell\ge 1$.
Denote by $Bun^{\ell}_G(X,t_1,...,t_N)$  the stack of principal $G$-bundles on $X$ trivialized on the $\ell$-th nilpotent neighborhood $D_\ell(t_i)$ of each $t_i$ (i.e., modulo $\zeta_i^\ell$). We have a torsor $ Bun^{\ell}_G(X,t_1,...,t_N)\to Bun_G(X)$  for the algebraic group 
$\prod_{i=1}^N {\rm Map}(D_\ell(t_i),G)\cong \prod_{i=1}^N G_\ell$, where for a commutative ring $R$, $G_\ell(R):=G(R[\zeta]/\zeta^\ell)$ (the identification is made using the local parameters $\zeta_i$). 

Now assume that $t_i$ are permuted by $\Gamma_F$ acting via $\tau$.
Let $\lbrace t_i,i\in S\rbrace, S\subset [1,n]$ be a set of representatives of orbits of this action. Then given $s: \Gamma_F\to {\rm Aut}\Delta_G$, we have an action of $\Gamma_F$ on $ Bun^{\ell}_G(X,t_1,...,t_N)(F_{\rm sep})$. Denote by ${\rm Bun}^{\ell\circ}_{G,s}(X,\tau,t_1,...,t_N)$ the subset of regularly stable bundles
fixed by this action. Let $F\subset E_i\subset F_{\rm sep}$ be the field of definition of $t_i$, i.e., $E_i=F_{\rm sep}^{\Gamma_{t_i}}$. Recall that in Subsection \ref{formpt}, to every $P\in  {\rm Bun}^{\ell\circ}_{G,s}(X,\tau,t_1,...,t_N)$ and each $i$ we have attached an $E_i$-form $G^{\sigma_i}$ of $G$ in the inner class $s$. This form, in turn, defines an $E_i$-form $G_\ell^{\sigma_i}$ of $G_\ell$. Let 
${\rm Bun}^{\ell\circ}_{G,s}(X,\tau,t_1,...,t_N,\sigma_1,...,\sigma_N)$
be the subset of bundles defining the fixed forms $\sigma_1,...,\sigma_N$
(clearly, if $\gamma t_i=t_j$ then $\gamma E_i=E_j$ and $\gamma \sigma_i=\sigma_j$). This is an analytic $F$-manifold with an action of $\bold G:=\prod_{i\in S} G_\ell^{\sigma_i}(E_i)$, and, as in the unramified case, it is the disjoint union of open submanifolds 
${\rm Bun}^{\ell\circ}_{G,s,\alpha}(X,\tau,t_1,...,t_N,\sigma_1,...,\sigma_N)$.

Fix a closed subgroup $\bold P\subset \bold G$ acting freely and properly on ${\rm Bun}^{\ell\circ}_{G,s}(X,\tau,t_1,...,t_N,\sigma_1,...,\sigma_N)$ and let $\pi$ be an irreducible unitary representation of $\bold P$. Analogously to the unramified case, define the Hilbert space 
$$
\mathcal H=\mathcal H(s,\tau,\alpha,\chi,t_1,...,t_N,\pi):=\mathcal H({\rm Bun}^{\ell\circ}_{G,s,\alpha}(X,\tau,t_1,...,t_N,\sigma_1,...,\sigma_N),\bold P,\pi,\mathcal L_\chi),
$$ 
where $\mathcal L_\chi$ is the line bundle defined in Subsection \ref{heop}
(it is clear that this space does not change if we replace the set of stable points by its $\bold P$-invariant dense open subset). 

The space $\mathcal H$ (conjecturally) carries an action of Hecke operators $H_\mu$ where $\mu(t_i)=0$ for all $i$, and the main question of the (ramified) analytic Langlands correspondence is to describe the joint spectrum of $H_\mu$ (which in general won't be discrete).\footnote{The definition of Hecke operators is the same as in \cite{EFK2} in the unramified case, using the Beilinson-Drinfeld isomorphism $a$ from \cite{EFK2}, Theorem 1.1. Ramification data does not alter this definition. However, 
the definition is still conjectural because of the analytic issues (landing in $L^2$, boundedness)
which are present already in the unramified case and are discussed in \cite{EFK2}.}
 Note that if $\pi$ is a unitary representation of a closed subgroup $\bold P\subset \bold N\subset N_{\bold G}(\bold P)$ (for example, $\pi=\Bbb C$, $\bold N=N_{\bold G}(\bold P)$), then this space carries a commuting action of $\bold N/\bold P$. 

Moreover, in the archimedean case the space $\mathcal H$ carries a (densely defined) action of the quantum Hitchin system (extended from the unramified case to produce a quantum integrable system), which conjecturally commutes with the Hecke operators and thus has compatible spectral decomposition.\footnote{Recall that the moduli stack $Bun^\ell_G(X,t_1,...,t_N)$  can be represented as a double quotient 
$$
Bun^\ell_G(X,t_1,...,t_N)=
{\rm Ker}
(G[[t]]\to G_\ell)^N\backslash G((t))^N/G(\mathcal O(X)),
$$  
and the quantum Hitchin system is obtained by reduction of two-sided invariant differential operators on the Kac-Moody central extension of $G((t))^N$ at the critical level (the Feigin-Frenkel center of the Kac-Moody algebra) down to the double quotient. Because of the critical level, we get commuting differential operators acting on half-forms on $Bun^\ell_G(X,t_1,...,t_N)$, and 
Beilinson and Drinfeld showed that they form a quantum integrable system (i.e., the number of algebraically independent commuting operators is the maximal possible, namely $\dim Bun^\ell_G(X,t_1,...,t_N)$).} 

Consider first the case ${\rm g}\ge 2$. In this case we may define 
stable bundles as above, to be the bundles stable in the usual sense with 
automorphism group $Z$. Then we may take $\bold P$ 
to be the entire group $\bold G$, so that $\pi=\bigotimes_{i\in S}\pi_i$, 
where $\pi_i$ are irreducible unitary representations of 
$G_\ell^{\sigma_i}(E_i)$. For instance, 
for $\ell=1$ (tame ramification) $\pi_i$ are irreducible unitary representations of the groups of $E_i$-points of the reductive groups $G^{\sigma_i}$. 

In the (generally simpler) case ${\rm g}\le 1$ the situation is a bit more tricky 
since $\bold G$ does not act freely any more, so one has to take proper subgroups $\bold P\subset \bold G$ which act freely. For simplicity consider the tamely ramified case $\ell=1$. For $i\in S$ let $\bold P_i$ be cocompact closed subgroups of $\bold G_i:=G^{\sigma_i}(E_i)$ and $\pi_i$ be finite dimensional unitary representations of $\bold P_i$. Let $\bold P:=\prod_{i\in S}\bold P_i$
and $\pi:=\bigotimes_{i\in S}\pi_i$. If $\bold P$ acts freely then 
we may define the space $\mathcal H$ as above and we expect the spectrum of Hecke operators on this space to be discrete. 

\begin{example}\label{exaa1} 1. $\bold P_i=P_i(E_i)$ for parabolic subgroups $P_i\subset G^{\sigma_i}$ defined over $E_i$; this corresponds to doing harmonic analysis on the moduli space of bundles with parabolic structures. 
For instance, if $P_i=B_i$ are Borel subgroups then we must have $\sigma_i=s$ for all $i$. 

Another extreme is $P_i=G^{\sigma_i}$ and $\pi_i=\Bbb C$. Then we recover the unramified case discussed above.  

2. $\bold G_i$ are compact. In this case we may take $\bold P_i=1$, so $\pi=\Bbb C$, and the space $\mathcal H$ carries a commuting action 
of the compact group $\bold G$, so we have 
$$
\mathcal H=\oplus_{\rho\in \widehat{\bold G}}\mathcal H_\rho\otimes \rho^*
$$
where $\rho=\bigotimes_{i\in S}\rho_i$ for some irreducibles $\rho_i$ of $\bold G_i$, and $\mathcal H_\rho:=(\mathcal H\otimes\rho)^{\bold G}$. In this case the Hecke operators act on $\mathcal H_\rho$ 
for each $\rho$.

Of course, these two examples can also be combined, with (1) occurring at  some of the points $t_i$ and (2) at other ones.  
\end{example}

\subsection{The ramified genus $0$ case}\label{ram2}
In the case ${\rm g}=0$, i.e., $X=\Bbb P^1$, the space $\mathcal H$ can be described entirely in terms of Lie theory. Indeed, suppose that the group $G$ is semisimple and simply connected and $\ell=1$ (the tamely ramified case). In this case, it suffices to consider only trivial $G$-bundles (with trivializations at $t_i$). Thus 
the moduli space is non-empty only if 
$\sigma_i=\sigma|_{\Gamma_{E_i}}$ for all $i$ up to conjugation. 
In this case, we have a natural inclusion $G^\sigma(F)\hookrightarrow G^{\sigma_i}(E_i)$, and the moduli space looks like 
$(\prod_{i\in S}G^{\sigma_i}(E_i))/G^\sigma(F)$. 
Thus 
$$
\mathcal H=\mathcal H\left((\prod_{i\in S}G^{\sigma_i}(E_i))/G^\sigma(F),\prod_{i\in S}\bold P_i,\bigotimes_{i\in S}\pi_i\right).
$$
 This picture extends in an obvious way to the case of wild ramification.\footnote{The parameters $\alpha$ and $\chi$ do not appear here, since 
the automorphism group of the trivial $G$-bundle trivialized on a  
a non-empty subset of $X$ is trivial.}   

For instance, in Example \ref{exaa1}(2), we have 
$$
\mathcal H_\rho=\left(\bigotimes_{i\in S} \rho_i\right)^{G^\sigma(F)}, 
$$
a finite dimensional space. In this case, all the analytic issues disappear and the spectral problem is automatically well defined.  

\begin{example}\label{bethan} Let $\ell=1$, ${\rm g}=0$, $N\ge 2$, $t_i\in X(F)$, $\sigma_i=\sigma$ for all $1\le i\le N$. Let $\bold P_i=1$, $\pi_i=\Bbb C$. Then 
$$
{\mathcal H}=L^2(G^\sigma(F)^N/G^{\sigma}(F)_{\rm diagonal}).
$$
So the spectral decomposition of $\mathcal H$ is an extension of the Plancherel formula for $G^\sigma(F)$. Namely, we have 
\begin{equation}\label{plan}
{\mathcal H}=\int^{\oplus}_{\widehat{G^\sigma(F)}^N}{\rm Mult}_{G^\sigma(F)}(V_N^*,V_1\otimes...\otimes V_{N-1})\otimes (V_1^*\otimes...\otimes V_N^*) d\mu(V_1)...d\mu(V_N),
\end{equation} 
where $d\mu(V)$ is the Plancherel measure of $G^\sigma(F)$ and 
${\rm Mult}_{G^\sigma(F)}(L,M)$ is the multiplicity space\footnote{The multiplicity space 
${\rm Mult}_{G^\sigma(F)}(L,M)$ is a generalization  of ${\rm Hom}_{G^\sigma(F)}(L,M)$ in presence of continuous spectrum. Namely, it coincides with the latter when 
$M$ is a Hilbert direct sum (finite or infinite) of irreducible unitary representations.} of an irreducible tempered representation $L$ of $G^\sigma(F)$ in a tempered representation 
$M$, defined using Plancherel theory (see \cite{D}, Subsection 18.8).\footnote{Note that 
the representation $V_1\otimes...\otimes V_N$ is tempered, so does not contain the trivial representation in its spectrum. Hence in the $L^2$ context we cannot talk about ${\rm Mult}_{G^\sigma(F)}(\Bbb C,V_1\otimes...\otimes V_N)$. This is why we break the symmetry and dualize one of the factors. However, it can be shown that the result does not depend on the choice of this factor.} 
For example, for $N=2$ by Schur's lemma the multiplicity space vanishes unless 
$V_1\cong V_2^*$, in which case it is 1-dimensional, so formula \eqref{plan} takes the form 
$$
{\mathcal H}=\int^{\oplus}_{\widehat{G^\sigma(F)}}(V^*\otimes V) d\mu(V),
$$
which is the usual Plancherel formula for $L^2(G^\sigma(F))$.

Each multiplicity space 
\begin{equation} 
{\mathcal H}_{V_1,...,V_N}:={\rm Mult}_{G^\sigma(F)}(V_N^*,V_1\otimes...\otimes V_{N-1})
\end{equation} 
carries an action of commuting Hecke operators. For example, 
as follows from \cite{EFK3}, Section 3, if $G=PGL_2$, $G^\sigma$ is the split form, and $x\in \Bbb P^1(F)$, then the Hecke operator has the form 
\begin{equation}\label{heckoper}
H_x\psi=\int_F \left(\begin{pmatrix}0& x-t_1\\ 1& -s\end{pmatrix}\otimes...\otimes \begin{pmatrix}0& x-t_{N-1}\\ 1& -s\end{pmatrix}\right)^{-1}\psi\norm{ds}
\end{equation}
where the tensor product acts in $V_1\otimes...\otimes V_{N-1}$.
\end{example}

\begin{remark}\label{projre} One may take $V_i$ to be {\bf  projective} representations of $G^\sigma$ (i.e., replace $G^\sigma$ with a central extension), as long as the product of Schur multipliers attached to them equals $1$. For instance,
in Example \ref{bethan} we may replace $SL_2(\Bbb R)$ 
by its universal cover $\widetilde{SL_2(\Bbb R)}$. For an example of this 
see Remark \ref{projre1} below. 
\end{remark} 

\subsection{The ramified genus $0$ case for $PGL_2$} \label{pgl2}

\subsubsection{Hypergeometric integrals}

Let $F$ be a local field and $a,b\in i\Bbb R$. For $\alpha\in \Bbb C$, ${\rm Re}\alpha>0$ let 
$$
K(\alpha):=\int_{v\in F: \norm{v}\le 1}{\norm{v}^{\alpha-1}}\norm{dv}.
$$
Under suitable normalization of the Lebesgue measure on $F$
we have $K(\alpha)=\frac{1}{\alpha}$ if $F$ is archimedean and 
$K(\alpha)=\frac{\log q}{q^\alpha-1}$ if not, where $q$ is the order of the residue field of $F$. 

For ${\rm Re}\alpha=0$ this integral does not converge (even conditionally). But we can regularize it if $\alpha\ne 0$ by using an $\epsilon$-deformation: 
\begin{equation}\label{epsdef}
\int_{v\in F: \norm{v}\le 1}{\norm{v}^{\alpha-1}}\norm{dv}:=\lim_{\epsilon\to 0+}\int_{v\in F: \norm{v}\le 1}{\norm{v}^{\alpha+\epsilon-1}}\norm{dv}.
\end{equation} 

Denote by $\Gamma^F$ the $\Gamma$-{\bf function} of $F$, which is the meromorphic function on $\Bbb C$ defined by the condition that 
$$
\mathcal F\norm{u}^{a-1}=\Gamma^F(a)\norm{u}^{-a},
$$ 
where $\mathcal F$ is the Fourier transform on $F$ (so $\Gamma^F(a)\Gamma^F(1-a)=1$).\footnote{Here for simplicity we consider the $\Gamma$-function for unramified characters, but all the formulas extend straightforwardly to general multiplicative characters of $F$.}

Recall the {\bf beta integral} 
$$
\int_F \norm{s}^{\alpha-1}\norm{s-1}^{\beta-1}\norm{ds}=B^F(\alpha,\beta):=\frac{\Gamma^F(\alpha)\Gamma^F(\beta)}{\Gamma^F(\alpha+\beta)}.
$$
Thus 
$$
B^F(\alpha,\beta)=B^F(\beta,\alpha)=B^F(\alpha,1-\alpha-\beta).
$$

Now consider the integral 
$$
I(x):=\int_{s\in F: \norm{s}\ge 1}\norm{s}^{\alpha-1}\norm{s-x}^{\beta-1}\norm{ds}
$$
where ${\rm Re}\alpha=0$, $0<{\rm Re}\beta<1$. 
Setting $s=xv^{-1}$, we get 
$$
I(x)=\norm{x}^{\alpha+\beta-1}\int_{\norm{v}\le \norm{x}}\norm{v}^{-\alpha-\beta}\norm{v-1}^{\beta-1}\norm{dv}.
$$
Thus $I(x)$ depends only on $\norm{x}$.

The following lemma is a generalization of \cite{EFK3}, Lemma 8.1. 

\begin{lemma}\label{8.1anal} If $\norm{\cdot}^{\alpha}\ne 1$ then we have 
$$
I(x)=K(-\alpha)\norm{x}^{\beta-1}+B^F(\alpha,\beta)\norm{x}^{\alpha+\beta-1}+o(\norm{x}^{{\rm Re}\beta-1}),\ x\to \infty.
$$ 
\end{lemma} 

\begin{proof} 
 We write 
$$
\norm{x}^{1-\alpha-\beta}I(x)=\int_{\norm{v}\le R}\norm{v}^{-\alpha-\beta}\norm{v-1}^{\beta-1}\norm{dv}
+\int_{R<\norm{v}\le \norm{x}}\norm{v}^{-\alpha-\beta}\norm{v-1}^{\beta-1}\norm{dv}
=
$$
$$
\int_{\norm{v}\le R}\norm{v}^{-\alpha-\beta}\norm{v-1}^{\beta-1}\norm{dv}
+
\int_{R<\norm{v}\le \norm{x}}\left(\norm{1-v^{-1}}^{\beta-1}-1\right)
\norm{v}^{-\alpha-1}\norm{dv}+
$$
$$
\int_{R<\norm{v}\le \norm{x}}\norm{v}^{-\alpha-1}\norm{dv}.
$$
The first two integrals have a limit as $x\to \infty$, namely
$$
C_R(\alpha,\beta):=\int_{\norm{v}\le R}\norm{v}^{-\alpha-\beta}\norm{v-1}^{\beta-1}\norm{dv}
+
\int_{R<\norm{v}}\left(\norm{1-v^{-1}}^{\beta-1}-1\right)
\norm{v}^{-\alpha-1}\norm{dv},
$$
while the last integral equals 
$$
K(-\alpha)(\norm{x}^{-\alpha}-R^{-\alpha}).
$$
So we get 
$$
I(x)=K(-\alpha)\norm{x}^{\beta-1}+(C_R(\alpha,\beta)-K(-\alpha)R^{-\alpha})\norm{x}^{\alpha+\beta-1}+o(\norm{x}^{{\rm Re}\beta-1}),\ x\to \infty. 
$$
Thus $C_R(\alpha,\beta)-K(-\alpha)R^{-\alpha}$ is independent of $R$. 
In fact, it is easy to show that for any $R$, 
$$
C_R(\alpha,\beta)-K(-\alpha)R^{-\alpha}=B^F(\alpha,\beta). 
$$
This implies the result.  
\end{proof} 

The following lemma is a generalization of \cite{EFK3}, Lemma 8.2.

\begin{lemma}\label{8.2anal} Let $\alpha,\beta$ be as above and $\varphi$ be a locally integrable function on $\Bbb P^1(F)$ 
which is smooth near $\infty$ and has power decay at $0$. Let 
$$
I(x,\varphi):=\int_{F}\varphi(s)\norm{s}^{\alpha-1}\norm{s-x}^{\beta-1}\norm{ds},\ \norm{x}\gg 1.
$$
Then if $\norm{\cdot}^\alpha\ne 1$ then 
$$
I(x,\varphi)=\norm{x}^{\beta-1}\int_{F}\varphi(s)\norm{s}^{\alpha-1}\norm{ds}+
\varphi(\infty)B^F(\alpha,\beta)\norm{x}^{\alpha+\beta-1}+o(\norm{x}^{{\rm Re}\beta-1}),\ x\to \infty,
$$
where the integral is understood in the sense of $\epsilon$-deformation. 
\end{lemma} 

\begin{proof} If $\varphi(\infty)=0$, this follows by taking the limit directly, and 
if $\varphi=\bold 1_{\norm{s}\ge 1}$ is the indicator function then this is Lemma \ref{8.1anal}. 
So the result follows from the decomposition 
$$
\varphi(s)=(\varphi(s)-\varphi(\infty)\bold 1_{\norm{s}\ge 1})+\varphi(\infty)\bold 1_{\norm{s}\ge 1}.
$$
\end{proof} 

Now for $x\in F$, $x\ne 0,1$, $0<{\rm Re}\gamma<1$, and consider the hypergeometric integral
$$
\Phi^F(\alpha,\beta,\gamma; x):=\int_{F}\norm{s}^{\alpha-\gamma}\norm{s-1}^{\gamma-1}\norm{s-x}^{\beta-1}\norm{ds}.
$$
Taking in Lemma \ref{8.2anal} $\varphi(s)=\norm{1-s^{-1}}^{\gamma-1}$, where $0<{\rm Re}\gamma<1$, 
we get 

\begin{proposition}\label{hgin}
$$
\Phi^F(\alpha,\beta,\gamma; x)=B^F(-\alpha,\gamma)\norm{x}^{\beta-1}+B^F(\alpha,\beta)\norm{x}^{\alpha+\beta-1}+o(\norm{x}^{{\rm Re}\beta-1}),\ x\to \infty.
$$
\end{proposition} 

\subsubsection{Principal series representations} 
Let $G$ be a split connected reductive algebraic group over a local field $F$, $B\subset G$ a Borel subgroup defined over $F$, and let $\bold G:=G(F),\bold B:=B(F)$. Then any integral weight $\mu$ of $G$ defines a character $\bold B\to F^\times$, denoted $b\mapsto b^\mu$. Let $\rho:=\rho_G$ be the sum of fundamental weights of $G$. Then densities on $\bold G/\bold B$ can be realized as functions $h: \bold G\to \Bbb C$ such that $h(gb)=\norm{b^{-2\rho}}h(g)$, 
$g\in \bold G$, $b\in \bold B$. 

Now let $\chi$ be a unitary character $\bold B\to \Bbb C^\times$ (or, equivalently, $\bold T\to \Bbb C^\times$, where $\bold T:=\bold B/[\bold B,\bold B]$). Then
we can define the {\bf principal series representation} $M(\chi)$ of $\bold G$, which 
is the space of functions $f: \bold G\to \Bbb C$ 
such that 
$$
f(gb)=\chi(b)\norm{b^{-\rho}}f(g),\ g\in \bold G, b\in \bold B,
$$ 
and 
$|f|^2$ is an $L^1$-density on $\bold G/\bold B$, with the action of 
$\bold G$ by left multiplication ($(gf)(x):=f(g^{-1}x)$). Then 
$M(\chi)$ is a unitary representation of $\bold G$ with inner product 
$$
(f_1,f_2)=\int_{\bold G/\bold B}f_1\overline f_2.
$$

It is well known (see e.g. \cite{ABV} for the archimedean cases and \cite{Ca} for non-archimedian ones) that the representations $M(\chi)$ are irreducible 
and $M(\chi_1)\cong M(\chi_2)$ iff 
$\chi_1=w\chi_2$ for some element $w$ of the Weyl group $W$
of $G$. These isomorphisms are nontrivial and are called 
{\bf intertwining operators}; we discuss them below for $G=PGL_2$.

\subsubsection{Principal series representations for $PGL_2$ and intertwining operators}  
Let us now consider the case $G=PGL_2$. Then $\bold T=F^\times$, so for $c\in i\Bbb R$ 
we may take $\chi_c(t)=\norm{t}^{c}, t\in \bold T$, and define the {\bf unramified (or spherical) principal series 
representation} $M(\chi_c)$. We will set $\lambda:=-1+c$ and denote 
$M(\chi_c)$ by $M_\lambda$. 
Thus $M_\lambda=L^2(\Bbb P^1(F),\norm{K}^{-\frac{\lambda}{2}})$, where $K$ is the canonical bundle
(as $c\in i\Bbb R$, this is naturally a Hilbert space with inner product $(f_1,f_2)=\int_{\Bbb P^1(F)} f_1\overline f_2$). 

It is well known that we have an isomorphism of unitary representations $\iota: M_\lambda\cong M_{-\lambda-2}$ such that $\iota^2={\rm Id}$, namely the {\bf  intertwining operator}
$$
\iota: L^2(\Bbb P^1(F),\norm{K}^{-\frac{\lambda}{2}})\to L^2(\Bbb P^1(F),\norm{K}^{\frac{\lambda+2}{2}})
$$
given by 
$$
\iota=\mathcal F\circ \norm{\cdot}^{\lambda+1}\circ \mathcal F^{-1},
$$
where $\norm{\cdot}^{\lambda+1}$ denotes the operator of 
multiplication by the function $\norm{\cdot}^{\lambda+1}$. In other words, for smooth $f$ vanishing near $\infty$ we have 
\begin{equation}\label{iotaform} 
(\iota f)(z)=\lim_{\epsilon\to 0+}\frac{1}{\Gamma^F(-\lambda-1+\epsilon)}\int_F f(w)\norm{z-w}^{-\lambda-2+\epsilon}\norm{dzdw}^{\frac{\lambda+2}{2}}.
\end{equation}
Observe that this integral is absolutely convergent near the diagonal $z=w$ and the right hand side
tends to the identity operator when $\lambda\to -1$.\footnote{\label{foo} Here we use the $\epsilon$-deformation as in \eqref{epsdef} to make sense of the divergent integral in \eqref{iotaform}.} Thus for $\lambda=-1$ we have $\iota=1$.
 
The existence of $\iota$ implies that $M_\lambda$ depends only on the Casimir eigenvalue $\frac{1}{2}(\lambda+1)^2=\frac{c^2}{2}$. 

Let $\lambda_j\in -1+i\Bbb R$, $j=0,...,m+1$, 
$V_j:=M_{\lambda_j}$, $j\in [0,m+1]$ and $\bla:=(\la_0,...,\la_m,\la_{m+1})$. 
Consider the Hilbert space 
\begin{equation}\label{hilspace}
\mathcal H(\bla):={\mathcal H}_{V_0,...,V_{m+1}}={\rm Mult}_{PGL_2(F)}(M_{\lambda_{m+1}}^*,M_{\lambda_0}\otimes...\otimes M_{\lambda_m}).
\end{equation}
Similarly to \cite{EFK3}, Section 3.3, the space $\mathcal H(\bla)$ may be realized as 
the space of functions $\psi(y_0,...,y_m)$ on $F^{m+1}$ invariant 
under simultaneous translations $y_j\mapsto y_j+C$, homogeneous of degree 
$\frac{1}{2}(\sum_{j=0}^{m}\lambda_j-\lambda_{m+1})$, and specializing at $y_0=0,y_m=1$ to square integrable functions on $F^{m-1}$ (see e.g. \cite{P,Re} for $F=\Bbb R$, \cite{Na,Wi} for $F=\Bbb C$, \cite{Ma} for a general local field). The inner product on $\mathcal H(\bla)$ is given by 
$$
(\psi,\eta)=\int_{F^{m-1}}\psi(0,y_1,...,y_{m-1},1)\overline{\eta(0,y_1,...,y_{m-1},1)}\norm{dy_1...dy_{m-1}}.
$$

Note that if $\varepsilon=(\varepsilon_0,...,\varepsilon_m,\varepsilon_{m+1})\in (\Bbb Z/2)^{m+2}$ is a collection of signs, and $(\varepsilon \circ \bla)_j:=\varepsilon_j (\lambda_j+1)-1$, 
then we have isomorphisms\footnote{Here and below we will consider various operators 
with source $\mathcal H(\bla)$. While these operators by definition depend on $\bla$, we will drop it from the notation when no confusion is possible.}   $R_\varepsilon=R_\varepsilon^\bla: 
\mathcal H(\bla)\cong \mathcal H(\varepsilon\circ  \bla)$ given by composing 
the isomorophisms $\iota$ in all variables $y_j$ for which $\varepsilon_j=-1$, and 
$R_{\varepsilon\varepsilon'}=R_{\varepsilon}R_{\varepsilon'}$. Thus we can think of the collection of spaces $\mathcal H(\bla)$ attached to an orbit 
$\mathcal O$ of $(\Bbb Z/2)^{m+2}$ as a single Hilbert space $\mathcal H(\mathcal O)$ 
with multiple realizations $\mathcal H(\bla)$, $\bla\in \mathcal O$ connected by a consistent family of isomorphisms. 

Furthermore, we can write an explicit formula for $R_\varepsilon$. 
For example, let us write an explicit formula for $R:=R_{1,...,1,-1} : \mathcal H(\bla)\to \mathcal H(\bla')$, where $\bla':=(\la_0,...,\la_m,-\la_{m+1}-2)$. In fact, we can write a more general formula for the operator 
$$
R: {\rm Mult}_{PGL_2(F)}(M_{\lambda}^*,V)\to {\rm Mult}_{PGL_2(F)}(M_{-\lambda-2}^*,V)
$$ 
induced by $\iota$ for any tempered representation $V$ of $PGL_2(F)$ (note that $R^2=1$, so $R^\dagger=R$). 
To this end, assume first that $V$ is a direct sum of irreducible representations and realize elements of ${\rm Mult}_{PGL_2(F)}(M_{\lambda}^*,V)$ as  
$V$-valued (generalized) $-\frac{\lambda}{2}$-densities $v=v(y)\norm{dy}^{-\frac{\la}{2}}$ 
on $\Bbb P^1(F)$ equivariant under $PGL_2(F)$. 
Then formula \eqref{iotaform} 
implies that 
$$
Rv=\lim_{\epsilon\to 0+}\frac{1}{\Gamma^F(-\lambda-1+\epsilon)}\int_F v(s)\norm{y-s}^{-\lambda-2+\epsilon}\norm{ds}\norm{dy}^{\frac{\lambda+2}{2}},
$$
which we will write for brevity as 
$$
Rv=\frac{1}{\Gamma^F(-\lambda-1)}\int_F v(s)\norm{y-s}^{-\lambda-2}\norm{ds}\norm{dy}^{\frac{\lambda+2}{2}}.
$$
Let $\widehat v:=v(u^{-1})\norm{u}^\la \norm{du}^{-\frac{\la}{2}}$ 
be the image of $v$ under the change of coordinates $u=y^{-1}$, so 
$\widehat v(u):=v(u^{-1})\norm{u}^\la$.  
Then  
$$
\widehat{Rv}(u)=\frac{1}{\Gamma^F(-\lambda-1)}\int_F v(s)\norm{1-su}^{-\lambda-2}\norm{ds}.
$$
Specializing this at $u=0$, we have 
$$
\widehat{Rv}(0)=\frac{1}{\Gamma^F(-\lambda-1)}\int_F v(s)\norm{ds}=\frac{1}{\Gamma^F(-\lambda-1)}\int_F \widehat v(s^{-1})\norm{s}^{\la}\norm{ds}.
$$
But $s^{-1}$ is obtained from $0$ by the element $g(s)\in PGL_2(F)$ sending $u$ to $u+s^{-1}$, i.e. 
$y$ to $(y^{-1}+s^{-1})^{-1}=\frac{y}{s^{-1}y+1}$. Thus $g(s)=\begin{pmatrix} 1& 0\\ s^{-1} & 1\end{pmatrix}$. So we get 
$$
\widehat{Rv}(0)=\frac{1}{\Gamma^F(-\lambda-1)}\int_F g(s)^{-1}\widehat v(0)\norm{s}^{\la}\norm{ds}.
$$
This formula continues to hold if $V$ is not necessarily a direct sum of irreducible representations but rather a direct integral. For example, 
in our situation $V=V_0\otimes...\otimes V_m$ 
is the space of translation invariant functions in $y_0,...,y_m$ 
which are homogeneous of degree $\frac{1}{2}(\sum_{i=0}^{m}\la_i-\la_{m+1})$, 
so we get 
$$
R\psi= \frac{1}{\Gamma^F(-\lambda_{m+1}-1)}\int_F g(s)^{-1}\psi\norm{s}^{\la_{m+1}}\norm{ds},
$$
i.e., 
$$
R\psi(y_0,...,y_m)=\tfrac{1}{\Gamma^F(-\lambda_{m+1}-1)}\int_F \psi\left(\tfrac{y_0s}{s-y_0},...,\tfrac{y_0s}{s-y_m}\right)\norm{s}^{\lambda_{m+1}-\sum_{i=0}^m \la_i}\prod_{i=0}^m\norm{s-y_i}^{\la_i}\norm{ds}=
$$
$$
\frac{1}{\Gamma^F(-\lambda_{m+1}-1)}\int_F \psi\left(\frac{s^2}{s-y_0},...,\frac{s^2}{s-y_m}\right)\norm{s}^{\lambda_{m+1}-\sum_{i=0}^m \la_i}\prod_{i=0}^m\norm{s-y_i}^{\la_i}\norm{ds}.
$$
We thus obtain the following lemma. 

\begin{lemma}\label{Rpsi} We have\footnote{Note that this integral is not convergent near $s=\infty$ and should be understood in the sense of $\epsilon$-deformation in $\lambda_{m+1}$, as explained in footnote \ref{foo}.}
$$
(R\psi)(y_0,...,y_m)=
\frac{1}{\Gamma^F(-\lambda_{m+1}-1)}\int_F \psi\left(\frac{1}{s-y_0},...,\frac{1}{s-y_m}\right)\prod_{i=0}^m\norm{s-y_i}^{\la_i}\norm{ds}
$$
\end{lemma} 

\begin{remark} Setting in Lemma \ref{Rpsi} $t_0=y_0=0$ and $\varphi(y_1,...,y_m):=\psi(0,y_1,...,y_m)$,  
we obtain 
\scriptsize
\begin{equation}\label{killsym1}
%\begin{split}
(R\varphi)(y_1,...,y_m)
=\frac{1}{\Gamma^F(-\lambda_{m+1}-1)}\int_F \varphi\left(\frac{y_1}{1-y_1s^{-1}},...,\frac{y_m}{1-y_ms^{-1}}\right)\prod_{j=1}^m \norm{1-y_js^{-1}}^{\la_j}\norm{s}^{\lambda_{m+1}}\norm{ds}.
%\end{split}
\end{equation}
\normalsize
\end{remark} 

We will also need to consider the operator $Q=Q^\bla: \mathcal H(\bla)\to \mathcal H(\bla')$ given by 
$$
(Q\psi)(y_0,...,y_m)=\frac{1}{\Gamma^F(\la_{m+1}+1)}\int_F\psi(y_0-t_0s,...,y_m-t_ms)\norm{s}^{\lambda_{m+1}}\norm{ds}. 
$$
Setting $t_0=y_0=0$ and $\varphi:=\psi|_{y_0=0}$, we have 
\begin{equation}\label{Qform} 
(Q\varphi)(y_1,...,y_m)=\frac{1}{\Gamma^F(\la_{m+1}+1)}\int_F\psi(y_1-t_1s,...,y_m-t_ms)\norm{s}^{\lambda_{m+1}}\norm{ds}. 
\end{equation} 
Recall that by \cite{EFK3}, Proposition 3.3, 
$$
S_0(y_1,...,y_m)=(\tfrac{t_1}{y_1},...,\tfrac{t_m}{y_m}).
$$
Thus, setting $\bla^*:=(-\la_0-2,\la_1,...,\la_m,-\la_{m+1}-2)$, 
we get the unitary involution 
$$
S_0: \mathcal H(\bla)\to \mathcal H(\bla^*)
$$ 
given by 
$$
(S_0\varphi)(y_1,...,y_m)=\prod_{j=1}^m \norm{\tfrac{t_j}{y_j^2}}^{-\frac{\la_j}{2}}\varphi(\tfrac{t_1}{y_1},...,\tfrac{t_m}{y_m}), 
$$
and we see that $Q=S_0RS_0$.
In fact, since for any $1\le i\le m$ the operator $S_iS_0$ commutes with $R$, we have 
$Q=S_iRS_i$ for any $0\le i\le m$. We also see that 
$$
Q^2=1,\ Q^\dagger=Q.
$$

\subsubsection{Formulas for Hecke operators} \label{heopp}

Now let us study the Hecke operators on the Hilbert space ${\mathcal H}_{V_0,...,V_{m+1}}$.
This is the setting of the {\bf  tamely ramified analytic Langlands correspondence with parameters $\lambda_j$}
for $G=PGL_2$ and  $X=\Bbb P^1$ with $N=m+2$ ramification points. It 
generalizes the setting of \cite{EFK3} where $\lambda_j=-1$ for all $j$, so that 
$$
{\mathcal H}_{V_0,...,V_{m+1}}=L^2({\rm Bun}_G^\circ(X,t_0,...,t_{m+1}),\norm{K}^{\frac{1}{2}})
$$ 
is the space of square-integrable half-densities. 
Similarly to \cite{EFK3}, Subsection 3.3, we'll denote the ramification points by $t_0,...,t_{m+1}$ to align notation with \cite{EFK3}, and  assume that $t_j\in X(F)$ for all $j$, while $t_{m+1}=\infty$. 

First we would like to write a formula for the (modified) 
Hecke operator $\Bbb H_x$ that generalizes the formula for $\Bbb H_x$ 
from \cite{EFK3}, Proposition 3.9. Recall that we have 
two components of the moduli of bundles, ${\rm Bun}_0$ and ${\rm Bun}_1$ (bundles of even and odd degree), which are identified by the Hecke modification $S_{m+1}$ at infinity. Thus, as in \cite{EFK3}, Section 3, 
we can use $S_{m+1}$ to identify 
the sectors $\mathcal H^0,\mathcal H^1$ 
of the Hilbert space corresponding to bundles of 
degree $0$ and $1$. Then the modified Hecke operator 
$\Bbb H_x: \mathcal H^0\to \mathcal H^1$ 
can be written as an endomorphism of $\mathcal H^0$. 
In fact, it is convenient to write 
this endomorphism as an operator $\mathcal H(\bla)\to \mathcal H(\bla')$. 
In other words, it maps homogeneous functions
of degree $\frac{1}{2}(\sum_{j=0}^{m} \la_j-\lambda_{m+1})$ 
to homogeneous functions of degree $1+\frac{1}{2}\sum_{j=0}^{m+1} \la_j$.

\begin{proposition}\label{hefor} The modified Hecke operator $\Bbb H_x=\Bbb H_x^\bla: \mathcal H(\bla)\to \mathcal H(\bla')$ 
is given by the formula
$$
(\Bbb H_x\psi)(y_0,...,y_m)=\int_F \psi\left(\frac{t_0-x}{s-y_0},...,\frac{t_m-x}{s-y_m}\right)\prod_{j=0}^m \norm{s-y_j}^{\la_j}\norm{ds}.
$$
\end{proposition} 

The proof of Proposition \ref{hefor} is parallel to the proof of \cite{EFK3}, Proposition 3.9. 

We also have the ordinary (unmodified) Hecke operator $H_x: \mathcal H(\bla)\to \mathcal H(\bla')$, which differs from $\Bbb H_x$ by normalization: 
$$
H_x=\prod_{j=0}^m \norm{x-t_j}^{-\frac{\la_j}{2}}\Bbb H_x.
$$
This is a special case of formula \eqref{heckoper}.  
It is not hard to check that 
$$
R_\varepsilon \circ H_x=H_x\circ R_\varepsilon
$$ 
when $\varepsilon_{m+1}=1$; however, this does not hold for 
$\varepsilon_{m+1}=-1$ since we used $S_{m+1}$ acting at $t_{m+1}=\infty$ 
to identify ${\rm Bun}_0$ and ${\rm Bun}_1$. 

\begin{remark} By setting $t_0=y_0=0$ and $\varphi(y_1,...,y_m):=\psi(0,y_1,...,y_m)$, 
we obtain the following formula for the action of $\Bbb H_x$ on homogeneous functions
of degree \linebreak $\frac{1}{2}(\sum_{j=0}^{m}\la_j-\lambda_{m+1})$:
\begin{equation}\label{killsym2}
(\Bbb H_x\varphi)(y_1,...,y_m)=
\int_F \varphi\left(\frac{t_1s-xy_1}{s(s-y_1)},...,\frac{t_ms-xy_m}{s(s-y_m)}\right)\prod_{j=0}^m \norm{s-y_j}^{\la_j}\norm{ds}.
\end{equation} 
This formula generalizes the formula of Theorem 3.6 in \cite{EFK3}. 

We can further set $t_m=y_m=1$ and get the following analog of \cite{EFK3}, Proposition 3.7. Let $\phi(y_1,...,y_{m-1})=\varphi(y_1,...,y_{m-1},1)\in L^2(F^{m-1})$. Then
$$
(\Bbb H_x\phi)(y_1,...,y_{m-1})=
$$
\scriptsize
$$
\int_F \phi\left(\frac{(t_1s-xy_1)(s-1)}{(s-y_1)(s-x)},...,\frac{(t_{m-1}s-xy_{m-1})(s-1)}{(s-y_{m-1})(s-x)}\right)
\norm{\frac{s(s-1)}{s-x}}^{-\frac{1}{2}(\sum_{j=0}^{m}\la_j-\la_{m+1})}
\prod_{j=0}^{m} \norm{s-y_j}^{\la_j}\norm{ds}.
$$
\normalsize
Similarly to \cite{EFK3}, Subsection 3.3, define the unitary operator
$U_{s,x}$ on $L^2(F^{m-1})$ by 
$$
(U_{s,x}\phi)(y_1,....,y_{m-1}):=
$$
$$
\phi\left(\frac{(t_1s-y_1x)(s-1)}{(s-y_1)(s-x)},...,
\frac{(t_{m-1}s-y_{m-1}x)(s-1)}{(s-y_{m-1})(s-x)}\right)\norm{\frac{s(s-1)}{s-x}}^{-\frac{1}{2}\sum_{j=1}^{m-1}\la_j}\prod_{j=1}^{m-1}\norm{\frac{t_j-x}{(s-y_j)^2}}^{-\frac{\la_j}{2}}. 
$$
Then we get 
\begin{multline} \label{uniop} 
H_x\phi=
\norm{x}^{-\frac{\lambda_0}{2}}\norm{x-1}^{-\frac{\lambda_m}{2}}
\times \\ \int_F  \norm{s}^{\frac{1}{2}(\la_0-\la_m+\la_{m+1})}\norm{s-1}^{\frac{1}{2}(-\la_0+\la_m+\la_{m+1})}\norm{s-x}^{\frac{1}{2}(\la_0+\la_m-\la_{m+1})}U_{s,x}\phi\norm{ds},
\end{multline} 
which generalizes \cite{EFK3}, Proposition 3.7. 
From this formula it follows by the argument of \cite{EFK3}, Proposition 3.10 that $\Bbb H_x$ is a bounded operator which depends norm-continuously on $x$ for $x\ne t_j,\infty$. 
\end{remark} 

\subsubsection{Properties of Hecke operators} 
The properties of the operator $H_x$ are analogous 
to those in the untwisted case (\cite{EFK3}, Section 3).
To avoid confusion, from now on the operators $R,Q,H_x: \mathcal H(\bla)\to \mathcal H(\bla')$ will be denoted by $R_+,Q_+,H_{x+}$ 
and the operators  $R,Q,H_x: \mathcal H(\bla')\to \mathcal H(\bla)$
by $R_-,Q_-,H_{x-}$ (thus $Q_+^\dagger=Q_-,R_+^\dagger=R_-$). 

\begin{proposition}\label{prophecke} (i) The operators $H_{x,\pm}$, $x\in \Bbb P^1(F), x\ne t_j$ are compact.

(ii) $H_{x-}H_{y+}=H_{y-}H_{x+}$, $x,y\in \Bbb P^1(F), x,y\ne t_j$. 

(iii) $H_{x+}^{\dagger}=H_{x-}$.  
\end{proposition} 

\begin{proof} (i) is proved analogously to \cite{EFK3} Proposition 3.13. 
(ii), (iii) are proved analogously to 
\cite{EFK3}, Proposition 3.11. (ii) can also be checked explicitly from the formula of Proposition \ref{hefor} as explained in \cite{EFK3}, Remark 3.28. 
\end{proof} 

Define the {\bf  full Hecke operator} on $\mathcal H=\mathcal H^0\oplus \mathcal H^1=\mathcal H(\bla)\oplus \mathcal H(\bla')$ by the formula
$$
H_{x,{\rm full}}=\begin{pmatrix} 0& H_{x-}\\ H_{x+} & 0\end{pmatrix}.
$$ 
It follows that the operators $H_{x,{\rm full}}$ are self-adjoint and pairwise commuting. 

\subsubsection{Asymptotics of Hecke operators and the spectral decomposition} 
Let us now discuss the asymptotics of Hecke operators as $x\to \infty$. 
Set $c:=\lambda_{m+1}+1$.

\begin{proposition}\label{asym} (i) If $c\ne 0$ then 
$$
\norm{x}^{-\frac{1}{2}}H_{x\pm}=\Gamma^F(\pm c){\norm x}^{\pm\frac{c}{2}}Q_\pm+\Gamma^F(\mp c){\norm x}^{\mp\frac{c}{2}}R_{\pm}
+o(1),\ x\to \infty.
$$
Thus 
$$
\norm{x}^{-\frac{1}{2}}H_{x,{\rm full}}=
\Gamma^F(c){\norm x}^{\frac{c}{2}}
D+\Gamma^F(-c){\norm x}^{-\frac{c}{2}}D^\dagger +o(1),\ x\to \infty,
$$
where 
$$
D :=\begin{pmatrix} 0& R_{-}\\ Q_+ & 0\end{pmatrix}.
$$

(ii) If $c=0$ then one has 
$$
\norm{x}^{-\frac{1}{2}}H_{x,\pm}=\log\norm{x}+M+o(1),\ x\to \infty, 
$$
where 
\scriptsize
$$
(M\varphi)(y_1,...,y_m):=\int_F \left(\varphi(y_1-t_1s,...,y_m-t_ms)+\frac{\varphi(\frac{y_1}{1-y_1s^{-1}},...,\frac{y_m}{1-y_ms^{-1}})}{\prod_{j=1}^m \norm{1-y_js^{-1}}^{-\la_j}}
-\varphi(y_1,...,y_m)\right)\norm{\frac{ds}{s}}.
$$
\end{proposition} 

Note that Proposition \ref{asym}(ii) generalizes \cite{EFK3}, Propositions 3.15(i) and 3.21. 

\begin{proof} (i) We follow the proof of \cite{EFK3}, Proposition 3.21. 
 By \eqref{killsym2} we have 
 \begin{equation}\label{killsym2a}
\begin{aligned}
\norm{x}^{-\frac{1}{2}}(H_{x+}\varphi)(y_1,...,y_m)=\quad\quad\quad\quad\quad\quad\quad\quad\quad{}\\
\norm{x}^{-\frac{c}{2}}\int_F \varphi\left(\frac{t_1sx^{-1}-y_1}{1-y_1s^{-1}},...,\frac{t_msx^{-1}-y_m}{1-y_ms^{-1}}\right)\prod_{j=1}^m \norm{1-y_js^{-1}}^{\la_j}\norm{s}^{c}\norm{\frac{ds}{s}}.
\end{aligned} 
\end{equation} 
Now, as explained in the proof of \cite{EFK3}, Proposition 3.21, in the limit $x\to \infty$
the curve $Z_{x,\bold y}$ with parametrization $s\mapsto (\frac{t_1sx^{-1}-y_1}{1-y_1s^{-1}},...,\frac{t_msx^{-1}-y_m}{1-y_ms^{-1}})$ 
along which we are integrating in \eqref{killsym2a} falls apart into two components 
corresponding to the regimes when 
$s=s(x)$ has a finite limit as $x\to \infty$ and when $sx^{-1}$ has a finite limit when $x\to \infty$, respectively. As a result, similarly to the proof of \cite{EFK3}, Proposition 3.21, the integral \eqref{killsym2a} is asymptotic to the sum of two integrals over these components. Namely, we have 
$$
\norm{x}^{-\frac{1}{2}}(H_{x+}\varphi)(y_1,...,y_m)=
$$
$$
\norm{x}^{-\frac{c}{2}}\int_F \varphi\left(\frac{y_1}{1-y_1s^{-1}},...,\frac{y_m}{1-y_ms^{-1}}\right)\prod_{j=1}^m \norm{1-y_js^{-1}}^{\la_j}\norm{s}^{c}\norm{\frac{ds}{s}}+
$$
$$
\norm{x}^{\frac{c}{2}}\int_F \varphi(y_1-t_1s,...,y_m-t_ms)\norm{s}^{c}\norm{\frac{ds}{s}}+o(1),\ x\to \infty.
$$
Now, the first integral is the operator $\Gamma(-c)R_+$ (formula \eqref{killsym1}), while the second integral is the operator $\Gamma(c)Q_+$ (formula \eqref{Qform}), which implies the claimed asymptotics for $H_{x+}$. The asymptotics for $H_{x-}$ is obtained by replacing $c$ by $-c$. 

(ii) follows from (i) by taking the limit $c\to 0$. Namely, write (i) in the form
$$
\norm{x}^{-\frac{1}{2}}(H_{x+}\varphi)(y_1,...,y_m)=
$$
$$
(\norm{x}^{-\frac{c}{2}}-1)\int_F \varphi\left(\frac{y_1}{1-y_1s^{-1}},...,\frac{y_m}{1-y_ms^{-1}}\right)\prod_{j=1}^m \norm{1-y_js^{-1}}^{\la_j}\norm{s}^{c}\norm{\frac{ds}{s}}+
$$
$$
(\norm{x}^{\frac{c}{2}}-1)\int_F \varphi(y_1-t_1s,...,y_m-t_ms)\norm{s}^{c}\norm{\frac{ds}{s}}+
$$
$$
\int_F \left(\varphi(y_1-t_1s,...,y_m-t_ms)+\frac{\varphi(\frac{y_1}{1-y_1s^{-1}},...,\frac{y_m}{1-y_ms^{-1}})}{\prod_{j=1}^m \norm{1-y_js^{-1}}^{-\la_j}}
-\varphi(y_1,...,y_m)\right)\norm{s}^{c}\norm{\frac{ds}{s}}+o(1)
$$
as $x\to \infty$.
Now, each of the first two summands tends to $\frac{1}{2}\log\norm{x}$ as $c\to 0$, while the third summand tends to $M$, as desired.\footnote{Here it needs to be checked that 
the $o(1)$ term remains $o(1)$ as $c\to 0$. We leave this argument to the reader.}
\end{proof} 

\begin{corollary}\label{specdec} (i) We have $\cap_x {\rm Ker}H_{x\pm}=
0$, $\cap_x {\rm Ker}H_{x,{\rm full}}=0$. 

(ii) We have a spectral decomposition 
$$
\mathcal H=\oplus_k \mathcal H_k^\pm,
$$ 
where $\mathcal H_k^\pm$ are finite dimensional joint eigenspaces of $H_{x,{\rm full}}$ with eigenvalues 
$\pm \beta_k(x)$. 
\end{corollary} 

\begin{proof} (i) It suffices to show that $\cap_x {\rm Ker}H_{x,{\rm full}}=0$. But this follows 
from Proposition \ref{asym} and the fact that $R_\pm,Q_\pm$ are unitary operators, hence so is $D$.

(ii) immediately follows from (i) and the compactness of $H_{x,{\rm full}}$. 
\end{proof} 

Let $\pm \delta_k\in \Bbb C$ be the eigenvalue of $D$ on $\mathcal H_k^\pm$, so $|\delta_k|=1$.  
We choose the signs so that ${\rm arg}\delta_k\in (-\frac{\pi}{2},\frac{\pi}{2}]$. Then 
Proposition \ref{asym} implies the following asymptotics for
$\beta_k(x)$.  

\begin{corollary} We have 
$$
\beta_k(x)=\norm{x}^{\frac{1}{2}}(2{\rm Re}(\delta_k\Gamma^F(c){\norm x}^{\frac{c}{2}})+o(1)),\ x\to \infty.
$$
\end{corollary} 

So setting $\delta_k^*=e^{2\pi i\theta_k}:=\frac{c\Gamma^F(c)}{|c\Gamma^F(c)|}\delta_k$, $-\pi<\theta_k\le \pi$, 
we get 
$$
\norm{x}^{-\frac{1}{2}}\beta_k(x)=\frac{|c\Gamma^F(c)|}{c}(e^{2\pi i\theta_k}{\norm x}^{\frac{c}{2}}-e^{-2\pi i\theta_k}{\norm x}^{-\frac{c}{2}})+o(1)=
$$
$$
\frac{|c\Gamma^F(c)|}{c}\left(\cos \theta_k\cdot ({\norm x}^{\frac{c}{2}}-{\norm x}^{-\frac{c}{2}})+i\sin \theta_k\cdot ({\norm x}^{\frac{c}{2}}+{\norm x}^{-\frac{c}{2}})\right)+
o(1),\ x\to \infty.
$$
Thus, Proposition \ref{asym}(ii) implies that 
$$
\theta_k=\frac{c}{2i}\mu^{(k)}+o(c),\ c\to 0.
$$
where $\mu^{(k)}\in \Bbb R$ are the eigenvalues of $M$. 

Thus we see that we have orthonormal bases $\lbrace \bold e_k^0\rbrace$ of $\mathcal H^0$ 
and $\lbrace \bold e_k^1\rbrace$ of $\mathcal H^1$ such that 
$$
Q_+\bold e_k^0=\delta_k \bold e_k^1,\ Q_-\bold e_k^1=\delta_k^{-1}\bold e_k^0,\ R_+\bold e_k^0=\delta_k^{-1}\bold e_k^1,\ R_-\bold e_k^1=\delta_k \bold e_k^0,
$$
and the eigenvectors of $H_{x,{\rm full}}$ with eigenvalues $\pm\beta_k(x)$ are $\bold e_k^\pm=\bold e_k^0\pm \bold e_k^1$. Thus the vectors $\bold e_k^0$ satisfy the equation 
$$
R_{-}H_{x+}\bold e_k^0=\delta_k\beta_k(x)\bold e_k^0.
$$
So the spectral problem for the operator $H_{x,{\rm full}}$ on $\mathcal H$ is equivalent 
to the spectral problem for $R_-H_{x+}$ on $\mathcal H^0$, and the eigenvalues 
of $R_-H_{x+}$ are 
$$
\widehat\beta_k(x):=\delta_k\beta_k(x)=\norm{x}^{\frac{1}{2}}(\delta_k^2\Gamma(c)\norm{x}^{\frac{c}{2}}+\Gamma(-c)\norm{x}^{-\frac{c}{2}}+o(1)),\ x\to \infty.
$$
Note that $\widehat\beta_k(x)$ do not depend on the above choice of signs for $\delta_k$. 

\begin{example} Let $m=1$ (3 points), then $H(\bla)\cong \Bbb C$ by sending 
$\varphi$ to $\varphi(1)$. 
Let 
$$
\lambda_0=-1+a,\ \lambda_1=-1+b,\ t_1=y_1=1.
$$ 
Hence 
$$
Q_+=\frac{1}{\Gamma^F(c)}\int_F \norm{1-s}^{\frac{a+b-c-1}{2}}\norm{s}^{c-1}\norm{ds}=\frac{\Gamma^F(\tfrac{a+b-c+1}{2})}{\Gamma^F(\tfrac{a+b+c+1}{2})},
$$
$$
R_{-}=\frac{1}{\Gamma^F(c)}\int_F \norm{1-s}^{\frac{-a+b-c-1}{2}}\norm{s}^{\frac{a-b-c-1}{2}}\norm{ds}=
\frac{\Gamma^F(\tfrac{a-b-c+1}{2})}{\Gamma^F(\tfrac{a-b+c+1}{2})}.
$$
Thus 
$$
\delta_k=\sqrt{\frac{\Gamma^F(\tfrac{a+b-c+1}{2})\Gamma^F(\tfrac{a-b-c+1}{2})}{\Gamma^F(\tfrac{a+b+c+1}{2})\Gamma^F(\tfrac{a-b+c+1}{2})}}.
$$
Also by \eqref{uniop}, 
$$
H_{x+}=\int_F  \norm{s}^{\frac{a-b+c-1}{2}}\norm{s-1}^{\frac{-a+b+c-1}{2}}\norm{s-x}^{\frac{a+b-c-1}{2}}\norm{ds},
$$
so Proposition \ref{asym}(i) takes the form 
$$
\int_F  \norm{s}^{\frac{a-b+c-1}{2}}\norm{s-1}^{\frac{-a+b+c-1}{2}}\norm{s-x}^{\frac{a+b-c-1}{2}}\norm{ds}=
$$
$$
B^F(\tfrac{-a+b+c+1}{2},\tfrac{a-b+c+1}{2}){\norm x}^{\frac{a+b-c-1}{2}}+
B^F(\tfrac{a+b-c+1}{2},\tfrac{-a-b-c+1}{2}){\norm x}^{\frac{a+b+c-1}{2}}+o(\norm{x}^{-\frac{1}{2}}),\ x\to \infty.
$$
This asymptotic formula is also a special case of \eqref{hgin}, when ${\rm Re}\beta={\rm Re}\gamma=\frac{1}{2}$. 
\end{example} 

\begin{remark} Analogously to \cite{EFK3}, Proposition 3.15(i), 
a similar asymptotic formula for $H_x$ to Proposition \ref{asym} holds when $x\to t_j$ 
$0\le j\le m$, with an additional factor $S_j$: if $\lambda_j\ne -1$ then 
$$
\norm{x-t_j}^{-\frac{1}{2}}H_{x}=\Gamma^F(\lambda_j+1)\norm {x-t_j}^{-\frac{\lambda_j+1}{2}}R^{(j)}S_j+\Gamma^F(-\lambda_j-1)\norm{x-t_j}^{\frac{\lambda_j+1}{2}}S_jR^{(j)}
+o(1)
$$
as $x\to t_j$, where $R^{(j)}:=R_{1,...,1,-1,1,...,1}$ with $-1$ in
the $j$-th position.
The proof and the computation of the limit $\la_j\to -1$ are parallel to the case $x\to \infty$. 
\end{remark} 

\begin{remark}\label{exte} This analysis may be extended to the {\bf  complementary series}, i.e., 
when some $\la_j$, instead of being in $-1+i\Bbb R$, are allowed to lie in the interval $(-2,0)$. For simplicity assume that $\la_{m+1}\in -1+i\Bbb R$, so that $M_{\la_{m+1}}$ is tempered and we can define a reasonable multiplicity space ${\rm Mult}_{PGL_2(F)}(M_{\la_{m+1}}^*,M_{\la_0}\otimes...\otimes M_{\la_m})$. If $\la\in (-2,0)$, we still have $M_\la=L^2(\Bbb P^1(F),\norm{K}^{-\frac{\la}{2}})$, but now with inner product 
$$
(f,g)=\frac{1}{\Gamma^F(-\la-1)}\int_{\Bbb P^1(F)^2}f(y)\overline{g(z)}{\norm{y-z}^{-\la-2}}\norm{dydz}^{\frac{\la+2}{2}}
$$
(more precisely, the integral converges for $\la\in (-2,-1)$ but analytically continues to $\la\in (-2,0)$ as a positive definite inner product). Thus the inner product in $\mathcal H(\bla)$ (translation invariant homogeneous functions of degree $\frac{1}{2}(\sum_{j=0}^{m}\la_j-\la_{m+1})$) also has to be modified accordingly and will become more complicated, but the formula for Hecke operators remains the same. 
\end{remark} 

\begin{remark}\label{schwar} At least for $\ell=1$, one should be able to extend this theory to the case of {\bf  admissible} representations $V_i$ (not necessarily tempered, or even unitarizable) by working in the Schwartz space context instead of $L^2$ space (for example, using the approach of \cite{BK2}). For instance, in this context the extension to complementary series 
from Remark \ref{exte} should be much more straightforward -- we don't need to worry about positive inner products and can just do analytic continuation with respect to the Casimir eigenvalues $\frac{1}{2}(\lambda_j+1)^2$ (which are no longer required to be real). 
\end{remark} 

\begin{remark} The material of Subsection \ref{pgl2} generalizes in a straightforward way 
when $\lambda_j$ are taken to be arbitrary multiplicative characters of $F$ 
of the form $\lambda_j(y)=\norm{y}^{-1}\lambda_j^0(y)$, where $\lambda_j^0$ are unitary 
characters. The above setting is the special case when $\lambda_j^0$ are imaginary powers of the norm. 
\end{remark} 

\section{Analytic Langlands correspondence over $\Bbb C$} 

In this section we discuss the analytic Langlands correspondence over
$\Bbb C$, including various twists. We begin with recalling the basic setting discussed in \cite{EFK2}. 

\subsection{The general setting of the analytic Langlands correspondence over $\Bbb C$} \label{ALcomp} 

When one talks of Langlands correspondence for a group $G$, one
usually means not just a formulation of a spectral problem for Hecke
operators, but also a parametrization of their spectrum by data
related to the Langlands dual group $G^\vee$. Such a description is
essentially available for the {\bf arithmetic} Langlands
correspondence for a curve $X$ over a finite field. In this case the
Langlands conjecture describes the spectrum of Hecke operators in
terms of \'etale $G^\vee$-local systems on $X$. On the other hand, for
the {\bf analytic} Langlands correspondence dealing with curves over a
general non-archimedean field local field $F$, we cannot yet formulate
even a conjectural description of the spectrum. But for archimedean
fields we can use quantum Hitchin Hamiltonians commuting with Hecke
operators to describe the spectrum (see Subsection \ref{archi}). The
most complete conjectural picture exists for $F=\Bbb
C$ (\cite{EFK1,EFK2,EFK3}); we discuss it in this section. The case
of curves over $F=\Bbb R$ is discussed in Section 4.

Consider first the unramified case. Let $B^\vee$ be a Borel subgroup
of $G^\vee$ with maximal torus $T^\vee$, $Z^\vee$ the center of
$G^\vee$, $\g^\vee:={\rm Lie}G^\vee$, $\mathfrak{b}^\vee:={\rm
  Lie}B^\vee$, $\mathfrak{t}^\vee:={\rm Lie}T^\vee$.  Let
$Q^\vee\subset \Lambda$ be the root lattice of $G^\vee$, then
$\Hom(\Lambda/Q^\vee,\Bbb C^\times)=Z^\vee$. Let $d_i,
i=1,\ldots,\on{rank} G$ be the degrees of the basic invariants for $G$
and $G^\vee$.

\begin{definition}[\cite{BD,BD2}] \label{operde}
  A $G^\vee$-{\bf oper} on $X$ is a triple
  $(\mcE,\mcE_{B^\vee},\nabla)$, where $\mcE$ is a $G^\vee$-bundle on $X$,
  $\mcE_{B^\vee}$ is its $B^\vee\cap [G^\vee,G^\vee]$-reduction, and
  $\nabla$ is a connection on $\mcE$ which has the form 
 $$
 \nabla=d+(f+b(z))dz,\ b\in \mathfrak{b}^\vee[[z]]
 $$
for any trivialization of $\mcE_{B^\vee}$ (and hence $\mcE$) on the formal
neighborhood of any
point $x_0 \in X$, where $f$ is the lower nilpotent element of a
principal $\mathfrak{sl}_2$-triple $\{ e,h,f \} \subset \g^\vee$ such
that $h\in {\mathfrak t}^\vee$ and $e\in {\mathfrak b}^\vee$ and $z$
is a formal coordinate at
 $x_0$.\footnote{In fact, this definition is more restrictive than the
 one in \cite{BD,BD2}, where $\mcE_{B^\vee}$ is assumed to be a
 $B^\vee$-bundle. But the two definitions coincide when $G^\vee$ is
 semisimple.}
\end{definition}

%Equivalently, one may say that a $G^\vee$-oper on $X$ is a $G^\vee$-bundle 
%on $X$ with a reduction to $B^\vee\cap [G^\vee,G^\vee]$ and a connection 
%$\nabla$ satisfying the condition of Definition \ref{operde}.

The above $\sw_2$ triple defines a {\bf principal homomorphism}
$\phi: SL_2\to G^\vee$.

Since any two Borel subgroups of $G^\vee$ are conjugate to each other,
any two maximal tori in a given Borel subgroup $B^\vee$ of $G^\vee$
are conjugate to each other by an element of $B^\vee$, and any two
$\sw_2$-triples of the kind considered in the above definition are
conjugate by an element of the torus $T^\vee$, we obtain the following
result.

\begin{lemma}    \label{doesnotdep}
Let $B'{}^\vee$ be another Borel subgroup of $G^\vee$. The spaces
of $G^\vee$-opers corresponding to $B^\vee$ and $B'{}^\vee$ are
canonically isomorphic.
\end{lemma}

Given a flat $G^\vee$-bundle $(\mcE,\nabla)$, we may speak of an
{\bf oper structure} on it, which is a reduction $\mcE_{B^\vee}$ of
$\mcE$ to $B^\vee\cap [G^\vee,G^\vee]$ satisfying the above
condition.

\begin{lemma}[\cite{BD,BD2}]    \label{atmost}
A flat $G^\vee$-bundle can have at most one oper structure.
\end{lemma}

Thus it makes sense to say that a given flat $G^\vee$-bundle is or is
not an oper.

\begin{example}\label{tor} If $T^\vee$ is a torus then a $T^\vee$-oper on $X$ is any connection $\nabla$ on the trivial $T^\vee$-bundle on $X$. Thus $\nabla=d+\omega$ where $\omega\in H^0(X,K_X\otimes {\rm Lie}T^\vee)$.
\end{example} 

As explained in \cite{BD,BD2}, $G_{\rm ad}^\vee$-opers on $X$ are
parametrized by a certain affine space ${\rm Op}_{G_{\rm ad}^\vee}(X)$
of dimension $({\rm g}-1)\dim G$ -- a torsor over the Hitchin base
\begin{equation}    \label{Hitch}
{\bf  Hitch}:=\oplus_i H^0(X,K_X^{\otimes d_i}).
\end{equation}
By Example \ref{tor}, this is also true for a torus, hence for a product of a torus with 
an adjoint group, i.e., for any $G^\vee$ such that $[G^\vee,G^\vee]$ is adjoint. 
In other words, denoting by $Z^\vee_{\rm der}$ the intersection 
$Z^\vee\cap [G^\vee,G^\vee]$ of $Z^\vee$ with the derived group $[G^\vee,G^\vee]$, we see that this description is always valid for $G^\vee/Z^\vee_{\rm der}$-opers. 

More generally, for arbitrary $G^\vee$ the variety ${\rm Op}_{G^\vee}(X)$ of $G^\vee$-opers on $X$ is a torsor over the affine space ${\rm Op}_{G^\vee/Z^\vee_{\rm der}}(X)$ with fiber $H^1(X,Z^\vee_{\rm der})$. Moreover, any choice of a spin structure $K_X^{\frac{1}{2}}$ on $X$ gives rise to a splitting of this torsor, i.e., fixes a canonical component ${\rm Op}_{G^\vee}^0(X)\cong {\rm Op}_{G^\vee/Z^\vee_{\rm der}}(X)$. Indeed, consider the unique up to isomorphism (by the Riemann-Roch theorem) 
nontrivial extension 
\begin{equation}    \label{nontr}
0\to K_X^{\frac{1}{2}}\to \mathcal E_{SL_2}\to K_X^{-\frac{1}{2}}\to 0
\end{equation}
defining an $SL_2$-bundle $\mathcal E_{SL_2}$ on $X$, and define
$\mathcal E_{G^\vee}:=\phi(\mathcal E_{SL_2})$. If the genus ${\rm g}
> 1$, then according to \cite{BD}, any connection on $\mathcal
E_{G^\vee}$ is an oper. Moreover, there is a canonical isomorphism
$H^0(X,K_X\otimes {\rm ad}\mathcal E_{G^\vee})\cong {\bf Hitch}$
identifying the translation actions of $H^0(X,K_X\otimes {\rm
  ad}\mathcal E_{G^\vee})$ on connections and of ${\bf Hitch}$ on
opers. So such opers form a component of ${\rm Op}_{G^\vee}(X)$, which
we will denote by ${\rm Op}_{G^\vee}^0(X)$.\footnote{If ${\rm g} = 0$,
the bundle underlying $SL_2$-opers is still the non-trivial extension
$\mathcal E_{SL_2}$ given by \eqref{nontr} but it is isomorphic to the
trivial $SL_2$-bundle in this case. The bundle underlying
$G^\vee$-opers is $\mathcal E_{G^\vee} = \phi(\mathcal E_{SL_2})$, so
it is isomorphic to the trivial $G^\vee$-bundle, and there is a unique
$G^\vee$-oper that corresponds to the trivial connection on $\mathcal
E_{G^\vee}$. If ${\rm g = 1}$, we should take instead the trivial
extension \eqref{nontr} and set $\mathcal E_{G^\vee}=\phi(\mathcal
E_{SL_2})$. In this case $K_X \simeq \OO_X$ and so $K_X^{\pm
  \frac{1}{2}}$ is a square root of $\OO_X$. With a choice of such a
square root, we obtain a component in the space of $G^\vee$-opers
which is isomorphic to ${\bf Hitch}$.} In other words, an oper from
this component is just a connection on a certain fixed $G^\vee$-bundle
$\mathcal E_{G^\vee}$.\footnote{Note that the associated
$PGL_2$-bundle to $\mathcal E_{SL_2}$ is independent on the choice of
$K_X^{\frac{1}{2}}$, Thus if $\phi$ factors through $PGL_2$ then the
component ${\rm Op}_{G^\vee}^0(X)$ does not depend on the choice of
$K_X^{\frac{1}{2}}$.}

Given a $G^\vee$-oper $\chi = (\mcE,\mcE_{B^\vee},\nabla)$, we have
the underlying flat $G^\vee$-bundle $(\mcE,\nabla)$ and the
corresponding $G^\vee$-local system on $X$. Recall that any flat
$G^\vee$-bundle has at most one oper structure. Also, according to
\cite{BD2}, \S 1.3, the automorphism group of every flat
$G^\vee$-bundle underlying a $G^\vee$-oper is the center
$Z^\vee$. Therefore, the space ${\rm Op}_{G^\vee}(X)$ can be realized
as a certain half-dimensional (in fact, Lagrangian in the Atiyah-Bott
holomorphic symplectic structure) complex analytic submanifold of the
complex manifold ${\rm LocSys}^\circ_{G^\vee}(X)$ of $G^\vee$-local
systems on $X$ with the smallest possible group of automorphisms;
namely, $Z^\vee$. The group $H^1(X,Z^\vee_{\rm der})$
naturally acts on ${\rm Op}_{G^\vee}(X)$.

If we choose a base point $x_0\in X$, then the $G^\vee$-local system on
$X$ corresponding to $\chi$ gives rise to a monodromy representation
$\rho_\chi: \pi_1(X,x_0)\to G^\vee(\C)$ (it is well-defined up to
conjugation by an element of $G^\vee(\C)$). In this section, to
simplify our notation when we discuss monodromy representations, we
will write $G^\vee$ instead of $G^\vee(\C)$.

\begin{remark}    \label{monim}
Let $G^\vee = SL_2$ and $\chi$ any $SL_2$-oper on $X$. Denote by $M$
the Zariski closure of the image of the monodromy representation
$\rho_\chi: \pi_1(X,p)\to SL_2$. We claim that for ${\rm g} > 1$,
$M=SL_2$. Indeed, by passing to a cover of $X$ if needed, we can
assume without loss of generality that $M$ is connected. Hence it is
either contained in a Borel subgroup of $SL_2$ or is $SL_2$
itself. If it's the former, then the vector bundle $\mathcal E_{SL_2}$
would contain a line subbundle $\mathcal L$ preserved by the oper
connection. Then $\on{deg}(\mathcal L)=0$, and since
$\on{deg}(K_X^{-\frac{1}{2}})<0$ for ${\rm g} > 1$, it follows that
the map $\mathcal L \to K_X^{-\frac{1}{2}}$ defined by the extension
\eqref{nontr} is 0. But then $\mathcal L$ must be isomorphic to
$K_X^{\frac{1}{2}}$ which is impossible since
$\on{deg}(K_X^{\frac{1}{2}})>0$. This also implies an analogous
statement for $G^\vee = PGL_2$: if ${\rm g} > 1$, then for any
$PGL_2$-oper $\chi$ the Zariski closure of the image of the
corresponding monodromy representation $\rho_\chi: \pi_1(X,p)\to
PGL_2(C)$ is equal to $PGL_2$.
\end{remark}

Given a $G^\vee$-local system $\rho$ on $X$ and an algebraic representation 
$\varphi: G^\vee\to GL_N(\Bbb C)$, we have a $GL_N(\Bbb C)$-local system
$\varphi(\rho)$ on $X$. There exists a unique 
$G^\vee$-local system $\overline\rho$ such that for every $\varphi$, $\varphi(\overline\rho)\cong \overline{\varphi(\rho)}$.

\begin{definition} We say that a $G^\vee$-local system $\rho$ on $X$ is ${\bf  real}$ if 
$\rho\cong \overline \rho$.
\end{definition} 

Thus the space ${\rm LocSys}^\circ_{G^\vee}(X)_{\Bbb R}$ of real 
local systems is a half-dimensional real submanifold of ${\rm LocSys}^\circ_{G^\vee}(X)$ (in fact, Lagrangian under the real part of the holomorphic symplectic form). 
As explained in \cite{EFK2}, Remark 1.9, $\rho$ is real iff
its monodromy group can be conjugated into 
an inner real form $G^\vee_{\Bbb R}$ of $G^\vee$. 

\begin{definition} 
A {\bf  real} $G^\vee$-{\bf  oper} is a $G^\vee$-oper such that
the corresponding $G^\vee$-local system $\rho$ is real.
\end{definition} 

In other words, a real oper is an intersection point of the above two half-dimensional submanifolds. It is expected (and known for $G^\vee=SL_2$,  see \cite{Fa}) that these manifolds intersect transversally, so the set of real opers is discrete. 
Moreover, it is conjectured in \cite{EFK2} that for real opers the inner form $G^\vee_{\Bbb R}$ is, in fact, split, and this is known for $G^\vee=SL_2$ (\cite{GKM}).\footnote{This is also easy to see for any $G^\vee$ in the tamely ramified case, see  \cite{EFK2}, Remark 1.9.}

Note that we may also consider the complex conjugate submanifold $$\overline{\rm Op}_{G^\vee}(X)\subset {\rm LocSys}^\circ_{G^\vee}(X).$$ The points of this submanifold are local systems that are realized by an antiholomorphic $G^\vee$-oper (which we call an {\bf  anti-oper} for short). This is a third Lagrangian submanifold of ${\rm LocSys}^\circ_{G^\vee}(X)$ with respect to the real part of the holomorphic symplectic form, which intersects the other two submanifolds exactly at the same points where they intersect each other (i.e., at real opers). In other words, a real oper is the same thing as a real anti-oper and also the same as a local system that's both an oper and an anti-oper.   

Now, the main conjecture of \cite{EFK2} is as follows (we formulate it
for semisimple $G$, as for abelian $G$ it is not difficult and proved
in \cite{F4}, see also \cite{EFK2}). 
Recall that the manifold ${\rm Bun}^\circ_G(X)$ is the union of 
connected components ${\rm Bun}^\circ_{G,\alpha}(X)$ labeled by 
the first Chern class $c\in H^2(X,\pi_1(G))=\pi_1(G)$ of a $G$-bundle on $X$, and 
that $\pi_1(G)=Z^{\vee *}=\Lambda/Q^\vee$.

\begin{conjecture}\label{mainc} (i) The Hilbert space $\mathcal H=L^2({\rm Bun}^\circ_G(X))$ can be written as an orthogonal direct sum of 1-dimensional spaces
$$
\mathcal H=\bigoplus_{\rho,\beta} \mathcal H_{\rho,\beta}
$$
invariant under Hecke operators, where $\rho$ runs over real $G^\vee$-opers in ${\rm Op}^0_{G^\vee}(X)$, and $\beta$ runs over eigenvalues of Hecke operators corresponding to $\rho$. The quantum Hitchin Hamiltonians act on $\mathcal H_{\rho,\beta}$ via the character $\rho$. 

(ii) The eigenvalue $\beta_\lambda(x,\overline x)$ 
for the Hecke operator $H_{x,\lambda}$ in $\mathcal H_{\rho,\beta}$ is
given up to scaling by the formula of \cite{EFK2}, Conjecture 5.1 (see
Conjecture \ref{51} below).

(iii) The set of such eigenvalues corresponding to a given $\rho$ is a torsor over the group $Z^\vee=\Hom(\Lambda/Q^\vee,\Bbb C^\times)$ where the action of this group on eigenvalues is by multiplication, i.e.
$$
(\xi\circ \beta)_{\lambda}=\xi(\lambda)\beta_{\lambda}.
$$ 

(iv) 
The decomposition $\mathcal H=\oplus_{c\in Z^{\vee *}}L^2({\rm Bun}^\circ_{G,c}(X))$
is invariant under quantum Hitchin Hamiltonians, and on each summand $L^2({\rm Bun}^\circ_{G,c}(X))$ they have simple spectrum labeled by real $G^\vee$-opers $\rho$ in ${\rm Op}^0_{G^\vee}(X)$. The Hecke operators $H_{x,\lambda}$ act between these summands, acting on labels $c$ by $c\mapsto c+\lambda$, which gives rise to the action in (iii). 
\end{conjecture}

In in Corollary \ref{eigH1} below, we will recall the
  formula for the Hecke eigenvalues $\beta_\la(x,\overline x)$
  obtained in \cite{EFK2}, Corollary 1.19 in the case $G=PGL_n$ and
  $\la=\omega_1$. For $G$ of types $B_\ell, C_\ell$, or $G_2$ and $\la
  = \omega_1$, we conjecture an analogous formula in Conjecture
  \ref{eigHG}. In the general case, the formula for the Hecke
  eigenvalues is given in Conjecture \ref{51} (it coincides with
  Conjecture 5.1 of \cite{EFK2}).

Note that we have a free action of the finite group $H^1(X,Z)$ on ${\rm Bun}_G^\circ(X)$, 
where $Z$ is the center of $G$, and this action commutes with Hecke and quantum Hitchin
operators. Hence this group acts by a character $\chi_\rho\in H^1(X,Z)^*\cong H^1(X,Z^*)$ 
on each (1-dimensional) joint eigenspace of these operators corresponding to a real $G^\vee$-local system $\rho$ and some choice of eigenvalue $\beta$. Let us explain how to compute 
$\chi_\rho$. 

Let $G^\vee_{\rm sc}$ be the simply connected cover of $G^\vee$. 
Recall that we have an exact sequence 
$$
1\to H^1(X,\pi_1(G^\vee))\to H^1(X,G^\vee_{\rm sc})\to H^1(X,G^\vee)\to H^2(X,\pi_1(G^\vee)),
$$
and that $\pi_1(G^\vee)=Z^*$. Thus every $G^\vee$-local system $\rho: \pi_1(X)\to G^\vee$ has a first Chern class $c_\rho\in H^2(X,Z^*)$. However, as explained above, if $\rho$ 
is an oper then as a holomorphic bundle it reduces to the principal $SL_2$, 
so $c_\rho=1$ (as $SL_2(\Bbb C)$ is simply connected). 
Moreover, in this case there is a unique lift of $\rho$ 
to a $G^\vee_{sc}$-oper $\rho'$ in the canonical component 
${\rm Op}^0_{G^\vee_{\rm sc}}(X)$. Now, the reality of $\rho$ means that $\rho\cong \overline \rho$, where $\overline \rho$ is the complex conjugate of $\rho$, but then $\rho'$ is not necessarily real: we have $\rho'\cong \eta\overline{\rho'}$ for a unique $\eta\in H^1(X,Z^*)$.
We expect that 
$\chi_\rho=\eta$.\footnote{Note that the isomorphism $H^1(X,Z)^*\cong H^1(X,Z^*)$ comes from the cup product on $H^1(X,\Bbb Z)$, so it is defined uniquely up to inversion and changes to inverse when we change the complex structure on $X$ (hence the orientation) to the opposite one. So replacing this isomorphism by its inverse results just in replacing $\eta$ by $\eta^{-1}$.}   

\subsection{Analytic Langlands correspondence twisted by $Z$-gerbes on $X$}\label{gerbes}

The setting of the previous subsection has a twisted generalization where we take $G$ simply connected, but instead of ordinary principal $G$-bundles take bundles twisted by $Z$-gerbes on $X$ defined by $c\in H^2(X,Z)=Z$ (this is mentioned in \cite{GW}, Subsection 9.2). Such twisted bundles are defined on an open cover $\lbrace U_i\rbrace$ of $X$ by holomorphic transition functions $g_{ij}: U_i\cap U_j\to G$ such that $g_{ij}g_{jk}g_{ki}=\widetilde c_{ijk}\in Z$ on $U_i\cap U_j\cap U_k$, where $\widetilde c$ is a \v Cech 2-cocycle representing $c$. Let ${\rm Bun}_G^\circ(X)_c$ 
be the variety of regularly stable twisted bundles with class $c$, 
and ${\rm Bun}_G^\circ(X)_{\rm tw}$ be the disjoint union of ${\rm Bun}_G^\circ(X)_c$ 
over all $c\in H^2(X,Z)$. We have a principal $H^1(X,Z)$-bundle 
$$
{\rm Bun}_G^\circ(X)_{\rm tw}\to {\rm Bun}_{G_{\rm ad}}^\circ(X).
$$ 
Thus the Hilbert space 
$\mathcal H=L^2({\rm Bun}_G^\circ(X)_{\rm tw})$ 
carries commuting actions of quantum Hitchin Hamiltonians and Hecke operators.

\begin{conjecture} \label{mainc1} Conjecture \ref{mainc} and the formula for $\chi_\rho$ 
holds in this twisted setting with the group $Z^\vee$ (trivial in our case) 
replaced by $\pi_1(G^\vee)$, the center of $G^\vee_{\rm sc}$. 
\end{conjecture} 

\subsection{Differential equations for the Hecke operators for
  $G=PGL_2$ and $X=\pone$}    \label{Differential equations for Hecke operators}

In this subsection we consider the case of $G=PGL_2$ and
  $X=\pone$ with parabolic structures at finitely many points. We will
  consider the case of $G=PGL_n$ and a smooth projective curve $X$ of
  genus ${\rm g} > 1$ in in the next subsection.

We generalize the results of \cite{EFK3}, Subsection 4.2 to the twisted setting of Subsection \ref{pgl2}. Namely, we show that for $F=\Bbb R$ the Hecke operators satisfy a second-order differential equation (the oper equation), while for $F=\Bbb C$ 
they satisfy a system of two such equations -- holomorphic (the oper equation) and anti-holomorphic (the anti-oper equation, conjugate to the oper equation), which can be used to describe their spectrum. Let us now derive the oper equation. 

We return to the setting of Subsection \ref{pgl2} for $F=\Bbb R$ or $F=\Bbb C$. Consider the Gaudin operators 
$$
G_i:=\sum_{j\ne i}\frac{\Omega_{ij}}{t_i-t_j},\ 0\le i\le m,
$$
where $\Omega=e\otimes f+f\otimes e+\frac{1}{2}h\otimes h$ and 
\begin{equation} \label{sl2for} 
e=\partial_y,\ h=-2y\partial_y+\lambda,\ f=-y^2\partial_y+\lambda y
\end{equation}
(this differs from \cite{EFK3}, (7.7) by the Chevalley involution). 
Thus, setting $\partial_i:=\frac{\partial}{\partial y_i}$, we have (\cite{EFK1}, (7.8)):
$$
G_i=\sum_{j\ne i} \frac{1}{t_i-t_j}\left(-(y_i-y_j)^2\partial_i\partial_j+(y_i-y_j)(\lambda_i\partial_j-\lambda_j\partial_i)+\frac{\lambda_i\lambda_j}{2}\right),
$$
$$
\widehat G_i:=G_i-\sum_{j\ne i}\frac{\lambda_i\lambda_j}{2(t_i-t_j)}=\sum_{j\ne i} \frac{1}{t_i-t_j}\left(-(y_i-y_j)^2\partial_i\partial_j+(y_i-y_j)(\lambda_i\partial_j-\lambda_j\partial_i)\right).
$$
Note that on translation invariant functions of $y_0,...,y_m$ we have 
$$
\sum_{i=0}^m G_i=\sum_{i=0}^m \widehat G_i=0,\ \sum_{i=0}^m t_iG_i=E(E-\lambda-1)+\frac{\lambda^2-\sum_i\lambda_i^2}{4},
\ \sum_{i=0}^m t_i\widehat G_i=E(E-\lambda-1),
$$
where $E:=\sum_{i=0}^m y_i\partial_i$ is the Euler vector field and $\lambda:=\sum_i \lambda_i$
(see e.g. \cite{EFK1}, Section 7). 

The following proposition is a complete analog of \cite{EFK3}, Proposition 4.3. 

\begin{proposition}\label{opereq} (Universal oper equations)
(i) We have 
$$
\left(\partial_x^2-\sum_{i\ge 0}\frac{\lambda_i}{x-t_i}\partial_x\right)\Bbb H_x-\Bbb H_x\sum_{i\ge 0}\frac{\widehat G_i}{x-t_i}=0.
$$

(ii) We have 
$$
\left(\partial_x^2-\sum_{i\ge 0}\frac{\lambda_i(\lambda_i+2)}{4(x-t_i)^2}\right)H_x-
H_x\sum_{i\ge 0}\frac{G_i}{x-t_i}=0. 
$$ 
\end{proposition} 

Here the differential equations hold in the same sense as in \cite{EFK3}, Proposition 4.3.

\begin{proof} (ii) easily follows from (i), so let us prove (i). The proof is parallel to the proof of \cite{EFK3}, Proposition 4.3. We only redo the algebraic part of the proof, as the analytic details are exactly the same. Let $u_i=u_i(s):=y_i-s$ and $\psi,\psi_i,\psi_{ij}$ be the zeroth, first and second derivatives of $\psi$ evaluated at the point $\bold z$ with coordinates $z_i:=\frac{t_i-x}{y_i-s}=\frac{t_i-x}{u_i}$. Thus 
\begin{equation}\label{equ1}
\partial_i\psi_j=\frac{(x-t_i)\psi_{ij}}{u_i^2},\ \sum_{i\ge 0}\partial_i\psi_j=-\partial_s\psi_j
\end{equation} 
Also let 
\begin{equation}\label{muofs}
d\mu(s):=\prod_{i=0}^m \norm{s-y_i}^{\lambda_i}\norm{ds}.
\end{equation}
We have
$$
\partial_x(\Bbb H_x\psi)=\int_{F}\sum_{i\ge 0}\frac{\psi_{i}}{u_i}d\mu(s),\
\partial_x^2(\Bbb H_x\psi)=\int_{F}\sum_{i,j\ge 0}\frac{\psi_{ij}}{u_iu_j}d\mu(s).
$$
Also
$$
\sum_{i\ge 0}\frac{\Bbb H_x\widehat G_i\psi}{x-t_i}=-\int_{F}\sum_{i\ne j}\frac{(\frac{t_i-x}{u_i}-\frac{t_j-x}{u_j})^2\psi_{ij}+(\frac{t_i-x}{u_i}-\frac{t_j-x}{u_j})(\lambda_j\psi_i-\lambda_i\psi_j)}{(x-t_i)(t_i-t_j)}d\mu(s).
$$
Subtracting, we get 
$$
\partial_x^2(\Bbb H_x\psi)-\Bbb H_x\sum_{i\ge 0}\frac{\widehat G_i\psi}{x-t_i}=
$$
$$
\int_{F}\left(\sum_i\frac{\psi_{ii}}{u_i^2}+\sum_{i\ne j}\frac{(\frac{(t_i-x)^2}{u_i^2}+\frac{(t_j-x)^2}{u_j^2})\psi_{ij}+(\frac{t_i-x}{u_i}-\frac{t_j-x}{u_j})(\lambda_j\psi_i-\lambda_i\psi_j)}{(x-t_i)(t_i-t_j)}\right)d\mu(s)=
$$
$$
\int_{F}\left(\sum_{i,j\ge 0}\frac{(t_i-x)\psi_{ij}}{(t_j-x)u_i^2}+\sum_{i\ne j}\frac{(\frac{t_i-x}{u_i}-\frac{t_j-x}{u_j})(\lambda_j\psi_i-\lambda_i\psi_j)}{(x-t_i)(t_i-t_j)}\right)d\mu(s).
$$
Now, using integration by parts, \eqref{equ1} and \eqref{muofs}, we have 
$$
\int_{F}\sum_{i,j\ge 0}\frac{(t_i-x)\psi_{ij}}{(t_j-x)u_i^2}d\mu(s)=
-\int_{F}\sum_{j\ge 0}\frac{1}{x-t_j}\sum_{i\ge 0}\partial_i\psi_{j}d\mu(s)=
$$
$$
\int_{F}\sum_{j\ge 0}\frac{1}{x-t_j}\partial_s \psi_jd\mu(s)=
-\int_{F}\sum_{j\ge 0}\frac{\psi_j}{x-t_j}\sum_{i\ge 0}\frac{\lambda_i}{u_i}d\mu(s).
$$
Thus 
$$
\partial_x^2(\Bbb H_x\psi)-\Bbb H_x\sum_{i\ge 0}\frac{\widehat G_i\psi}{x-t_i}=
$$
$$
-\int_{F}\left(\sum_{j\ge 0}\frac{\psi_j}{x-t_j}\sum_{i\ge 0}\frac{\lambda_i}{u_i}+\frac{1}{2}\sum_{i\ne j}\frac{(\frac{t_i-x}{u_i}-\frac{t_j-x}{u_j})(\lambda_j\psi_i-\lambda_i\psi_j)}{(t_i-x)(t_j-x)}\right)d\mu(s)=
$$
$$
-\int_{F}\left(\sum_{i\ge 0}\frac{\lambda_i}{(t_i-x)u_i}\psi_i+\sum_{i\ne j}\frac{\lambda_j}{(t_j-x)u_i}\psi_i\right)d\mu(s)=
$$
$$
-\int_F\left(\sum_{j\ge 0}\frac{\lambda_j}{t_j-x} \sum_{i\ge 0}\frac{\psi_i}{u_i}\right)d\mu(s)=
\sum_{j\ge 0}\frac{\lambda_j}{x-t_j}\partial_x(\Bbb H_x\psi).
$$
\end{proof} 

Also, similarly to \cite{EFK3}, Proposition 4.11, we have 

\begin{proposition}\label{commu}  
$$
[H_x,G_j]=0.
$$
\end{proposition} 

As shown in \cite{EFK2}, if $F=\Bbb C$ then the Hecke operators also satisfy an anti-holomorphic second-order differential equation (the anti-oper equation) which is the complex conjugate of the oper equation of Proposition \ref{opereq}(ii). Thus, 
introducing the operator-valued oper $\pa_x^2 - S(x)$, where
\begin{equation}    \label{Sz}
  S(x) := \sum_{i=0}^m \frac{\la_i(\la_i+2)}{4(x-t_i)^2} +
  \sum_{i=0}^m \frac{G_i}{x-t_i},
\end{equation}
for $F=\Bbb C$ we obtain 
\begin{equation}    \label{diffeq1}
(\pa_x^2-S(x)) H_x = 0, \qquad (\ol\pa_x^2 - \ol{S(x)})
H_x = 0
\end{equation}
(for $F=\Bbb R$ we only have the first equation). 

We also obtain the following equation for the eigenvalues $\beta_k(x)$ of $H_{x,{\rm full}}$, which is a generalization of Corollary 4.14 of \cite{EFK3}. 

\begin{corollary}\label{spec2} The function $\beta_k(x)$ 
satisfies the differential equation 
\begin{equation}\label{ope}
L(\bmu_k)\beta_k(x)=0,
\end{equation}
where 
$$
L(\bmu_k):=\partial_x^2-\sum_{i\ge 0}\frac{\lambda_i(\lambda_i+2)}{4(x-t_i)^2}-\sum_{i\ge 0}\frac{\mu_{i,k}}{x-t_i}
$$ 
is an $SL_2$-oper, 
with
\begin{equation}\label{muik}
\sum_{i=0}^{m} \mu_{i,k}=0,\ \sum_{i=0}^{m} t_i\mu_{i,k}=\frac{\lambda_{m+1}(\lambda_{m+1}+2)}{4}-\sum_{i=0}^m \frac{\lambda_i(\lambda_i+2)}{4}. 
\end{equation} 

Moreover, if $F=\Bbb C$, then $\beta$ also satisfies the complex conjugate equation 
$\overline {L(\bmu_k)}\beta_k(x)=0$. 

\end{corollary} 

Note that equation \eqref{ope} is Fuchsian at the points $t_j$ with characteristic exponents $\frac{1}{2}\pm \frac{\la_j+1}{2}$, and by \eqref{muik} it is also Fuchsian at $\infty$ with characteristic exponents 
 $-\frac{1}{2}\mp \frac{\la_{m+1}+1}{2}$. In other words, basic solutions behave 
 near $t_j$ as $(x-t_j)^{\frac{1}{2}}$ and $(x-t_j)^{\frac{1}{2}}\log(x-t_j)$ 
 if $\la_j=-1$ and as $(x-t_j)^{\frac{1}{2}\pm \frac{\la_j+1}{2}}$ 
 else, while at $\infty$ they behave as 
 $x^{\frac{1}{2}},x^{\frac{1}{2}}\log x$ if $\la_{m+1}=-1$ and $x^{\frac{1}{2}\pm \frac{\la_{m+1}+1}{2}}$ else.\footnote{We remind that since the oper $L(\bmu)$ is a map from $K^{-\frac{1}{2}}$ to $K^{\frac{3}{2}}$, solutions of the equation $L(\bmu)\beta=0$ are sections of 
 $K^{-\frac{1}{2}}$, but we view them as functions by multiplying by $(dx)^{-\frac{1}{2}}$.}
Thus the monodromy operators of \eqref{ope} at $t_j$, $0\le j\le m+1$ are conjugate 
 to $\begin{pmatrix} -1 & 1\\ 0 & -1\end{pmatrix}$ if $\la_j=-1$ and 
 to $\begin{pmatrix} e^{\pi i\la_j} & 0\\ 0 & e^{-\pi i\la_j}\end{pmatrix}$ 
 else. 
 
Thus for $F=\Bbb C$ the spectral opers have a property that the system $L\beta=0$, $\overline L\beta=0$ has a single-valued solution. So the monodromy of such an oper must preserve 
a nondegenerate hermitian form. Moreover, as the above matrices cannot be conjugated to 
$SU(2)$, this form must be of signature (1,1). Thus spectral opers must have monodromy in 
$SU(1,1)\cong SL(2,\Bbb R)$, i.e. they belong to the (discrete) set $\mathcal B=\mathcal B(\lambda_0,...,\lambda_{m+1})$ 
of real opers of the form of Corollary \ref{spec2}. 
Furthermore, the joint eigenspaces of $H_{x,{\rm full}}$ are 1-dimensional. 
In other words, we have 

\begin{theorem}    \label{mutatis}
Theorem 4.15 of \cite{EFK3} extends mutatis mutandis to the ramified setting with any weights $\lambda_j\in -1+i\Bbb R$. 
\end{theorem} 

Moreover, similarly to \cite{EFK3}, the set of spectral opers conjecturally coincides 
with $\mathcal B$, and this is definitely true at least for 4 and 5 points. The proofs of these facts 
are analogous to the proofs in the untwisted case given in
\cite{EFK3}.

\subsection{Differential equations for the Hecke operators for
  $G=PGL_n$}    \label{diffPGLn}

Let $G=PGL_n$ and $X$ a smooth projective curve of genus ${\rm g} >
1$. Analogues of the universal oper equations of Proposition
\ref{opereq} were obtained in Subsection 1.4 and
Section 4 of \cite{EFK2}. In this subsection we summarize these results.

Recall the component $\Op^0_{SL_n}(X)$ of the space of $SL_n$-opers on
$X$ introduced in Subsection \ref{ALcomp}. For even $n$ it depends on
the choice of an equivalence class of the square root
$K_X^{\frac{1}{2}}$ (a spin structure), which we will denote by
$\ga$.\footnote{In \cite{EFK2}, we denoted this component by
$\Op^\ga_{SL_n}(X)$.} It is an affine space which is a torsor over the
vector space ${\bf Hitch}$ defined by formula \eqref{Hitch}.

Given $\chi \in \Op^0_{SL_n}(X)$, denote by
$(\V_{\omega_1},\nabla_\chi)$ the corresponding holomorphic flat
rank $n$ vector bundle on $X$ whose determinant is identified with the
trivial flat line bundle. The oper Borel reduction gives rise to an
embedding
$$
\kappa_{\omega_1}: K_X^{\frac{n-1}{2}} \hookrightarrow \V_{\omega_1},
$$
and therefore an embedding
$$
\wt\kappa_{\omega_1}: \OO_X \hookrightarrow \V_{\omega_1} \otimes K_X^{-\frac{n-1}{2}}.
$$
Hence we obtain a section
$$
s_{\omega_1} := \wt\kappa_{\omega_1}(1) \in \Gamma(X,\V_{\omega_1}
\otimes K_X^{-\frac{n-1}{2}}).
$$
Likewise, we obtain a section
$$
s_{\omega_{n-1}} \in \Gamma(X,\V_{\omega_{n-1}} \otimes
K_X^{-\frac{n-1}{2}}) = \Gamma(X,\V^*_{\omega_1} \otimes
K_X^{-\frac{n-1}{2}}).
$$

Let $D^\ga_n(X)$ be the affine space of $n$th order differential
operators $P: K_X^{-\frac{n-1}{2}} \to 
K_X^{\frac{n+1}{2}}$, where for even $n$ we use our chosen square root
$K_X^{\frac{1}{2}}$, such that
\begin{enumerate}
\item $\on{symb}(P) \in H^0(X,\OO_X)$ equals $1$;
\item The operator $P - (-1)^n P^*$, where $P^*:
  K_X^{-\frac{n-1}{2}} \to K_X^{\frac{n+1}{2}}$ is the algebraic adjoint operator
  (see \cite{BB}, Sect. 2.4), has order $n-2$.
\end{enumerate}

The following statement is proved in \cite{BD2}, \S 2.8.

\begin{lemma}    \label{Dga}
The assignment
  $$
  \chi \in \on{Op}^0_{SL_n}(X) \quad \mapsto \quad P_\chi \in
  D^\ga_n(X)
  $$
  defines a bijection $\on{Op}^0_{SL_n}(X) \simeq D^\ga_n(X)$
  such that the sections $s_{\omega_1} \in \Gamma(X,\V_{\omega_1}
  \otimes K_X^{-\frac{n-1}{2}})$ and $s_{\omega_{n-1}} \in \Gamma(X,\V_{\omega_{n-1}}
  \otimes K_X^{-\frac{n-1}{2}})$ satisfy
  $$
  P_\chi \cdot s_{\omega_1} = 0, \qquad P_\chi^* \cdot s_{\omega_{n-1}} = 0,
  $$
  where $P_\chi^*$ is the algebraic adjoint of $P_\chi$.
\end{lemma}

Let $\V_{\omega_1}^{\on{univ}}$ be the universal vector bundle over
$\on{Op}^0_{SL_n}(X) \times X$ equipped with a partial connection
$\nabla^{\on{univ}}$ (along $X$) defined by the property
$$
(\V_{\omega_1}^{\on{univ}},\nabla^{\on{univ}})|_{\chi \times X} =
(\V_{\omega_1},\nabla_\chi).
$$
Set $\V_{\omega_1,X}^{\on{univ}} := \pi_*(\V_{\omega_1}^{\on{univ}})$,
where $\pi: \on{Op}^0_{SL_n}(X) \times X \to X$ is the natural
projection and $\pi_*$ is the $\OO$-module direct image. Using
the connection $\nabla^{\on{univ}}$, we obtain a left $\D_X$-module
$\V_{\omega_1,X}^{\on{univ}}$.

Now let $D_{PGL_n,\alpha}$ be the algebra of global holomorphic
differential operators acting on the component $\Bun_{PGL_n,\alpha}$
of $\Bun_{PGL_n}$. According to the results of \cite{BD}, these
algebras are isomorphic to each other and
$$
D_{PGL_n,\alpha} \simeq \on{Fun} \on{Op}^0_{SL_n}(X).
$$
From now on, we will use the notation $D_{PGL_n}$ for $D_{PGL_n,\alpha}$.

Thus, $D_{PGL_n}$ naturally acts on $\V_{\la,X}^{\on{univ}}$, and this
action commutes with the action of $\D_X$. We obtain the following
result.

\begin{lemma}    \label{nsigma}
  There is a unique $n$-th order differential operator
\begin{equation}    \label{sigman}
\sigma: K_X^{-\frac{n-1}{2}} \to D_{PGL_n} \otimes K_X^{\frac{n+1}{2}}
\end{equation}
satisfying the following property: for any $\chi \in
\on{Op}^0_{SL_n}(X) = \on{Spec} D_{PGL_n}$, applying the
corresponding homomorphism $D_{PGL_n} \to \C$ we obtain $P_\chi$.
\end{lemma}

Let $\on{Op}^0_{SL_n}(X)_{\R}$ be the set of real $SL_n$-opers in
$\on{Op}^0_{SL_n}(X)$. If $\chi \in \on{Op}^0_{SL_n}(X)_{\R}$, then
the monodromy representations associated to $\chi$ and $\ol\chi$ are
isomorphic. Therefore, $(\V_{\omega_1},\nabla_\chi)$ and
$(\ol{\V}_{\omega_1},\ol\nabla_\chi)$ are isomorphic as $C^\infty$
flat vector bundles on $X$. Hence we obtain a non-degenerate pairing
$$
h_{\chi,\omega_1}(\cdot,\cdot): (\V_{\omega_1},\nabla_{\chi}) \otimes
  (\ol{\V}_{\omega_{n-1}},\ol\nabla_{\chi}) \to ({\mc
  C}^\infty_X,d)
$$
of $C^\infty$ flat vector bundles on $X$. The flat vector bundle
$(\V_{\omega_1},\nabla_{\chi})$ is known to be irreducible if ${\rm
  g}>1$ (see \cite{BD}, Sect. 3.1.5(iii)); therefore this pairing is
unique up to a scalar.

The following results were proved in \cite{EFK2}, Theorem
1.18 and Corollary 1.16.

\begin{theorem}    \label{mainthm}
  The Hecke operator $H_{\omega_1}$, viewed as an operator-valued
  section of $\Omega_X^{-\frac{n-1}{2}} = K_X^{-\frac{n-1}{2}} \otimes
  \ol{K}_X^{-\frac{n-1}{2}}$, satisfies the system of differential equations
  \begin{equation}    \label{eqHom1}
    \sigma \cdot H_{\omega_1} = 0, \qquad \ol\sigma \cdot
    H_{\omega_1} = 0.
  \end{equation}
\end{theorem}

\begin{proposition}    \label{unique}
    $h_{\chi,\omega_1}(s_{\omega_1}, \ol{s_{\omega_{n-1}}})$ is a
  unique, up to a scalar, section $\Phi$ of $\Omega_X^{-\frac{n-1}{2}}$
  which is a solution of the system of differential equations
\begin{equation}    \label{diffeqs}
    P_\chi \cdot \Phi = 0, \qquad \ol{P^*_\chi} \cdot
    \Phi = 0
\end{equation}
\end{proposition}

These two results have the following corollary describing the
eigenvalues of the Hecke operator $H_{\omega_1}$ obtained in
\cite{EFK2}, Corollary 1.19.

\begin{corollary}    \label{eigH1}
Each of the eigenvalues $\beta_{\omega_1}(x,\ol{x})$ of the Hecke operator
$H_{\omega_1}$ on ${\mc H}$ corresponding to a real oper $\chi \in
\on{Op}^0_{SL_n}(X)_{\R}$ (see Conjecture \ref{mainc}) is equal to a
scalar multiple of
$h_{\chi,\omega_1}(s_{\omega_1},\ol{s_{\omega_{n-1}}})$.
\end{corollary}

\subsection{General case}    \label{gencase}

In \cite{EFK2}, Section 5, we described a conjectural analogue of this
picture for general $G, X$, and $\lambda\in \Lambda^+$. In the general
case, the analogues of the differential equations \eqref{eqHom1}
satisfied by the Hecke operators are more complicated (see Subsection
\ref{genwt} below). However, there are several cases in which these
equations can be presented in a simple form similar to equations
\eqref{eqHom1}. Those cases will be discussed in Subsection
\ref{prin}.

We start by introducing some notation following \cite{EFK2}, Section
5. For $\chi \in \on{Op}^0_{\LG}(X)$ and $\la \in \Lambda^+$, we have
the flat holomorphic vector bundle $(\V_\la,\nabla_{\chi,\la})$ on $X$
associated to the irreducible representation $V_\la$ of $\LG$ with
highest weight $\la$ (according to \cite{BD2}, \S 3, the corresponding
vector bundles are isomorphic to each other for all $\chi \in
\on{Op}^0_{\LG}(X)$, which justifies the notation $\V_{\la}$; see
Theorem \ref{bundle} below). The oper Borel reduction gives rise to an
embedding
\begin{equation}    \label{kala}
  \kappa_\la: K_X^{\langle \la,\rho \rangle} \hookrightarrow
  \V_\la
\end{equation}
(if $n$ is a half-integer, we take the power
of $K_X^{\frac{1}{2}}$ in our chosen isomorphism class $\gamma$).

For $n \in \frac{1}{2}\Z$, denote by $\D_{X,n}$ the sheaf of
differential operators acting on the line bundle $K_X^{n}$ on $X$. We
have
$$
\D_{X,n} \simeq K_X^n
\underset{\OO_X}\otimes \D_X \underset{\OO_X}\otimes K_X^{-n}
$$

For $\la \in \Lambda_+$, set
\begin{equation}    \label{dla}
d(\la) := 2\langle \la,\rho \rangle.
\end{equation}
The $\D_X$-module structure on $\V_\la$ defined by the oper connection
$\nabla_{\chi,\la}$ gives rise to a $\D_{X,-\frac{d(\la)}{2}}$-module
structure on the $\OO_X$-module
$$
\mcV^K_\la := K_X^{-\frac{d(\la)}{2}} \underset{\OO_X}\otimes
\mcV_\la.
$$
We denote this $\D_{X,-\frac{d(\la)}{2}}$-module by $\V^K_{\chi,\la}$.

The map \eqref{kala} gives rise to a map
$$
  \wt\kappa_\la: \OO_X \hookrightarrow K_X^{-\frac{d(\la)}{2}} \otimes
  \V_\la
$$
and a non-zero section
\begin{equation}    \label{sla}
  s_\la := \wt\kappa_\la(1) \in \Gamma(X,K_X^{-\frac{d(\la)}{2}} \otimes
  \V_\la).
\end{equation}
Following \cite{EFK2}, Section 5.2, we denote by
$I_{\la,\chi}$ the left annihilating ideal of $s_\la$ in the sheaf
$\D_{X,-\frac{d(\la)}{2}}$.

Now suppose that $\chi \in \on{Op}^0_{\LG}(X)_{\R}$. Then we have an
isomorphism of $C^\infty$ flat bundles 
$$
  (\V_\la,\nabla_{\chi,\la}) \simeq (\ol{\V}_\la,\ol\nabla_{\chi,\la})
$$
and hence a pairing
  $$
  h_{\chi,\la}(\cdot,\cdot): (\V_\la,\nabla_{\chi,\la}) \otimes
  (\ol{\V}_{-w_0(\la)},\ol\nabla_{\chi,-w_0(\la)}) \to ({\mc
    C}^\infty_X,d)
  $$
  as $V^*_\la \simeq V_{-w_0(\la)}$. Since $\langle
  -w_0(\la),\rho \rangle = \langle \la,\rho \rangle =
  \frac{d(\la)}{2}$, we have
  $$
  \ol{s_{-w_0(\la)}} \in \Gamma(X,\ol{K}_X^{-\frac{d(\la)}{2}} \otimes
  \ol{\V}_{-w_0(\la)}).
  $$

We recall the results of \cite{BD2}, \S 3, on the
structure of $(\V_\la,\nabla_{\chi,\la})$ and $\kappa_\la$. Fix a
principal $\sw_2$ subalgebra
  $$
  \sw_2 = \on{span}\{ e,h,f \} \subset \lg
  $$
  such that $\on{span}\{ e,h \}$ is
  in the Borel subalgebra $\lb \subset \lg$ used in the definition of
  $\LG$-opers.

Denote by $V_m$ the irreducible $(m+1)$-dimensional representation of
$\sw_2$. The irreducible representation $V_\la$ of $\LG$ decomposes
into a direct sum of irreducible representations of the principal
$\sw_2$ subalgebra:
  \begin{equation}    \label{decsl2}
    V_\la \simeq V_{d(\la)} \oplus \left( \bigoplus_{m <
    d(\la)} V_m^{\oplus c_{m,\la}} \right), \qquad
    c_{m,\la} \in \Z_{\geq 0},
  \end{equation}
  where $d(\la)$ is given by formula \eqref{dla}.
  
Recall the rank two vector bundle $\mcE_{SL_2}$ which is a non-trivial
extension \eqref{nontr} (as before, here we take $K_X^{\frac{1}{2}}$ in the
isomorphism class $\gamma$ that we have chosen). Define
  the rank $(m+1)$ vector bundle $\mcE_m := \on{Sym}^m(\mcE_{SL_2})$ on $X$.
  By construction, $\mcE_m$ is equipped with a filtration $\{
\mcE^{\leq i}_m \}_{i=0,\ldots,m}$ such that
$$
\mcE^{\leq i}_m/\mcE^{\leq (i-1)}_m \simeq K_X^{\frac{m}{2}-i}.
$$
Fix isomorphisms
  $$
  \jmath^i_m: \mcE^{\leq i}_m/\mcE^{\leq (i-1)}_m \to (\mcE^{\leq
  (i+1)}_m/\mcE^{\leq i}_m) \otimes K_X, \qquad i=0,\ldots,m.
  $$
Let $B_m \subset \on{End}(\mcE_m)$ be the subbundle of endomorphisms
preserving this filtration, and $p$ the canonical section of
$(\on{End}(\mcE_m)/B_m) \otimes K_X$ such that $p(\mcE^{\leq i}_m) \subset
\mcE^{\leq (i+1)}_m \otimes K_X$ and it induces the isomorphisms
$\jmath^i_m$ on successive quotients.

  The following results are due to Beilinson and Drinfeld \cite{BD,BD2}.

  \begin{theorem}    \label{bundle}
(1) For any $\chi \in \on{Op}^0_{SL_{m+1}}(X)$, we have
    $\V_{\chi,\omega_1} \simeq \mcE_m$, the oper Borel reduction corresponds
    to $B_m$, and $\nabla_{\chi,\omega_1} = d + p$ mod $B_m$.

(2) For any $\chi \in \on{Op}^0_{\LG}(X)$,
  \begin{equation}    \label{decomV}
    \mcV_\la \simeq \mcE_{d(\la)} \oplus \left( \bigoplus_{m <
      d(\la)} \mcE_m^{\oplus c_{m,\la}} \right),
  \end{equation}
  where the numbers $c_{m,\la}$ are defined by
  \eqref{decsl2}, and
  \begin{equation}    \label{nablala}
   \nabla_{\chi,\la} = d + p \quad \on{mod} \quad B
  \end{equation}
  where $B$ is the direct sum of $B_{d(\la)}$ and all $B_m$'s
  corresponding to the summands of \eqref{decomV}.
  \end{theorem}

\subsection{Differential equations for Hecke operators corresponding
  to principal weights}    \label{prin}

In the Subsection \ref{diffPGLn} we used the interpretation of
$SL_n$-opers as scalar differential operators of order $n$ (see Lemma
\ref{Dga}).  As we show in this subsection, an analogous
interpretation is possible if $G$ is a connected simple algebraic
group such that $G^\vee$ has an irreducible representation $V_\la$
with highest weight $\la$ that remains irreducible under a principal
$\sw_2$ subalgebra of $\g^\vee$. We will call such weights {\em
  principal}. According to the results of \cite{SS}, which go back to
E. Dynkin in characteristic 0 (see Theorem \ref{list} below),
principal weights are $\la=\omega_1$ and $\omega_\ell$ for $\g^\vee$ of
type $A_\ell$, and $\la=\omega_1$ for $\g^\vee$ of types $B_\ell$, $C_\ell$
and $G_2$. For $A_\ell, B_\ell$, and $C_\ell$, the corresponding
scalar differential operators were described in \cite{DrS,BD2} (see
also \cite{FG} and \cite{LM}, where such operators were discussed in
special cases).

We have $V_\la \simeq V_{d(\la)}$ in the decomposition \eqref{decsl2}
(i.e. there are no lower terms) if and only if $\la$ is a principal
weight of $\LG$. In this case, we have the following corollary of
Theorem \ref{bundle}.
  
\begin{corollary}    \label{irr}
Suppose that $\la$ is a principal weight of $\LG$.

(1) $(\mcV_\la,\nabla_{\chi,\la})$ together with its
  oper Borel reduction is an $SL_{d(\la)+1}$-oper.

(2) If ${\rm g}>1$, then the
  flat vector bundle $(\mcV_\la,\nabla_{\chi,\la})$ is
  irreducible.
\end{corollary}

\begin{proof}
If $\la$ is principal, then $\mcV_\la \simeq \mcE_{d(\la)}$. Part (1)
readily follows from Theorem \ref{bundle}. By \cite{BD}, Proposition
3.1.5(iii), if ${\rm g}>1$, then the flat $\LG$-bundle underlying any
$\LG$-oper does not admit a reduction to a nontrivial parabolic
subgroup of $\LG$. This proves part (2).
\end{proof}

\begin{remark}    \label{nonmin}
If $\la$ is not a principal weight, so $V_\la$ is reducible as a
representation of a principal $\sw_2$ subalgebra of $\lg$ (see
\eqref{decsl2}), then there exists $\chi \in \Op^0_{\LG}(X)$ such that
the flat vector bundle $(\mcV_\la,\nabla_{\chi,\la})$ is
reducible. For example, this is so if the $\LG$-oper $\chi$ is in
the image of the canonical embedding $\Op_{PGL_2}(X) \hookrightarrow
\Op^0_{\LG}(X)$ constructed in \cite{BD2}, \S 3. Indeed, for such
$\chi$, the oper connection $\nabla_{\chi,\la}$ preserves the
decomposition \eqref{decomV}.

However, in the next subsection we will show that for generic $\chi
\in \Op^0_{\LG}(X)$ the flat vector bundle
$(\mcV_\la,\nabla_{\chi,\la})$ is irreducible for any $\la
\in \Lambda^+$, if ${\rm g}>1$ (see Corollary \ref{irrgen},(2)).\qed
\end{remark}

Recall from Subsection \ref{gencase} that for any $\la \in \Lambda^+$
we have a canonical section $s_\la \in \Gamma(X,\mcV^K_\la)$ defined
by the oper Borel reduction and the left annihilating ideal
$I_{\la,\chi}$ of $s_\la$ in the sheaf $\D_{X,-\frac{d(\la)}{2}}$. Let us
specialize to the case of a principal weight $\la$. Corollary
\ref{irr} then implies the following.

\begin{lemma}    \label{minu}
Suppose that $\la$ is a principal weight of a group $G$. Then
$\V^K_{\chi,\la}$ is an irreducible $\D_{X,-\frac{d(\la)}{2}}$-module,
and we have an exact sequence of left $\D_{X,-\frac{d(\la)}{2}}$-modules
$$
0 \to I_{\la,\chi} \to \D_{X,-\frac{d(\la)}{2}} \to
\mcV^K_{\chi,\la} \to 0
$$
\end{lemma}

Recall from Subsection \ref{gencase} that $s_\la$ and
$\ol{s}_{-w_0(\la)}$ are sections of $K^{-\frac{d(\la)}{2}}$ and
$\ol{K}^{-\frac{d(\la)}{2}}$, respectively.

\begin{proposition}    \label{monG}
Let $\la$ is a principal weight of a group $G$ and $\chi \in
\on{Op}_{\LG}(X)_{\R}$. Then $h_{\chi,\la}(s_\la, \ol{s_{-w_0(\la)}})$
is a unique, up to a scalar, non-zero global $C^\infty$ section of
$\Omega_X^{-\frac{d(\la)}{2}}$ over $X$ annihilated by the
ideals $I_{\la,\chi}$ and $\ol{I_{-w_0(\la),\chi}}$.
\end{proposition}

\begin{proof}
  Clearly, $h_{\chi,\la}(s_\la, \ol{s_{-w_0(\la)}})$ satisfies the
  conditions of the lemma. Conversely, suppose that $\phi_\la(\chi)$
  is a non-zero section of $\Omega_X^{-\frac{d(\la)}{2}}$
  annihilated by the ideals $I_{\la,\chi}$ and
  $\ol{I_{-w_0(\la),\chi}}$. By the definition of these ideals, we then
  have a non-zero homomorphism of $\D_{X,-\frac{d(\la)}{2}} \otimes
  \ol\D_{X,-\frac{d(\la)}{2}}$-modules
  $$
  \alpha_{\la,\chi}: \mcV^K_{\chi,\la} \otimes
  \ol\mcV^K_{\chi,-w_0(\la)} \to \Omega_X^{-\frac{d(\la)}{2}}
  $$
  sending $s_\la \otimes \ol{s_{-w_0(\la)}}$ to
  $\phi_\la(\chi)$. Equivalently, we have a non-zero homomorphism of
flat $C^\infty$ vector bundles
  $$
  (\mcV_\la,\nabla_{\chi,\la}) \otimes
  (\ol\mcV_{-w_0(\la)},\ol\nabla_{\chi,-w_0(\la)}) \to ({\mc C}^\infty,d)
  $$
  By Corollary \ref{irr}, the flat vector bundle
  $(\mcV_\la,\nabla_{\chi,\la})$ is
  irreducible. Therefore the vector space of such homomorphisms is
  one-dimensional. Hence $\phi_\la$ is equal to a scalar multiple of
  $h_{\chi,\la}(s_\la, \ol{s_{-w_0(\la)}})$.
\end{proof}

 For $G=PGL_n$, $\la = \omega_1$, let
$$
I'_{\omega_1,\chi} := K_X^n \otimes I_{\omega_1,\chi}.
$$
This is a left submodule of the $({\mc
  D}_{X,\frac{n+1}{2}},{\mc D}_{X,\frac{-n+1}{2}})$-bimodule of differential
operators acting from $K_X^{\frac{-n+1}{2}}$ to $K_X^{\frac{n+1}{2}}$. The submodule
$I'_{\omega_1,\chi}$ is generated by a globally defined
$n$th order differential operator $P_\chi$ on $X$ associated to $\chi$
by Lemma \ref{Dga}, that is
$$
I'_{\omega_1,\chi} = {\mc D}_{X,\frac{n+1}{2}} \cdot P_\chi.
$$
Therefore in this case a section
annihilated by the ideal $I_{\la,\chi}$ is the same as a section
satisfying the $n$th order differential equation \eqref{diffeqs}.

A similar statement is true for a general principal weight $\la$ of a
group $G$. For $\chi \in \Op^0_{\LG}(X)$ let
$$
I'_{\la,\chi} := K_X^{d(\la)+1} \otimes I_{\la,\chi},
$$ which is a left submodule of the $({\mc D}_{X,\frac{1+d(\la)}{2}},{\mc
  D}_{X,-\frac{d(\la}{2}})$-bimodule of differential operators acting from
$K_X^{-\frac{d(\la)}{2}}$ to $K_X^{\frac{1+d(\la)}{2}}$. As shown in the proof of
Corollary \ref{irr}, the flat vector bundle
$(\mcV_\la,\nabla_{\chi,\la})$ has the structure of an
$SL_{d(\la)+1}$-oper, which we will denote by $\wt\chi_\la$. This
implies that the ideal $I'_{\la,\chi}$ is generated by the
corresponding differential operator $P_{\wt\chi_\la}$ of order
$d(\la)+1$ on $X$ acting from $K_X^{-\frac{d(\la)}{2}}$ to $K_X^{1+\frac{d(\la)}{2}}$.
Therefore, we again obtain that a section annihilated by the ideal
$I_{\la,\chi}$ is the same as a section satisfying a differential
equation of the form \eqref{diffeqs} of order $d(\la)+1$. Thus,
Proposition \ref{monG} has the following equivalent reformulation
which is an analogue of Corollary \ref{unique} for a general principal
weight.

\begin{corollary}
For a principal weight $\la$ of a group $G$, $h_{\chi,\la}(s_\la,
\ol{s_{-w_0(\la)}})$ is a unique, up to a scalar, non-zero section
$\Phi$ of $\Omega_X^{-\frac{d(\la)}{2}}$ which is a solution of the system of
differential equations
\begin{equation}    \label{diffeqs1}
    P_{\wt\chi_\la} \cdot \Phi = 0, \qquad \ol{P^*_{\wt\chi_\la}} \cdot
    \Phi = 0.
\end{equation}
\end{corollary}

\begin{remark}
(1) For $\LG$ of type $B_\ell$ (resp. $C_\ell$) the map $\chi
\mapsto P_{\wt\chi_\la}$ sets up a one-to-one correspondence between
$\Op^0_{\LG}(X)$ and the space of self-adjoint (resp. anti-self
adjoint) scalar differential operators, respectively, of order
$d(\la)+1$ acting from $K_X^{-\frac{d(\la)}{2}}$ to $K_X^{1+\frac{d(\la)}{2}}$ and
having symbol 1. This is proved in \cite{BD2}, \S 3 (following
\cite{DrS}) together with Lemma \ref{Dga}, which is the analogous
statement for $\LG$ of type $A_\ell$.

(2) If $\la$ is not principal, then we do not expect that the ideal
$I_{\la,\chi}$, or its twist such as $I'_{\la,\chi}$ is generated by a
single globally defined differential operator. Hence in \cite{EFK2},
Section 5, we formulated everything in terms of the ideal
$I_{\la,\chi}$ itself. See also Subsection \ref{genwt} below.
\end{remark}

Now we are going to formulate a conjectural analogue of Theorem
\ref{mainthm} for principal weights (Conjecture \ref{eqHG}).

For any $\la \in \Lambda^+$, let $\V_{\la}^{\on{univ}}$ be the
universal vector bundle over $\on{Op}^0_{\LG}(X) \times X$ with a
partial connection $\nabla^{\on{univ}}$ along $X$, such that
$$
(\V_{\la}^{\on{univ}},\nabla^{\on{univ}})|_{\chi \times X} =
(\V_{\la},\nabla_{\chi,\la}), \qquad \chi \in \on{Op}^0_{\LG}(X).
$$
Let $\pi:
\on{Op}^0_{\LG}(X) \times X \to X$ be the projection and set
$$
\V_{X,\la}^{\on{univ}} := \pi_*(\V_{\la}^{\on{univ}}), \qquad
\V_{X,\la}^{K,\on{univ}} := K_X^{-\frac{d(\la)}{2}} \otimes \V_{X,\la}^{\on{univ}}.
$$
Then $\V_{X,\la}^{K,\on{univ}}$ is naturally a
$\D_{X,-\frac{d(\la)}{2}}$-module on $X$, equipped with a commuting action
of $\on{Fun} \on{Op}^0_{\LG}(X) \simeq D_G$.

Moreover, the oper Borel reduction gives rise to an embedding
$$
\kappa_{\la}^{\on{univ}}: K_X^{-\frac{d(\la)}{2}} \hookrightarrow
\V_{\la,X}^{\on{univ}}
$$
and hence a canonical section
$$
s_{\la}^{\on{univ}} \in \Gamma(X,\V_{X,\la}^{K,\on{univ}}).
$$
Consider the cyclic $D_G \otimes \D_{X,-\frac{d(\la)}{2}}$-module $(D_G
\otimes \D_{X,-\frac{d(\la)}{2}}) \cdot s_{\la}^{\on{univ}}$ generated by
$s_{\la}^{\on{univ}}$.

The next lemma, which follows from Corollary
  \ref{irr},(2), is an analogue of Lemma \ref{nsigma} for a general
  principal weight.

\begin{lemma}    \label{Ds}
  For a principal weight $\la$ of a group $G$, there is an isomorphism
  \begin{equation}
  (D_G \otimes \D_{X,-\frac{d(\la)}{2}}) \cdot
  s_{\la}^{\on{univ}}  \simeq \V_{X,\la}^{K,\on{univ}}
  \end{equation}
  of $D_G \otimes \D_{X,-\frac{d(\la)}{2}}$-modules.
\end{lemma}

Recall that
the Hecke operator $H_\la$ is an operator-valued section of
$\Omega_X^{-\frac{d(\la)}{2}}$. Hence we can apply to it the
sheaf $\D_{X,-\frac{d(\la)}{2}}$ as well as the algebra $D_G$. The two actions
commute, and they
generate a $\D_{X,-\frac{d(\la)}{2}}$-module inside the sheaf
of operator-valued $C^\infty$ sections of $\Omega_X^{-\frac{d(\la)}{2}}$ on
$X$. Let us denote this
$\D_{X,-\frac{d(\la)}{2}}$-module by $\langle H_\la \rangle$.
Likewise, we can apply to $H_\la$ the sheaf $\ol{\D}_{X,-\frac{d(\la)}{2}}$
and the algebra $\ol{D}_G$. Denote the resulting
$\ol{\D}_{X,-\frac{d(\la)}{2}}$-module by $\ol{\langle H_\la
  \rangle}$.

Recall that for any $\chi \in \on{Op}^0_{\LG}$ we then have the
corresponding differential operator $P_{\wt\chi_\la}$. These operators
give rise to an analogue of the operator \eqref{nsigma}; namely,
\begin{equation}    \label{sigman1}
\sigma: K_X^{-\frac{d(\la)}{2}} \to D_G \otimes K_X^{1+\frac{d(\la)}{2}}
  \end{equation}
satisfying the property that for any $\chi \in
\on{Op}^0_{\LG}(X) = \on{Spec} D_G$, applying the
corresponding homomorphism $D_G \to \C$ we obtain
$P_{\wt\chi_\la}$.

The following statement is an analogue of Theorem \ref{mainthm} for a
general principal weight.

\begin{conjecture}    \label{eqHG}
  For a principal weight $\la$ of a group $G$, the Hecke operator
  $H_\la$ satisfies the system of equations
\begin{equation}    \label{eqHomG}
    \sigma \cdot H_\la = 0, \qquad \ol\sigma \cdot
    H_\la = 0.
  \end{equation}

Equivalently, there are isomorphisms
  \begin{equation}
    \langle H_\la
  \rangle \simeq \V_{X,\la}^{K,\on{univ}}, \qquad \ol{ \langle H_\la
      \rangle } \simeq \ol{\V_{X,\la}^{K,\on{univ}}}
  \end{equation}
  of $D_G \otimes \D_{X,-\frac{d(\la)}{2}}$-modules
  (resp. $\ol{D}_G \otimes \ol{\D}_{X,-\frac{d(\la)}{2}}$-modules).
\end{conjecture}

The following conjecture
follows from Conjecture \ref{eqHG} and Proposition \ref{monG} in the
same way as Corollary \ref{eigH1} follows from Theorem \ref{mainthm}
and Corollary \ref{unique} in the case $G=PGL_n, \la=\omega_1$.

\begin{conjecture}    \label{eigHG}
Let $\la$ be a principal weight of a group $G$. Each of the
eigenvalues $\beta_\la(x,\ol{x})$ of the Hecke operator $H_\la$ on
${\mc H}$ corresponding to a real oper $\chi \in
\on{Op}^0_{G^\vee}(X)_{\R}$ (see Conjecture \ref{mainc}) is equal to a
scalar multiple of $h_{\chi}(s_\la,\ol{s_{-w_0(\la)}})$.
\end{conjecture}

This is a special case (corresponding to the principal weights) of
Conjecture 5.1 of \cite{EFK2} which we mentioned in Conjecture
\ref{mainc},(ii) above. In Subsection \ref{genwt} we will discuss the
general case (see Conjecture \ref{51}).

\subsection{Monodromy of opers}    \label{irred}

Let $G$ be a connected reductive algebraic group over
  $\C$ and $X$ a smooth projective curve over $\C$ of genus ${\rm g} >
  1$.

Let ${\rm LocSys}_{G^\vee}(X)$ be the stack of Betti $G^\vee$-local
systems on $X$, and let ${\rm Conn}_{G^\vee}(X)$ be the stack of
$G^\vee$-connections
(i.e., de Rham $G^\vee$-local systems) on $X$. It contains the stack
of $G^\vee$-opers which, according to \cite{BD,BD2}, is the quotient
of the variety of $G^\vee$-opers (which is a union of affine spaces
which are torsors over ${\bf Hitch}$) by the trivial action of the center
$Z^\vee$ of $G^\vee$. Slightly abusing notation, in this section we
will denote this stack by $\Op_{G^\vee}(X)$. 

We have the analytic monodromy map $M: {\rm Conn}_{G^\vee}(X)\to {\rm LocSys}_{G^\vee}(X)$, which is an analytic isomorphism. Let $Z\subset {\rm Conn}_{G^\vee}(X)$ be the Zariski closed substack of connections whose differential Galois group 
is a proper subgroup of $G^\vee$. 

Despite the map $M$ not being algebraic, we have 

\begin{lemma} $M(Z)$ is a Zariski closed substack of 
${\rm LocSys}_{G^\vee}(X)$. 
\end{lemma} 

\begin{proof} $M(Z)$ can be defined algebraically in the Betti realization --
it is the substack of local systems whose structure group 
is a proper subgroup of $G^\vee$. This implies the statement. 
\end{proof}  

\begin{theorem}\label{densee} $Z$ does not contain ${\rm Op}_{G^\vee}(X)$. 
In other words, there exists a $G^\vee$-oper $\chi$ whose monodromy is
Zariski dense in $G^\vee$ (equivalently, whose differential Galois
group is the entire $G^\vee$). 
\end{theorem} 

\begin{proof} 
It is sufficient to prove the theorem in the case that
  $G^\vee$ is simple of adjoint type. In this case $\Op_{G^\vee}(X)$ is an
  affine space which we will view as a subvariety of ${\rm
    LocSys}_{G^\vee}(X)$. Since the automorphism group of the flat
  $G^\vee$-bundle underlying any $G^\vee$-oper is trivial if $G^\vee$
  is of adjoint type (see \cite{BD2}, \S 1.3), it follows that any
  $\chi \in \Op_{G^\vee}(X)$ has a Zariski open neighborhood ${\rm
    LocSys}_{G^\vee}(X)_\chi$ in ${\rm
    LocSys}_{G^\vee}(X)$ which is a smooth subvariety. The set ${\rm
    LocSys}_{G^\vee}(X)_\chi(\C)$ of its $\C$-points is a smooth
complex manifold.

Suppose that $\chi$ is a real $G^\vee$-oper. Then the set ${\rm
  LocSys}_{G^\vee}(X)_\chi(\R)$ of $\R$-points of the variety ${\rm
  LocSys}_{G^\vee}(X)_\chi$ is a smooth real submanifold of ${\rm
  LocSys}_{G^\vee}(X)_\chi(\C)$. Hence $\chi$ is a point of
intersection of two smooth manifolds, $\Op_{G^\vee}(X)(\C)$ and ${\rm
  LocSys}_{G^\vee}(X)_\chi(\R)$, in ${\rm
  LocSys}_{G^\vee}(X)_\chi(\C)$ (all viewed as real manifolds).

Following \cite{Wa}, call a real $G^\vee$-oper coming from a principal
$PGL_2$ subgroup of $G^\vee$ {\em permissible}. According to
Theorem A of \cite{Wa}, if $\chi$ is permissible, then the above two
subvarieties are transversal at $\chi$ (note that this implies that
permissible opers are discrete in $\Op_{G^\vee}(X)$). We will prove
the Zariski density of the image of the monodromy representation of a
generic $G^\vee$-oper by doing linear analysis around any given
permissible $G^\vee$-oper (for instance, we can take the image of the
real $PGL_2$-oper uniformizing $X$).

We start with an obvious lemma from linear algebra.

\begin{lemma}\label{auxle} Let $V$ be a finite dimensional real vector
  space of dimension $2d$ and $U$ a complex subspace of $V_{\Bbb C}$
  of dimension d transversal to $V$ (as a real vector space), i.e.,
  such that $U+V=V_{\Bbb C}$. Also, let $W$ be a subspace of $V$. If
  $U$ is contained in $W_{\Bbb C}$ then $W=V$.
\end{lemma}

\begin{proof} Since $U\oplus V=V_{\Bbb C}$, we have $W_{\Bbb C}\oplus
  V=V_{\Bbb C}$, hence $W=V$.
\end{proof} 

According to Remark \ref{monim}, for a permissible $G^\vee$-oper the
Zariski closure of the corresponding monodromy representation is
equal to a principal $PGL_2$ subgroup of $G^\vee$. Therefore, the
Zariski closure of the monodromy group of a sufficiently generic
$G^\vee$-oper $\chi$ has to be a subgroup $K$ of $G^\vee$
containing its principal $PGL_2$ subgroup. Hence $K$ is a semisimple
group (it is the same for all sufficiently generic $\chi$ up to
conjugacy). Fix a permissible $G^\vee$-oper $\psi$. Note that $\psi$
is a smooth point of both ${\rm LocSys}_{G^\vee}(X)$ and ${\rm
  LocSys}_{K}(X)$, since the centralizer of the principal $PGL_2$ in
$G^\vee$ is trivial (this also follows from the fact \cite{BD2} we
mentioned above that the group of automorphisms of any $G^\vee$-oper
is trivial if $G^\vee$ is of adjoint type). Consider the tangent
spaces $V:=T_\psi{\rm LocSys}_{G^\vee}(X)(\Bbb R)$, $U:=T_\psi {\rm
  Op}_{G^\vee}(X)(\Bbb C) = {\bf Hitch}$, and $W:=T_\psi{\rm
  LocSys}_{K}(X)(\Bbb R)$. Then by assumption, $U$ is contained in
$W_{\Bbb C}$, and by Theorem A of \cite{Wa}, $U$ and $V$ (which have
the same real dimension) are transversal in $V_{\Bbb C}$. Thus by
Lemma \ref{auxle} $W=V$, and hence $K=G^\vee$. This completes the
proof of the theorem.
\end{proof} 

Theorem \ref{densee} implies
 
\begin{corollary}    \label{irrgen}
  (1) Opers $\chi\in {\rm Op}_{G^\vee}(X)$ whose differential Galois
  group is $G^\vee$ (equivalently, such that the Zariski closure of
  their monodromy is $G^\vee$) form a dense Zariski open subset
  $\mathcal U\subset {\rm Op}_{G^\vee}(X)$.

  (2) For a generic $\chi \in {\rm Op}_{G^\vee}(X)$
    (namely, for $\chi\in \mathcal U$ from part (1)) the associated
    flat vector bundle $(\mathcal V_\lambda,\nabla_{\chi,\la})$ on $X$
    corresponding to an arbitrary dominant integral weight $\lambda$
    of $G^\vee$ is irreducible.
\end{corollary} 

Suppose that $G^\vee$ is a simple algebraic group of adjoint type. Let
$\chi$ be a $G^\vee$-oper and denote by $M_\chi$ the Zariski closure
of the monodromy of $\chi$. By Corollary \ref{irrgen}, we have $M_\chi
= G^\vee$ for generic $\chi$. But this is not the case for
some $G^\vee$-opers. For example, for the permissible $G^\vee$-opers
$\chi$ discussed above (coming from a principal $PGL_2$ subgroup of
$G^\vee$) we have $M_\chi = PGL_2$. The proof of the following
result was communicated to us by D. Arinkin \cite{Ar}.

\begin{theorem}    \label{zar}
Suppose that the Zariski closure $M_\chi$ of the monodromy group of a
$G^\vee$-oper $\chi$ on a curve of genus ${\rm g} > 1$ is a proper
subgroup of $G^\vee$. Then $M_\chi \subset G'$, a proper simple
subgroup of $G^\vee$ that contains a principal $PGL_2$ subgroup of
$G^\vee$ and such that the flat $G^\vee$-bundle
  underlying $\chi$ is induced from a flat $G'$-bundle admitting a
  $G'$-oper structure.
\end{theorem}

\begin{proof}
By going, if needed, to a finite cover of $X$, we can assume without
loss of generality that $M_\chi$ is connected (the statement for this
cover will imply the statement for $X$, since pullback of an oper to
the cover is still an oper). By \cite{BD}, Proposition 3.1.5(iii),
$M_\chi$ is not contained in any non-trivial parabolic subgroup of
$G^\vee$. Hence by Morozov's theorem (see e.g. \cite{Bou},
Ch. VIII, Section 10) it is contained in a proper connected reductive
subgroup $G' \subset G^\vee$.

The flat $G^\vee$-bundle $E_\chi = (\mcE_{G^\vee},\nabla_\chi)$ (where
$\mcE_{G^\vee}$ is the $G^\vee$-bundle introduced after equation
\eqref{nontr}) underlying
$\chi$ can then be reduced to a flat $G'$-bundle $E_{G',\chi} =
(\mcE_{G',\chi},\nabla'_\chi)$, i.e. we have an isomorphism of flat
$G^\vee$-bundles $E_{G',\chi}\times_{G'}G^\vee \cong E_\chi$.

Fix a point $x_0\in X$. Conjugating inside $G^\vee$ if needed, we may
assume that $G'$ is the fiber of ${\rm Ad}(\mcE_{G',\chi})$ at $x_0$.

It is shown in \cite{BH} that if $H$ is a connected reductive group
over $\C$, then any $H$-bundle $\mcE_H$ on $X$ has a canonical
Harder-Narasimhan reduction to a parabolic subgroup $P(\mcE_H)$ of
$H$. These canonical Harder-Narasimhan reductions for $\mcE_{G',\chi}$
and $\mcE_{G^\vee}$ define a parabolic subgroup $P(\mcE_{G',\chi})$ of
$G'$ and the Borel subgroup $P(\mcE_{G^\vee})=B^\vee\subset G^\vee$
(the one we have used in the definition of $G^\vee$-opers),
respectively. These subgroups must be compatible. Therefore
$P(\mcE_{G',\chi}) = B':=G'\cap B^\vee$ is a Borel subgroup of $G'$, and
$\mcE_{G',\chi}\cong \mcE_{B',\chi} \times_{B'}G'$, where $\mcE_{B',\chi}$
is a $B'$-bundle on $X$.

Let $T'$ be a maximal torus in $B'$ and $T^\vee$ a maximal torus of
$B^\vee$, such that $T' \subset T^\vee$. Consider the $B'/[B',B']\cong
T'$-bundle $\mcE_{T',\chi}$ associated to $\mcE_{B',\chi}$. By
construction, the induced $T^\vee$-bundle is the $T^\vee$-bundle
$\mcE_{T^\vee}$ associated to the $B^\vee$-bundle $\mcE_{B^\vee}$
which is the oper $B^\vee$-reduction of the $G^\vee$-bundle
$\mcE_{G^\vee}$. As explained at the beginning of Subsection
\ref{ALcomp}, the $G^\vee$-bundle $\mcE_{G^\vee}$ is induced from the
$PGL_2$-bundle corresponding to the vector bundle \eqref{nontr} under
the principal embedding $PGL_2 \hookrightarrow G^\vee$. This implies
that $\mcE_{T^\vee} \simeq K_X^{\rho}$ in the sense that for any
character $\psi\in \bold X^*(T^\vee)$, the corresponding line bundle
$\psi(\mcE_{T^\vee})$ is isomorphic to $K_X^{\langle \psi,\rho
  \rangle}$. This implies that the image of the cocharacter $\rho:
\mathbb G_m \to T^\vee$ is contained in $T' \subset T^\vee$.

Let us trivialize the bundle $\mcE_{B',\chi}$ over the formal
neighborhood of $x_0$, on which we pick a formal coordinate $z$. This
trivializes $E_{G',\chi}$ as well as $\mcE_{B^\vee}$ and
$\mcE_{G^\vee}$. We also obtain a trivialization of the corresponding
adjoint vector bundles. The affine space of connections on
$E_{G^\vee}$ is then identified with the space $d+\g^\vee[[z]]dz$. The
connections that come from connections on the $G'$-bundle
$E_{G',\chi}$ belong to its subspace $d+\g'[[z]]dz$. On the other
hand, according to Definition \ref{operde}, the oper connection
$\nabla_\chi$ has the form
\begin{equation}    \label{nabchi}
\nabla_\chi=d+(f+b(z))dz,
\end{equation}
where $f\in \g^\vee$ is a principal nilpotent element satisfying
$[\rho,f]=-f$ and $b(z) \in {\mathfrak b}^\vee[[z]]$.

Therefore $\g'$ contains $f+b(0)$. Since $\g'$ also contains
$\rho$, we obtain that
$$
\lim_{t\to 0} t Ad(\rho(t))(f+b(0))=f\in \g'.
$$
By the Jacobson-Morozov theorem, it now follows that $G'$ contains a
principal subgroup of $G^\vee$, as desired. Subtracting $fdz$ from
\eqref{nabchi}, we obtain that $b(z)\in \g'
\cap \mathfrak b^\vee[[z]] = {\mathfrak b}'[[z]]$. Therefore,
$\mcE_{G',\chi}$ with its $B'$-reduction $\mcE_{B',\chi}$ and
connection \eqref{nabchi} is a $G'$-oper.

From the classical result of Dynkin (see \cite{SS,EO}), which we
recall in Theorem \ref{list} below, it also follows that $G'$ is
simple. This completes the proof.
\end{proof}

We recall the classification of pairs $G'\subset G^\vee$, where
$G^\vee$ is a simple algebraic group over $\C$ and $G'$ is its
connected reductive subgroup containing a regular unipotent element of
$G^\vee$ (and hence a principal $PGL_2$ subgroup of
$G^\vee$). Following \cite{EO}, we will call such a pair (and the
corresponding pair of Lie algebras) a {\it principal pair}. In the
case when $G^\vee$ is of adjoint type that we are considering, it
suffices to classify the principal pairs of Lie algebras. This
classification is given by the following theorem, which is due to
\cite{SS} and in the characteristic 0 case goes back to the work of
Dynkin (see \cite{EO}, Theorem 6.4).

\begin{theorem}    \label{list}
The principal pairs of Lie algebras $\mathfrak{g}'\subset
\mathfrak{g}^\vee$ (with a proper inclusion) are given by the
following list:

(1) $\mathfrak{sp}(2n)\subset \mathfrak{sl}(2n)$, $n\ge 2$;

(2) $\mathfrak{so}(2n+1)\subset \mathfrak{sl}(2n+1)$, $n\ge 2$;

(3) $\mathfrak{so}(2n+1)\subset \mathfrak{so}(2n+2)$, $n\ge 3$;

(4) $G_2\subset \mathfrak{so}(7)$;

(5) $G_2\subset \mathfrak{so}(8)$;

(6) $G_2\subset \mathfrak{sl}(7)$;

(7) $F_4\subset E_6$.

(8) ${\mathfrak{sl}_2}\subset \g^\vee$ for any simple $\g^\vee$.

The Lie subalgebras $\g'$ given in (1),(2),(3),(5),(7) are the
invariant subalgebras of an automorphism of the Dynkin diagram
$\g^\vee$; (4) is obtained by composing (5) and (3); and (6) is
obtained by composing (5) and (2).
\end{theorem}

Using Theorem \ref{zar} and Theorem \ref{list}, we can give a
description of the possible Zariski closures $M_\chi$ of the monodromy
groups of arbitrary $G^\vee$-opers $\chi$. Namely, $M_\chi$
must be a simple subgroup of $G^\vee$ that contains a principal
$PGL_2$ subgroup of $G^\vee$.

\begin{lemma}    \label{al}
Suppose that a connected simple subgroup $H^\vee$ of a connected
simple group $G^\vee$ contains a principal $(P)SL_2$ subgroup of
$G^\vee$. Then we have a canonical map
\begin{equation}    \label{alM}
\al_{H^\vee,G^\vee}: \on{Op}_{H^\vee}(X) \to \on{Op}_{G^\vee}(X).
\end{equation}
\end{lemma}

\begin{proof}
Under the condition of the lemma, the Borel subgroup $B_H^\vee$ of
$H^\vee$ used in the definition of $H^\vee$-opers contains the Borel
subgroup of a particular principal $(P)SL_2$ subgroup of $G^\vee$ with
the Lie algebra spanned by the elements $h$ and $e$ of the
corresponding principal $\sw_2$ triple. Then $B_H^\vee$ is contained
in a unique Borel subgroup $B^\vee$ of $G^\vee$. Definition
\ref{operde} implies that for every $H^\vee$-oper $\eta$, the flat
$G^\vee$-bundle induced from the flat $H^\vee$-bundle underlying
$\eta$ has a structure of $G^\vee$-oper with respect to
$B^\vee$. This structure is unique by Lemma \ref{atmost}. Since the
space of opers does not depend on the choice of a Borel subgroup by
Lemma \ref{doesnotdep}, we obtain a canonical map \eqref{alM}.
\end{proof}

\begin{theorem}    \label{zar1}
Let $M_\chi$ be the Zariski closure of the monodromy group of a
$G^\vee$-oper $\chi$ on a curve of genus ${\rm g} > 1$, where $G^\vee$
is a simple algebraic group of adjoint type. Then $M_\chi$ is a simple
algebraic group of adjoint type containing a principal $PGL_2$
subgroup of $G^\vee$ and $\chi$ is in the image of the map
$\al_{M_\chi,G^\vee}$.
\end{theorem}

\begin{proof}
We argue by induction on $d=\dim(G^\vee)$, with the base being the
case $d=3$, which follows from Remark \ref{monim}. Suppose that
$M_\chi \neq G^\vee$. Then by Theorem \ref{zar} and Lemma \ref{al},
$M_\chi$ is contained in a simple subgroup $G' \subset G^\vee$ and
$\chi$ is in the image of the map $\al_{G',G^\vee}$. But $\on{dim}G' <
\on{dim}G^\vee$ and by passing to a finite cover of $X$, we can assume
without loss of generality that $G'$ is connected. Moreover, according
to the list of Theorem \ref{list}, if $G^\vee$ is of adjoint type,
then so is $G'$. Passing from $G^\vee$ to $G'$ and using our inductive
assumption, we obtain the result.
\end{proof}

\begin{remark}
Explicit examples of opers with proper subgroups $M_\chi \subset
G^\vee$ (in genus 0 with ramification) can be found in \cite{FG}.
\end{remark}

\subsection{General case, continued}    \label{genwt}

Let $G$ be a connected simple algebraic group and $\la \in
\Lambda^+$. We will use the notation introduced in Subsection
\ref{prin}. In Conjecture 5.1 of \cite{EFK2} we gave a conjectural
formula for the eigenvalues $\beta_\la(x,\ol{x})$ of the Hecke
operator $H_\la$. It generalizes Corollary \ref{eigH1} in the case
$G=PGL_n$ and Conjecture \ref{eigHG} in the case of principal
weights. Actually, for each $\chi \in \on{Op}^0_{\LG}(X)_{\R}$ there
are (conjecturally) finitely many eigenvalues which differ from each
other by a root of unity (see \cite{EFK2}, Remark 5.1). Hence the
statement Conjecture 5.1 of \cite{EFK2} is made in the form ``up to
a non-zero scalar.'' Here it is in terms of the notation of the
present paper.

\begin{conjecture}   \label{51}
The eigenvalues $\beta_\la(x,\ol{x})$ of the Hecke operator $H_\la$
corresponding to $\chi \in \on{Op}^\ga_{\LG}(X)_{\R}$ are
equal to $h_{\chi,\la}(s_\la,\ol{s_{-w_0(\la)}})$ up to a non-zero
scalar.
\end{conjecture}

In \cite{EFK2}, Section 5, we gave the following strategy for proving
this conjecture:

(1) Prove that eigenvalues
$\beta_\la(x,\ol{x})$ satisfy a system or differential equations
\begin{equation}    \label{Ichi}
  D \beta_\la(x,\ol{x}) = 0, \qquad \ol{D'} \beta_\la(x,\ol{x}) = 0,
  \qquad D \in I_{\la,\chi}, D' \in I_{-w_0(\la),\chi}.
\end{equation}
This is essentially the statement of \cite{EFK2}, Conjecture 5.5. 

(2) Show that every solution of the system \eqref{Ichi} is equal to
$h_{\chi,\la}(s_\la,\ol{s_{-w_0(\la)}})$ up to a non-zero scalar. In
the case when the monodromy representation of the flat vector bundle
$(\mcV_\la,\nabla_{\chi,\la})$ is irreducible, this is the statement
of \cite{EFK2}, Corollary 5.4.

We now consider this statement in the case when the monodromy
representation of the flat vector bundle
$(\mcV_\la,\nabla_{\chi,\la})$ is reducible, using the results of
the previous subsection. For simplicity, we will assume that $G^\vee$
is of adjoint type.

\begin{proposition}    \label{forevery}
For every $\chi \in \on{Op}^0_{\LG}(X)_{\R}$, $h_{\chi,\la}(s_\la,
\ol{s_{-w_0(\la)}})$ is a unique, up to a scalar, non-zero section of
$\Omega_X^{-\frac{d(\la)}{2}}$ annihilated by the ideals
$I_{\la,\chi}$ and $\ol{I_{-w_0(\la),\chi}}$.
\end{proposition}

\begin{proof}
In the case when the monodromy representation of the flat vector
bundle $(\mcV_\la,\nabla_{\chi,\la})$ is irreducible, this was proved
in \cite{EFK2}, Corollary 5.4.

Suppose now that $(\mcV_\la,\nabla_{\chi,\la})$ is reducible. This
means that the Zariksi closure $M = M_\chi$ of the monodromy of $\chi$
is a proper subgroup of $G^\vee$. By Theorem \ref{zar1}, $M$ is a
simple subgroup of $G^\vee$ (also of adjoint type) that contains a
principal $PGL_2$ subgroup of $G^\vee$, and $\chi$ is the image of an
$M$-oper $\chi^M$ under the map $\al_{M,G^\vee}: \on{Op}_M(X) \to
\on{Op}_{G^\vee}(X)$.

Under the action of $M$, the representation $V_\la$ of $\LG$
decomposes into a direct sum of irreducible representations of
$M$. Denote by $V^{M}_\la$ its component containing the highest weight
subspace of $V_\la$ with respect to the Borel subgroup $B^\vee \cap M$
of $M$ (where $B^\vee$ is the Borel subgroup of $G^\vee$ we have used
to define $G^\vee$-opers). Since ${\mathfrak m} := \on{Lie}(M)$
contains a principal $\sw_2$ subalgebra, $V^{M}_\la$ contains the
highest component $V_{d(\la)}$ in the decomposition \eqref{decsl2} of
$V_\la$ under the principal $\sw_2$. This implies that the dual
representation $(V^{M}_\la)^*$ is isomorphic to $V^{M}_{-w_0(\la)}$,
the component containing the highest weight subspace of $V_{-w_0(\la)}
= V_\la^*$.

The decomposition of $V_\la$ into a direct sum of irreducible
representations of $M$ gives rise to a direct sum decomposition of
the flat vector bundle $(\mcV_\la,\nabla_{\chi,\la})$ into irreducible
flat vector bundles each having real monodromy. In particular, the
flat subbundle $(\mcV^{M}_\la,\nabla_{\chi,\la})$ of
$(\mcV_\la,\nabla_{\chi,\la})$ corresponding to the highest weight
component $V^{M}_\la \subset V_\la$ is irreducible and has real
monodromy. Therefore, there is a unique up to a scalar non-zero
pairing
$$
h^{M}_{\chi,\la}(\cdot,\cdot): (\V^M_\la,\nabla_{\chi,\la}) \otimes
  (\ol{\V}^M_{-w_0(\la)},\ol\nabla_{\chi,-w_0(\la)}) \to ({\mc
    C}^\infty_X,d),
$$
where $\V^M_{-w_0(\la)}$ is the subbundle of
$\V_{-w_0(\la)}$ corresponding to $V^M_{-w_0(\la)} \simeq (V^M_\la)^*$.
From the above direct sum decomposition it is clear that
$h^M_{\chi,\la}(\cdot,\cdot)$ is equal to the restriction of
$h_{\chi,\la}(\cdot,\cdot)$ to $\V^M_\la \otimes
\ol{\V}^M_{-w_0(\la)}$ up to a non-zero scalar.

Moreover, the canonical section $s_\la \in
\Gamma(X,K_X^{-\frac{d(\la)}{2}} \otimes \V_\la)$ corresponding to the
Borel reduction of the $G^\vee$-oper $\chi$ (see formula \eqref{sla})
is equal to the image of the canonical section $s^M_\la \in
\Gamma(X,K_X^{-\frac{d(\la)}{2}} \otimes \V^M_\la)$ for the $M$-oper
$\chi^M$ under the embedding $\V^M_\la \hookrightarrow \V_\la$. And
likewise, for the canonical sections $s_{-w_0(\la)}$ and
$s^M_{-w_0(\la)}$.

Using irreducibility of the flat bundle $(\V^M_\la,\nabla_{\chi,\la})$
in the same way as in the proof of Corollary \ref{monG}, we obtain
that $h^M_{\chi,\la}(s^M_\la,\ol{s^M_{-w_0(\la)}})$ is a unique, up to
a scalar, non-zero section of $\Omega_X^{-\frac{d(\la)}{2}}$
annihilated by the ideals $I_{\la,\chi}$ and
$\ol{I_{-w_0(\la),\chi}}$. But according to the above discussion,
$h^M_{\chi,\la}(s^M_\la,\ol{s^M_{-w_0(\la)}})$ is equal to
$h_{\chi,\la}(s_\la,\ol{s_{-w_0(\la)}})$ up to a non-zero
scalar. This implies the statement of the proposition.
\end{proof}

This proves part (2) of the above argument. Hence in order to prove
Conjecture 5.1 of \cite{EFK2} it remains to prove Conjecture 5.5 of
\cite{EFK2} (see \cite{EFK2}, Section 5.3 for an outline of how to
prove it using the results of \cite{BD}).

\subsection{Functoriality in the analytic Langlands
  correspondence}    \label{funct}

Consider the framework of the Langlands Program for a smooth
projective curve $X$ over a finite field $\Fq$. Let $G$ and $H$ be two
split reductive algebraic groups over $\Fq$, and $G^\vee$ and $H^\vee$
their Langlands dual groups. The {\em Langlands functoriality
principle} (see \cite{Arthur:funct} for a survey) is the statement
that for any homomorphism
\begin{equation}    \label{homa}
a: H^\vee \to G^\vee
\end{equation}
between them there should be a map (sometimes called {\em transfer}),
from the set of $L$-packets of tempered automorphic representations of
$H(\AD_F)$ to the set of $L$-packets of tempered automorphic
representations of $G(\AD_F)$ (where $\AD_F$ is the ring of adeles of
$F = \Fq(X)$, the function field of $X$).

The existence of such a map is quite surprising: even though we have a
homomorphism \eqref{homa} of dual groups $a: H^\vee \to G^\vee$, there
is {\em a priori} no connection between the groups $G$ and $H$. The
explanation is found on the dual side of the Langlands correspondence,
which gives a parameterization of $L$-packets of tempered automorphic
representations of $G(\AD)$ in terms of homomorphisms $W(F) \to
G^\vee$, where $W(F)$ is the Weil group of $F$. Given a homomorphism
\eqref{homa}, every Langlands parameter $\sigma: W(F) \to H^\vee$ for
$H(\AD)$ gives rise to a Langlands parameter $a \circ \sigma: W(F) \to
G^\vee$ for $G(\AD_F)$.

This interpretation also makes it clear that functoriality should
satisfy the following {\em transitivity} property: if $K$ is another
reductive group and we have a chain of homomorphisms of dual groups:
\begin{equation}    \label{transitivity}
K^\vee \to H^\vee \to G^\vee,
\end{equation}
then the composition of the transfers from
$K(\AD_F)$ to $H(\AD_F)$ and from $H(\AD_F)$ to $G(\AD_F)$ should
coincide with the transfer obtained directly from the composition
$K^\vee \to \LG$.

Another important property is that the Hecke eigenvalues of the
automorphic representations (at the unramified places) should match
under the transfer in a natural way.

What should be the analogues of the Langlands functoriality and the
transfers in the analytic Langlands correspondence for a curve $X$
over $\C$?

According to our main Conjecture \ref{mainc}, the role of the
Langlands parameters for a reductive group $X$ is now played by real
$G^\vee$-opers. From Theorem \ref{zar1}, it is clear that the
homomorphisms \eqref{homa} that we should consider are the ones that
map a principal $SL_2$ (or $PGL_2$) subgroup of $H^\vee$ to a
principal $SL_2$ (or $PGL_2$) subgroup of $G^\vee$. We will call such
homomorphisms {\em principal}. For simple $G^\vee$, at the level of
Lie algebras, the list of principal homomorphisms is given in Theorem
\ref{list} following \cite{SS} and \cite{EO} (note also that we have
discussed principal embeddings in the case when $\g^\vee = \sw_n$ in
Subsection \ref{prin}.)

Suppose for simplicity that $G^\vee$ is of adjoint type. Then it
follows from Theorem \ref{list} that $H^\vee$ is also of adjoint
type. Given a principal homomorphism \eqref{homa}, by
  Lemma \ref{al} we obtain a canonical map
\begin{equation}    \label{embop}
  \al_{H^\vee,G^\vee}: \on{Op}_{H^\vee}(X) \to
  \on{Op}_{G^\vee}(X)
\end{equation}
which is an embedding of affine spaces in the case when $H^\vee$ and
$G^\vee$ are of adjoint type. The following result follows immediately
from Theorem \ref{zar1}.

\begin{proposition}
A $G^\vee$-oper in the image of $\al_{H^\vee,G^\vee}$ which is not in
the image of $\al_{K^\vee,G^\vee}$ for any $K^\vee \subset H^\vee$
consists of the $G^\vee$-opers on $X$ such that the Zariski closure of
their monodromy is equal to $H^\vee$.
\end{proposition}

Thus, we obtain a stratification of the affine space
$\on{Op}_{G^\vee}(X)$ by affine subspaces given by the images of
the embeddings $\al_{H^\vee,G^\vee}$ corresponding to all principal
homomorphisms \eqref{homa}. It gives rise to the corresponding
family of embeddings of the sets of real opers, which are the
Langlands parameters of the analytic Langlands correspondence:
\begin{equation}    \label{embopR}
  \al^{\R}_{H^\vee,G^\vee}: \on{Op}_{H^\vee}(X)_{\R} \hookrightarrow
  \on{Op}_{G^\vee}(X)_{\R}.
\end{equation}

Since these maps are transitive for a pair of embeddings
\eqref{transitivity}, Conjecture \ref{mainc} implies that an analogue
of the transitivity property holds in the analytic Langlands
correspondence.

Next, by analogy with the Langlands functoriality discussed above, we
expect that the eigenvalues of the Hecke operators for the groups $G$
and $H$ related by a principal homomorphism \eqref{homa} should
match. Let us verify that this is compatible with our Conjecture
\ref{51} (which is Conjecture 5.1 of \cite{EFK2}) that gives an
explicit formula for these eigenvalues (up to a scalar).

Given $\la \in \Lambda^+_G$, let $a(\la) \in \Lambda^+_H$ be the
dominant integral weight of $H$ obtained via the homomorphism
\eqref{homa}. We have the corresponding Hecke operators $H_\la$ and
$H_{a(\la)}$ for the groups $G$ and $H$, respectively, depending on a
point of $X$. According to Conjecture \ref{51}, the eigenvalues
$H_\la$ and $H_{a(\la)}$ are parametrized by real opers in
$\on{Op}_{G^\vee}(X)_{\R}$ and $\on{Op}_{H^\vee}(X)_{\R}$,
respectively (more precisely, for each real oper, we expect finitely
many Hecke eigenvalues differing by a root of unity; see \cite{EFK2},
Remark 5.1 for more details). Given $\chi \in
\on{Op}_{H^\vee}(X)_{\R}$, denote by $a(\chi)$ the image of $\chi$
under the embedding \eqref{embopR}. Let
$\beta_{a(\la)}^\chi(x,\ol{x})$ be the eigenvalue of $H_{a(\la)}$, and
$\beta_{\la}^{a(\chi)}(x,\ol{x})$ the corresponding eigenvalue of
$H_{\la}$. Matching of these eigenvalues means that they are equal up
to an overall non-zero scalar (independent of the point of $X$).

Our conjectural formula for these eigenvalues in Conjecture \ref{51}
says that $\beta_{a(\la)}^\chi(x,\ol{x})$
(resp. $\beta_{\la}^{a(\chi)}(x,\ol{x})$) is equal to
$h^{H^\vee}_{\chi,\la}(s_{a(\la)}, \ol{s_{a(-w_0(\la))}})$
(resp. $h^{G^\vee}_{a(\chi),\la}(s_\la, \ol{s_{-w_0(\la)}})$) up to a
non-zero scalar. But we have shown in the proof of Proposition
\ref{forevery} that the two expressions are proportional to each
other. Hence we find that our conjectural formulas for the Hecke
eigenvalues are indeed compatible with the analytic Langlands version
of functoriality.

\begin{remark}
An important example of functoriality in the case of a curve over a
finite field comes from the embedding of a maximal torus of the group
$G^\vee$, $T^\vee \hookrightarrow G^\vee$. In this case, the
corresponding automorphic functions for the group $G$ over the adeles
are known as the {\em Eisenstein series}. The above discussion
explains why we do not expect analogues of Eisenstein series in the
analytic Langlands correspondence for a simple algebraic group $G$ and
a curve over $\C$ (as we can see from Conjecture \ref{mainc}): such an
embedding is not principal and therefore should not lead to
functoriality.
\end{remark}

\subsection{Analytic Langlands correspondence twisted by an ${\rm
    Aut}G$-torsor on $X$}     \label{twisted}

Analytic Langlands correspondence can be naturally generalized to the case when the connected reductive group $G$ is replaced by a flat group scheme $\mathcal G$ 
over $X$ with fibers isomorphic to $G$. For simplicity let us discuss this theory in the case of complex curves ($F=\Bbb C$).

\begin{example}\label{previ} 
Suppose $G=\Bbb G_m^n$ is an $n$-dimensional torus. Then the possible 
groups $\mathcal G$ are parametrized by homomorphisms $\phi: \pi_1(X)\to GL_n(\Bbb Z)$ 
with finite image $\Gamma$. Let 
$\widetilde X$ be the cover of $X$ corresponding to the kernel of $\phi$. Then $\Gamma$ acts 
on $\widetilde X$ and $\mathcal G$ trivializes over $\widetilde X$, 
so the corresponding moduli stack $Bun_{\mathcal G}(X)$ 
is the stack of $\Gamma$-equivariant $G$-bundles on $\widetilde X$. 
Set-theoretically, this is the subgroup of ${\rm Pic}(\widetilde X)^n$ 
consisting of bundles $E$ with a consistent family of isomorphisms 
$\gamma_*E\cong \phi(\gamma)E$, $\gamma\in \Gamma$. 
In this case it is easy to see (see \cite{EFK2}) that the spectrum of Hecke operators 
is parametrized by $\Gamma$-equivariant real
 $G^\vee=\Bbb G_m^n$-opers on $\widetilde X$, 
i.e., lifts $\rho: \pi_1(X)\to \Gamma\ltimes \Bbb G_m^n(\Bbb R)$ of $\phi$ 
such that the local system $\rho: {\rm Ker}\phi=\pi_1(\widetilde X)\to \Bbb G_m^n(\Bbb R)$ is an oper 
(hence also an anti-oper), similarly to 
Subsection \ref{ALcomp}.  
\end{example} 

\begin{example} 
Suppose that $G$ is adjoint. In this case, the possible groups $\mathcal G$ are classified by $H^1(X,{\rm Aut}\Delta_G\ltimes \mathcal O_{X,G})$, where $\mathcal O_{X,G}$ is the sheaf of regular functions on $X$ with values in $G$. 
However, two elements $\theta_1,\theta_2\in H^1(X,{\rm Aut}\Delta_G\ltimes \mathcal O_{X,G})$ which map to the same element in 
$H^1(X,{\rm Aut}\Delta_G)=\Hom(\pi_1(X),{\rm Aut}\Delta_G)$ 
define Morita equivalent groups $\mathcal G$, so 
the corresponding moduli spaces are the same
(\cite{Br}, Subsection 1.6, Proposition 1.2). In other words, similarly to Subsection \ref{princibun}, the theory depends only on the inner class of $\mathcal G$ (see \cite{Br}, Remark 1.2). This is, in fact, a general feature which extends beyond $F=\Bbb C$. 

Thus we may restrict ourselves to groups $\mathcal G$ obtained from maps 
$\phi: \pi_1(X)\to {\rm Aut}\Delta_G$. So, similarly to Example \ref{previ} 
we may define $\Gamma:={\rm Im}\phi$ and realize the corresponding stack 
$Bun_{\mathcal G}(X)$ as the stack of 
$\Gamma$-equivariant principal $G$-bundles on the 
$\Gamma$-cover $\widetilde X$ of $X$. As in Example \ref{previ}, we expect that 
the spectrum of Hecke operators is parametrized by $\Gamma$-equivariant 
real $G^\vee$-opers on $\widetilde X$, i.e. lifts $\rho: \pi_1(X)\to \Gamma\ltimes G^\vee(\Bbb R)$ of $\phi$ whose restriction to ${\rm Ker}\phi$ is an oper (hence also an anti-oper). 
\end{example} 

\begin{remark} More generally, suppose a finite group $\Gamma$ acts simultaneously on $G$ by root datum automorphisms and on a curve $\widetilde X$. Then we can consider harmonic analysis 
on the space ${\rm Bun}^\circ_G(\widetilde X)^\Gamma$ of $\Gamma$-equivariant 
regularly stable $G$-bundles on $\widetilde X$. If $\Gamma$ acts 
on $\widetilde X$ freely, this reduces to the above setting 
with $X=\widetilde X/\Gamma$, but the theory extends naturally 
to the case when the action is not necessarily free. 
We note that such moduli spaces of twisted bundles 
have been recently studied in connection with twisted conformal blocks
and twisted Verlinde formula, see \cite{DM,HK}. 
In the framework of the usual Langlands correspondence 
over function fields, such twisted setting is considered in \cite{L}, Section 12. 
\end{remark} 

 \section{Analytic Langlands correspondence over $\Bbb R$}    \label{R}

\subsection{The general setup} 
In this section we will focus on the case $F=\Bbb R$ and propose a conjectural description of
the spectrum of Hecke operators in terms of $G^\vee$-opers satisfying suitable reality conditions, generalizing the results of \cite{EFK3}, Subsection 4.7. Much of our analysis is based on \cite{GW}, Section 6. 

We first specialize the setting of Subsection \ref{formsredgps} to the case $F=\Bbb R$, so 
$F_{\rm sep}=\Bbb C$ and $\Gamma_F=\Bbb Z/2$. We will only consider either real or complex points of algebraic groups $G$, and will write $G$ for $G(\Bbb C)$ when no confusion is possible. Let $G$ a split connected reductive group defined over $\Bbb Q$ and $G^\vee$ its Langlands dual group. Let $Z,Z^\vee$ be the centers of $G,G^\vee$. These groups are equipped with a natural operation of complex conjugation, $g\mapsto \overline g$. A {\bf  real structure} on $G$ is a holomorphic automorphism $\theta: G\to G$ satisfying the 1-cocycle condition $\theta\circ \theta^*={\rm Id}$, where $\theta^*(g):=\overline {\theta(\overline g)}$. Two such 1-cocycles differ by a coboundary iff the corresponding real structures are isomorphic. In fact, one can (and usually does) choose a representative $\theta$ of its cohomology class so that $\theta$ commutes with complex conjugation, i.e., $\theta^*=\theta$ and $\theta^2={\rm Id}$, which gives rise to the {\bf Satake diagram} of the corresponding real form.\footnote{Another possibility is to choose $\theta$ to commute with the complex conjugation of the compact form of $G$, which gives rise to the {\bf Vogan diagram} of the real form.} Such $\theta$ gives rise to an {\bf antiholomorphic involution} $\sigma(g):=\theta(\overline g)$. The corresponding {\bf  group of real points} $G^\sigma=G^\sigma(\Bbb R)$ (which may be disconnected) is the 
subgroup of $g\in G$ stable under $\sigma$, i.e., satisfying 
$\theta(g)=\overline g$. The inner class of $\sigma$ gives rise to a root datum involution $s=s_\sigma$ for $G$ which is also one for $G^\vee$. 

Recall \cite{ABV} that to $G,s$ we may attach the {\bf  Langlands L-group} ${}^LG={}^LG_s$, the semidirect product of $\Bbb Z/2={\rm Gal}(\Bbb C/\Bbb R)$ by $G^\vee$, with the action of $\Bbb Z/2$ defined by $\omega\circ s$, where $\omega$ is the {\bf Chevalley involution} defining the compact form of $G$.   

Let $X=X(\Bbb C)$ be a compact complex Riemann surface of genus ${\rm g}\ge 2$. Let $\tau: X\to X$ 
be an antiholomorphic involution. We specialize the setting of 
Subsections \ref{princibun},\ref{smpro},\ref{heop} to the case $F=\Bbb R$. 
Given a holomorphic principal $G$-bundle $P$ on $X$, we can define 
the antiholomorphic bundle $\tau(P)$, hence a holomorphic bundle 
$(\sigma,\tau)(P)$. A {\bf  pseudo-real structure} on $P$ is an isomorphism  $A: (\sigma,\tau)(P)\to P$. Such a structure defines a class 
$\alpha_P$ in 
$$
 H^2(\Bbb Z/2,Z^s(\Bbb C))={\rm Ker}(1-s|_Z)/{\rm Im}(1+s|_Z)
 $$ 
 which depends only on $P$ and not on $A$. A pseudo-real structure $A$ on $P$ is a {\bf  real structure} if 
\begin{equation}\label{condi}
A\circ (\sigma,\tau)(A)=1
\end{equation} 
(in particular, this means that the cocycle $a_A$ and hence the class $\alpha_P$ equals $1$). 
For example, if $G$ is adjoint then the isomorphism $A$ is unique if exists and \eqref{condi} is automatic if ${\rm Aut}(P)=1$, which happens for regularly stable bundles. Note that if  
$g\in  G$ and $g\sigma(g)=1$ then 
$$
Ag^{-1}: ({\rm Ad}(g)\sigma,\tau)(P)\to P
$$ 
satisfies \eqref{condi} with $\sigma$ replaced by
 $\sigma'={\rm Ad}(g)\sigma$. Thus the space 
of such regularly stable bundles depends only on the inner class $s$ of $\sigma$, in agreement with Subsection \ref{princibun} (see also \cite{BGH}, Proposition 3.8). Following Subsection \ref{princibun}, we denote this space by ${\rm Bun}^\circ_{G,s}(X,\tau)$. 

The space ${\rm Bun}^\circ_{G,s}(X,\tau)$ is a real analytic manifold, which 
is a disjoint union of open submanifolds ${\rm Bun}^\circ_{G,s,\alpha}(X,\tau)$, $\alpha\in H^2(\Bbb Z/2,Z^s(\Bbb C))$. Moreover, 
for every character $\chi$ of 
$$
H^1(\Bbb Z/2,Z^s(\Bbb C))={\rm Ker}(1+s|_Z)/{\rm Im}(1-s|_Z)
$$ 
we have a Hermitian line bundle $\mathcal L_\chi$ on 
each ${\rm Bun}^\circ_{G,s,\alpha}(X,\tau)$ defined in Subsection \ref{heop}.

Let 
$$
\mathcal H(s,\tau,\alpha,\chi):=L^2({\rm Bun}^\circ_{G,s,\alpha}(X,\tau),\mathcal L_\chi)
$$ 
be the Hilbert space of $L^2$ half-densities on ${\rm Bun}^\circ_{G,s,\alpha}(X,\tau)$ valued in $\mathcal L_\chi$. 
Let 
$$ 
\mathcal H(s,\tau,\alpha)=\oplus_{\chi} \mathcal H(s,\tau,\alpha,\chi),\  \mathcal H(s,\tau)=\oplus_{\alpha} \mathcal H(s,\tau,\alpha).
$$
We have (conjecturally) a spectral decomposition of $\mathcal H(s,\tau)$ under the action Hecke operators compatible with the $(\alpha,\chi)$-grading.

\begin{remark} To be more precise, the definition of the Langlands L-group in \cite{ABV} uses $s$ instead of $\omega\circ s$. 
This is in fact a major difference between the classical Langlands correspondence for real groups and the analytic Langlands correspondence for curves over $\Bbb R$. An explanation of this phenomenon is provided by \cite{EFK1}, Proposition 3.6, which says that taking the formal adjoint of quantum Hitchin Hamiltonians corresponds to applying the Chevalley involution on opers. See also Remark \ref{witrema} below.
\end{remark}

\subsection{The case when $\tau$ has no fixed points}\label{nofix} 
We first consider the easier case when $\tau$ has no fixed points. Let $\rho$ be a local system on the non-orientable surface $X/\tau$ with structure group ${}^LG$. If we choose a base point $p\in X/\tau$ 
then such a local system corresponds to a homomorphism $\pi_1(X/\tau,p)\to {}^LG$ which is unique up to conjugation. We will say that $\rho$ is an {\bf  L-system} if it attaches to every orientation-reversing path in $X/\tau$ a conjugacy class in ${}^LG$ that maps to the nontrivial element in $\Bbb Z/2$. The following conjecture is equivalent to the conjecture made in \cite{GW}, Section  6.2 on the basis of insights from 4-dimensional supersymmetric gauge theory
(as well as the duality proposal from \cite{BS}). 

\begin{conjecture}\label{gawi} (i) There is an orthogonal decomposition 
$$
\mathcal H(s,\tau,1)=\bigoplus_\rho \mathcal H(s,\tau,1)_\rho,
$$
where $\rho$ runs over L-systems on $X/\tau$ with values in ${}^LG={}^LG_s$ whose pullback to $X$ have the structure of a $G^\vee$-oper. 

(ii) For $\lambda\in \Lambda_+$ the Hecke operator $H_{\lambda x+s(\lambda)\overline x}$  
acts on $\mathcal H(s,\tau,1)_\rho$ by the eigenvalue 
$\bbe_{\rho,\lambda}(x,\overline x)$ defined by the formula in \cite{EFK2}, Conjecture 5.1. In particular, if $G=PGL_n$, $\lambda=\omega_1$ and 
$L_\rho=\partial^n+a_2\partial^{n-2}+...+a_n$ is the $SL_n$-oper 
(i.e., holomorphic differential operator 
$K_X^{\frac{1-n}{2}}\to K_X^{\frac{1+n}{2}}$) corresponding to $\rho$ 
then $\bbe_{\lambda,\rho}(x,\overline x)$ 
is the (unique up to scaling) single-valued 
section of $|K_X|^{1-n}$ satisfying the system of oper equations 
$L_\rho\bbe=0,\overline{L_\rho^*}\bbe=0$. 
\end{conjecture} 

\begin{remark}\label{witrema} More precisely, as was explained to us by E. Witten, what comes from ordinary gauge theory is this picture for the {\bf  compact} inner class $s$. To obtain other inner classes, one needs to consider {\bf twisted gauge theory} where the twisting is by a root datum automorphism of $G$. Namely, gauge fields in this theory are invariant under complex conjugation $\tau$ composed with this automorphism. This may be seen as the physical explanation of the appearance 
of the Chevalley involution in the definition of ${}^LG$ in analytic Langlands correspondence, which does not happen in the usual Langlands correspondence for real groups. 
\end{remark} 

\begin{example} Let $s=\omega$ (the compact inner class). 
Then ${}^LG=\Bbb Z/2\times G^\vee$, so 
an L-system is the same thing as a $G^\vee$-local system on $X/\tau$. 
So in this case according to Conjecture \ref{gawi}, 
the spectral local systems are $\rho$ which are isomorphic to $\rho^\tau$ and such that $\rho$ is an oper (hence also an anti-oper), so $\rho$ is a real oper 
``with real coefficients". But among these we should only choose 
those local systems that descend to $X/\tau$ and the eigenspaces are labeled by these extensions.

More precisely, recall that opers for adjoint groups have no
nontrivial automorphisms (\cite{BD2}, \S 1.3). So for any connected reductive $G$ we get an obstruction for such $\rho$ to descend to $X/\tau$ which lies in $Z^\vee/(Z^\vee)^2=H^2(\Bbb Z/2,Z^\vee)$. Moreover, if this obstruction vanishes then the freedom for choosing the extension  is in a torsor over $H^1(\Bbb Z/2,Z^\vee)=Z_2^\vee$, the 2-torsion subgroup in $Z^\vee$. 

Indeed, $\pi_1(X/\tau)$ is generated by $\pi_1(X)$ and an element $t$ such that $tbt^{-1}=\gamma(b)$
for some automorphism $\gamma$ of $\pi_1(X)$, and 
$t^2=c\in \pi_1(X)$, so that $\gamma^2(b)=cbc^{-1}$. 
So given a representation $\rho: \pi_1(X)\to G^\vee$, 
an L-system would be given by an assignment $\rho(t)=T\in G^\vee$ such that
(1) $T^2=\rho(c)$ and (2) $T\rho(a)T^{-1}=\rho(\gamma(a))$. 
If $\rho\cong \rho\circ \gamma$ then 
$T$ satisfying (2) is unique 
up to multiplying by $u\in Z^\vee$, and $T^2=\rho(c)z$, $z\in Z^\vee$. 
Moreover, if $T$ is replaced by $Tu$ then $z$ is replaced by $zu^2$, hence 
the obstruction to satisfying (1) lies in $Z^\vee/(Z^\vee)^2$. And if this obstruction vanishes, 
then the choices of $T$ form a torsor over $Z_2^\vee$ acting by 
$T\mapsto Tz$. 

\begin{remark} As pointed out in \cite{GW}, Section 6, 
this reality condition on the $G^\vee$-oper on $X$ 
is equivalent to the condition that $\rho$ extends as a topological local system 
to the 3-manifold 
$$
U_\tau:=(X\times [-1,1])/(\tau, -{\rm Id})
$$ 
whose boundary is $X$, introduced in \cite{GW}, and this extension is a part of the data. This follows from the fact that the inclusion $X/\tau\hookrightarrow U_\tau$ is a homotopy equivalence. 
\end{remark}

\end{example} 

\begin{remark} We have the inflation-restriction exact sequence
$$
H^1(\pi_1(X),Z^s(\Bbb C))^{\Bbb Z/2}\to H^2(\Bbb Z/2,Z^s(\Bbb C))\to H^2(\pi_1(X/\tau),Z^s(\Bbb C)).
$$
Let $\overline\alpha$ be the image in $H^2(\pi_1(X/\tau),Z^s(\Bbb C))$ 
of $\alpha\in H^2(\Bbb Z/2,Z^s(\Bbb C))$. So $\overline\alpha=1$ iff $\alpha$ is the image of 
$\eta\in H^1(\pi_1(X),Z^s(\Bbb C))^{\Bbb Z/2}=H^1(X,Z^s(\Bbb C))^{\Bbb Z/2}$, which corresponds to a pseudo-real $Z^s$-bundle 
on $X$. Multiplication by $\eta$ acts on the space of pseudo-real bundles commuting with Hecke operators, changing $\alpha_P$ to $\alpha_P+\alpha$. 
This implies that if $\overline \alpha=1$ then 
Conjecture \ref{gawi} generalizes in a straightforward way to 
give the spectral decomposition of $\mathcal H(s,\tau,\alpha)$: namely, the spectrum of the Hecke operators is the same as in $\mathcal H(s,\tau,1)$. More generally, this shows that the spectrum 
of Hecke operators on $\mathcal H(s,\tau,\alpha)$ for general $\alpha$ 
depends only on $\overline\alpha$. 

It remains to describe the spectrum in the case when $\overline\alpha\ne 1$. 
As was explained to us by D. Gaiotto, in this case Conjecture \ref{gawi} can be generalized 
by considering gauge theory on the 3-manifold $U_\tau$ (homotopy equivalent to $X/\tau$)
twisted by the $Z^s$-gerbe corresponding to $\overline \alpha$. Mathematically 
this corresponds to the setting of Subsection \ref{gerbes} extended to the case $F=\Bbb R$. 
We omit the details. 
\end{remark} 

\begin{example} Let $G=K\times K$ for some complex group $K$, and 
$s$ be the permutation of components (the only real form  in this inner class is $K$ 
regarded as a real group). In this case ${\rm Bun}^\circ_{G,s}(X,\tau)={\rm Bun}^\circ_{G}(X)$, 
the usual moduli space for the complex field. Also ${}^LG_s={}^LG_{\omega\circ s}=\Bbb Z/2\ltimes (K^\vee\times K^\vee)$, where $\Bbb Z/2$ acts by permutation. So an L-system is a $K^\vee\times K^\vee$ local system on $X$ of the form $(\rho,\rho^\tau)$. Thus the spectrum is parametrized by $\rho$ such that both $\rho$ and $\rho^\tau$ are opers, i.e., $\rho$ is both an oper and an anti-oper, i.e. a real oper, which agrees with the main conjecture from \cite{EFK2}. (Note that in this case 
$H^i(\Bbb Z/2,Z^\vee)=1$ so there is no obstructions or freedom for extensions).   
\end{example} 

\begin{example}\footnote{This is based on the letter \cite{W} in which E. Witten kindly explained to us the predictions of \cite{GW} in the abelian case.} Let us verify Conjecture \ref{gawi} 
for $G=GL_1$. In this case the possible $s$ are $1$ and 
$-1$, each being its entire inner class. So consider two cases: 

{\bf  1.} Compact case: $s=-1$. Then the spectrum is parametrized by characters of 
$\pi_1(X/\tau)$, i.e., elements of $H^1(X/\tau,\Bbb C^\times)=(\Bbb C^\times)^g\times \Bbb Z/2$, which come from $GL_1$-opers. 

{\bf  2.} Split case: $s=1$. Then the spectrum is parametrized by 
$H^1(X/\tau,\Bbb C^\times_\tau)$, where $\Bbb C^\times_\tau$ 
is the local system where $\tau$ acts by inversion. We have 
$H^1(X/\tau,\Bbb C^\times_\tau)=(\Bbb C^\times)^g$, and the spectrum is 
parametrized by such local systems that come from $GL_1$-opers. 

In both cases the resulting ``spectral" opers form a lattice $\Bbb Z^g$. 
They are of the form $d+\phi$ where (roughly speaking) $\phi$ in the first case has 
integral periods on $\tau$-antiinvariant cycles and in the second case integral periods on $\tau$-invariant cycles.
\end{example} 

\begin{example} Consider the simplest instance of the previous example, with genus 1 curve $X=\Bbb C/(\Bbb Z\oplus \Bbb Zi)$ and coordinate $z=x+iy$, with $\tau(z)=\overline{z}+\frac{1}{2}$. Then $X/\tau$ is the Klein bottle with $\pi_1(X/\tau)$ generated by $t$ and $b$ with $tbt^{-1}=b^{-1}$. 

{\bf  1.} In the compact case $s=-1$ we need to consider characters of this group, which send $b$ to $\pm 1$ and $t$ to any nonzero number. So the corresponding opers are $L=d+\phi$ where $\phi$ has half-integral period in the imaginary direction, i.e. $\phi=\pi n$, $n\in \Bbb Z$. Thus the Hecke eigenvalue is expected to be proportional 
to $e^{2\pi iny_0}=e^{\pi n(z_0-\overline z_0)}$. 

And indeed, this is what we see if we compute the eigenvalues 
of Hecke operators. Namely, the moduli of bundles of degree $0$ admitting a real structure 
consist of two circles $x=0$ and $x=\frac{1}{2}$, call them $S_0$ and $S_1$, with 
coordinate $y\in [0,1)$, swapped by $\tau$. However, the space of real bundles of degree $0$ is the union of their {\bf  double covers}  
$\widetilde S_0$ and $\widetilde S_1$. The reason is that a real bundle is a bundle admitting a real structure with a choice of an isomorphism $A: (s,\tau)(E)\to E$ such that $A\circ (s,\tau)(A)=1$, which is defined up to sign, and there is no canonical choice of this sign (one can check that it changes as we go around the circle). 
So the eigenbasis of Hecke operators is $\psi_n^+=(e^{\pi iny},e^{\pi iny})$ 
and $\psi_n^-=(e^{\pi iny},-e^{\pi iny})$, $n\in \Bbb Z$, with eigenvalues 
$e^{2\pi iny_0}$ and $-e^{2\pi iny_0}$. This also shows that we have a 2-dimensional space corresponding to each oper, 
which agrees with the fact that we have two extensions for each oper to a local system on $X/\tau$ (as $Z_2^\vee=\Bbb Z/2$). 

{\bf  2.} In the split case $s=1$ we need to consider homomorphisms $\pi_1(X/\tau)\to \Bbb Z/2\ltimes \Bbb C^\times$ that send $t$ to $\lbrace -1,1\rbrace$, so $b$ goes to any nonzero number, while $t^2$ maps to $1$. So the corresponding opers are $L=d+\phi$ where $\phi$ has integral period in the real direction, i.e. $\phi=2\pi in$, $n\in \Bbb Z$.   
So the Hecke eigenvalue is expected to be proportional to 
$e^{4\pi inx_0}=e^{2\pi in(z_0+\overline z_0)}$. 

And indeed, this is what we see. Namely, in this case bundles admitting a real structure 
form the circle $S$ defined by the equation $y=0$, with coordinate $x$. The circle $y=\frac{1}{2}$ consists of pseudo-real bundles, i.e., those for which $A\circ (s,\tau)(A)<0$ for any isomorphism $A: (s,\tau)E\to E$, so it does not contribute. Moreover, in this case the choice of $A$ such that $A\circ (s,\tau)(A)=1$ is unique up to isomorphism if exists. So the set of real bundles is $S$ (i.e., we don't get double covers) and the basis of eigenfunctions is $\psi_n=e^{2\pi inx}$, $n\in \Bbb Z$, with Hecke eigenvalue $e^{4\pi inx_0}$. Also extension of opers is unique and the space corresponding to each oper is $1$-dimensional.  
\end{example} 

\begin{remark} We can derive that every eigenvalue $\bbe(x,\overline x)$ of the Hecke operator 
$H_{\lambda x+s(\lambda)\overline x}$ is indeed of the form $\bbe_{\rho,\la}(x,\overline x)$ for some oper $\rho$ 
satisfying the reality condition of Conjecture \ref{gawi} (for
semisimple $G$) from Conjecture 5.5 of \cite{EFK2} (which is
proved in
Theorem 1.18 of {\it loc. cit.} for $G=PGL_n$ and $\lambda=\omega_1$).
Namely, this statement implies that each eigenvalue
of the Hecke operators is a unique (up to a scalar)
solution of the system of linear differential
equations $L \bbe=0,\overline L^*\bbe=0$, where $L$ runs over the holomorphic
differential operators from the annihilating ideal $I_{\lambda,\rho}$
introduced in Section 5 of \cite{EFK2}. Therefore, we obtain that this system
has a single-valued solution on $X$ invariant under $\tau$. 
The topological condition on the oper $\rho$ given in Conjecture \ref{gawi} 
should follow from this similarly to the argument of \cite{EFK1},
Corollary 1.19.  
\end{remark} 

\subsection{The case when $\tau$ has fixed points:  
genus $0$ with $m+2$ real ramification points}\label{fixpo}

\subsubsection{The untwisted case} 
We will now consider the case when $\tau$ has fixed points, which is more complicated. We restrict ourselves to $G=PGL_2$. We start with the genus zero case with ramification points considered in \cite{EFK3}. 

Let $t_0<t_1<...<t_m, t_{m+1}=\infty\in \Bbb R\Bbb P^1$ be the ramification points. 
Let $a,b\in \Bbb R$, $x=a+ib\in \Bbb C$, and let $H_{x,\overline x}$ be the Hecke operator 
from \cite{EFK3}, Example 3.30, obtained by averaging over Hecke modifications at $(x,\overline{x})$ 
using the lines $(s,\overline{s})$. The following lemma is straightforward. 

\begin{lemma}\label{l1} The eigenvalues 
$\bbe_k(x,\overline x)$ of $H_{x,\overline x}$ satisfy the equality 
$$
\bbe_k(x,\overline x)|_{x=a}=\beta_k(a)^2
$$
for $a\in \Bbb R$, where $\beta_k$ are the eigenvalues of $H_a$. 
\end{lemma}  

Let us now study the function $\bbe_k(x,\overline x)$ for $x\notin \Bbb R$. 
To do so, note that $\bbe=\bbe_k$ satisfies the oper equations 
$$
L\bbe=0,\ \overline L\bbe=0,
$$ 
where $L=L(\bmu_k)$ (using the notation of \cite{EFK3}, Subsection 4.4). 
Recall also that the points $t_j$ divide $\Bbb R\Bbb P^1$ into intervals $I_j=(t_{j},t_{j+1})$, and that in \cite{EFK3}, Subsection 4.7 we defined the functions $f_j,g_j$ on $I_j$. 
Recall that on $I_j$ we have $\beta(x)=f_j(x)$. Thus by Lemma \ref{l1} along the interval $I_j$ we have  
$$
\bbe(x,\overline x)=|f_j(x)|^2+\gamma_j{\rm Im}(\overline{f_j(x)}g_j(x)),\ \gamma_j\in \Bbb R,
$$
for $x$ on and above $I_j$. 

Let the function $g_j^*: I_j\to \Bbb R$ be defined by 
$$
g_j=b_jf_j-a_jg_j^*,
$$
where $a_j,b_j\in \Bbb R$ are as in \cite{EFK3}, Subsection 4.7. 
Then 
$$
\bbe(x,\overline x)=|f_j(x)|^2-a_j\gamma_j{\rm Im}(\overline{f_j(x)}g_j^*(x)).
$$
For an analytic function $h$ on $I_{j}$ let $h^{\rm a}$ be its analytic continuation from $I_{j}$ to $I_{j-1}$ along a path passing above $t_{j}$. Recall from \cite{EFK3}, Subsection 4.7, that 
\begin{equation}\label{eqq1}
f_{j-1}=if_{j}^{\rm a}+g_{j}^{\rm a},\ g_{j-1}^*=ig_{j}^{\rm a}.
\end{equation}
This yields 
$$
\bbe=|if_{j}^{\rm a}+g_{j}^{\rm a}|^2-a_{j-1}\gamma_{j-1}{\rm Im}((\overline{f_{j}^{\rm a}}+i\overline{g_{j}^{\rm a}})g_{j}^{\rm a})=
$$
$$
|f_{j}^{\rm a}|^2+(2-a_{j-1}\gamma_{j-1}){\rm Im}(\overline{f_{j}^{\rm a}}g_{j}^{\rm a})+|g_{j}^{\rm a}|^2-a_{j-1}\gamma_{j-1}|g_{j}^{\rm a}|^2.
$$
on $I_{j-1}$. 
The last two terms must cancel, so $a_{j-1}\gamma_{j-1}=1$.  
Thus we get 
$$
\bbe=|f_{j}^{\rm a}|^2+{\rm Im}(\overline{f_{j}^{\rm a}}g_{j}^{\rm a}).
$$
on $I_{j-1}$. But we also have 
$$
\bbe=|f_{j}^{\rm a}|^2+\gamma_{j}{\rm Im}(\overline{f_{j}^{\rm a}}g_{j}^{\rm a}).
$$
Thus for all $j$ we have $\gamma_j=1$, hence $a_j=1$ (i.e. the local system is {\bf balanced}, in agreement with \cite{EFK3}, Subsection 4.7).
Thus we obtain 

\begin{proposition}\label{l2} We have 
$$
\bbe(x,\overline x)=|f_j(x)|^2+{\rm Im}(\overline{f_j(x)}g_j(x))
$$
on and above $I_j$. 
\end{proposition} 

In particular, for 3 points this gives an explicit formula for the function $\bbe$ in terms of classical elliptic integrals (see \cite{EFK3}, Example 4.5).  

\begin{corollary}\label{c3} The one-sided normal derivative of $\bbe(x,\overline x)$ at the real line 
(with $x$ approaching from above) equals $\pi$. Thus 
$$
\bbe(a+ib,a-ib)=\beta(a)^2+\pi |b|+o(|b|),\ b\to 0. 
$$
In particular, $\bbe(x,\overline x)$ is continuous, but 
only one-sided differentiable on the real locus (excluding ramification points). 
\end{corollary} 

\begin{proof} It is easy to check that the normal derivative equals the Wronskian $W(f_j,g_j)$, which equals $\pi$, as explained in \cite{EFK3} (Proof of Proposition 4.25). 
\end{proof} 

\begin{remark}\label{anycurve} Note that the statement of Corollary \ref{c3} makes sense 
on any real curve near its real point (indeed $f_j,g_j$ are $-1/2$-forms, so their Wronskian is a function, and it makes sense to say that it equals $\pi$). Moreover, it holds for any real curve since it is a local statement 
and it holds in genus zero by Corollary \ref{c3}. It can also be checked by direct computation of the normal derivative. 
\end{remark} 

\subsubsection{The twisted case} \label{twica}

Consider now the twisted case with arbitrary twisting 
parameters $\lambda_j=-1+c_j, c_j\in i\Bbb R$. In this case 
the story is similar to the previous subsection and \cite{EFK3}, Subsection 4.7. Namely, fix $k$ and consider the function $\beta_k(x)$.  On the interval $I_j$, $0\le j\le m$, we have the following solutions of the oper equation $L(\bmu_k)\beta=0$: first of all,  
$$
f_j(x)=\beta_k(x)|_{I_j}\sim \delta_{kj}\Gamma^{\Bbb R}(c_j)(x-t_j)^{\frac{1-c_j}{2}}+
\delta_{kj}^{-1}\Gamma^{\Bbb R}(-c_j)(x-t_j)^{\frac{1+c_j}{2}},\ x\to t_j+, 
$$
and also 
$$
\widehat g_j(x)\sim \frac{\pi\delta_{kj}^{-1}}{c_j\Gamma^{\Bbb R}(c_j)}(x-t_j)^{\frac{1+c_j}{2}},\ x\to t_j+.
$$
The function $\widehat g_j$ is not real-valued on the real axis, however,
so let us look for a real-valued solution of the form 
$$
g_j=\widehat g_j+i\xi_j f_j,\ \xi_j\in \Bbb R. 
$$
A short calculation using that 
$$
\Gamma^{\Bbb R}(c)=\Gamma(c)\cos \frac{\pi c}{2}
$$
yields 
$$
\xi_j=\frac{\Lambda_j-\Lambda_j^{-1}}{\Lambda_j+\Lambda_j^{-1}},\quad
\Lambda_j:=e^{\frac{\pi ic_j}{2}},
$$ 
and the Wronskian $W(f_j,g_j)=W(f_j,\widehat g_j)$ equals $\pi$. 
 
Similarly, let $\widehat g_{j-1}^*(x)$ be the solution of the oper equation on $I_{j-1}$ 
of the form 
$$
\widehat g_{j-1}^*(x)=\frac{\pi\delta_{kj}^{-1}}{c_{j}\Gamma^{\Bbb R}(c_{j})}(t_{j}-x)^{\frac{1+c_{j}}{2}},\ x\to t_{j}-.
$$
and $g_{j-1}^*:=\widehat g_{j-1}^*+i\xi_{j-1}f_{j-1}$ be the corresponding real solution. 
Then  
$$
g_j=b_jf_j-a_jg_j^*,\ a_j,b_j\in \Bbb R, a_j\ne 0. 
$$
Also instead of \eqref{eqq1} we get 
\begin{equation}\label{eqq2}
\begin{pmatrix} f_{j-1}\\ g_{j-1}^*\end{pmatrix}=J_j
\begin{pmatrix} f_{j}^{\rm a}\\ g_{j}^{\rm a}\end{pmatrix},
\end{equation}
where 
$$
J_j:=\frac{\Lambda_j+\Lambda_j^{-1}}{2}
\begin{pmatrix} i & -\frac{(\Lambda_j-\Lambda_j^{-1})^2}{(\Lambda_j+\Lambda_j^{-1})^2}\\ 1 & i
\end{pmatrix}
$$
It follows that the monodromy of our oper in appropriate bases looks as follows:
\begin{equation}\label{locsy}
M(t_{j+1}-\to t_j+)=B_j:=\begin{pmatrix} 1& b_j\\ 0& -a_j\end{pmatrix},\
M^\pm(t_j+\to t_j-)=J_j^{\pm 1},
\end{equation}
where $M^\pm$ denotes the monodromy above and below the real axis, respectively. 
Also $W(f_{j-1},g_{j-1}^*)=-\pi$, hence $a_j=1$. 

When $j=m+1$, the formulas are the same, except that $x-t_j$ is replaced by $-1/x$ and 
$J_j^{\pm 1}$ is replaced by $-J_j^{\pm 1}$. 

We thus obtain a deformation of the theory of balanced local systems and opers 
described in \cite{EFK3}, Subsection 4.7. 
Namely, similarly to \cite{EFK3}, given a sufficiently generic 2-dimensional 
local system $\nabla$ on $X$ with ramifications at $t_j$ and regular local monodromies with eigenvalues $-\Lambda_j^{\pm 2}$, it can be written (generically in two different ways) in the form \eqref{locsy}, where $a_j,b_j\in \Bbb C$ and $B_j$ must satisfy the equations 
\begin{equation}\label{defoo}
\prod_{j=0}^{m+1}J_j B_j=-1,\ \prod_{j=0}^{m+1}J_j^{-1} B_j=-1,
\end{equation} 
which are deformations of equations (4.7) of \cite{EFK3} (the total monodromy around the circle above and below the real axis is trivial). 

Define a $\bLa$-{\bf balancing} of $\nabla$ 
to be an isomorphism (considered up to scaling) of $\nabla$ with a local system \eqref{locsy} such that $a_j=1$ for all $j$. In this case, as in \cite{EFK3}, the two equations in \eqref{defoo} are, in fact, equivalent, since $SJ_j^{-1}B_jS^{-1}=J_jB_j$, where $S:=\begin{pmatrix}1 & 2i\\ 0 & 1\end{pmatrix}$.
We call such a local system $\nabla$ $\bLa$-{\bf balanced}
if it is equipped with a balancing. Generically a local system admits 
at most one balancing, as in \cite{EFK3}. 

We obtain the following  
analog of \cite{EFK3}, Proposition 4.25 and Theorem 4.29. 
Denote by $\mathcal B_{\bLa}$ the set 
of $\bLa$-balanced opers. 

\begin{theorem}\label{twii} The spectral opers for Hecke operators $H_{x,{\rm full}}$ 
are $\bLa$-balanced with balancing defined by the eigenvalue $\beta_k(x)$, and $b_i\in \Bbb R$, as in \cite{EFK3}, Proposition 4.7. Thus the spectrum $\Sigma_{\bLa}$ of the Hecke operators 
is a subset of $\mathcal B_{\bLa}$.  
\end{theorem} 

Moreover, as in \cite{EFK3}, we expect that these sets are, in fact, equal, and 
can show this for $4$ and $5$ points. 
 
\subsection{The case when $\tau$ has fixed points: real and quaternionic ovals}\label{requatco}

Now suppose $X$ is an arbitrary real curve and $C$ is a connected component (oval)
of $X(\Bbb R)$. Given a real $PGL_2$-bundle $P$ on $X$ (with respect to the split form of the group), recall that its fiber $P_x$ is a $PGL_2(\Bbb C)$-torsor. So for a point $x\in C$, the real structure on $P$ defines a map $A: P_x\to P_x$ such that $Ag=\overline g A$ for $g\in PGL_2(\Bbb C)$, and $A^2=1$. Pick $p\in P_x$, then $A(p)=bp$ for a unique $b\in PGL_2(\Bbb C)$, so $A(bp)=\overline{b}A(p)=\overline bbp$, thus $\overline bb=1$. Replacing $p$ with $q:=gp$, we get 
$$
A(q)=A(gp)=\overline gA(p)=\overline gbp=\overline gbg^{-1}q.
$$ 
Thus $b$ is well defined up to $b\mapsto \overline{g}bg^{-1}$. It is easy to show that such $b$ fall into two orbits of this action -- that of $b=1$ (which we call {\bf  real}) and that of
$b=\begin{pmatrix}0 &1\\ -1 &0\end{pmatrix}$
(which we call {\bf  quaternionic}). Namely, if $b_*$ is a lift of $b$ to $SL_2(\Bbb C)$ then $\overline{b_*}b_*=1$ or $-1$, and this is what determines the type of $b$ (real for $1$, quaternionic for $-1$). In the language of Subsections \ref{formpt}, \ref{pblocfie}, 
$P$ is real at $x$ if the associated real form $G^\sigma$ of $G$ is split and 
quaternionic if $G^\sigma$ is compact. 

It is clear that the type of $P$ at $x$ is independent on $x$ as it varies along $C$. 
Thus given $P$, on every oval $C_i$ of $X(\Bbb R)$, $i=1,...,r$, $P$ is either real or quaternionic. So the manifold ${\rm Bun}_{G,s}^\circ(X,\tau)$ splits into $2^r$ disconnected parts according to the type of $P$ at each $C_i$ (some of which can be empty). In fact, how many of them (and which ones) are non-empty is specified on p.18 of \cite{BGH} and references therein. 

\begin{remark} 1. The analog of this for general groups is as follows (see Subsections
\ref{formpt},\ref{pblocfie}).
By Subsection \ref{formpt}, every real $G$-bundle $P$ on $X$ and every $x\in X(\Bbb R)$ defines a real form of $G$ in the inner class $C(s)$ which is continuous, hence locally constant, with respect to $x$ (see also \cite{GW}, Section 6). So each component $C_i$ of $X(\Bbb R)$ carries a real form $G^{\sigma_i}$ of $G$ in  $C(s)$ -- the type of $P$ at $C_i$. For example, as explained above, if $G=SL_2$ then $C_i$ is real if the form of $G$ attached to $C_i$
is $SL_2(\Bbb R)$ and quaternionic if it is $SU_2$. 

2. On components containing tame ramification points, in the untwisted setting 
this form has to be quasi-split, as we need a real Borel subalgebra to define parabolic structures. So in particular for $G=PGL_2$ all contours containing ramification points must be real. 
More generally, if we consider parabolic structures for an arbitrary parabolic subgroup 
${\rm P}$ of $G$, the corresponding ramification points can occur only on components 
for which the corresponding real group contains a form of ${\rm P}$. 

3. More generally, following Subsections \ref{ram1},\ref{ram2}, at ramification points $p$ one can place unitary representations $\pi_i$
of the complex group $G$ if $\tau(p)\ne p$ and 
of the real form $G_i$ if $p\in C_i$ (the twisted setting). In this case we no longer have the restriction that $G_i$ should be quasi-split if $p\in C_i$ (e.g. see Subsection \ref{comgr}). The untwisted setting of (2) then corresponds to taking $\pi_i$ to be the spherical principal series representations with central character of $-\rho$, consisting of $L^2$-half-densities on the flag manifold, which requires the groups $G_i$ to be quasi-split.
\end{remark} 

\subsection{Real ovals, separating real locus}
Now consider the case of $PGL_2$-bundles on an arbitrary real curve $X$ without ramification points, and let us characterize the part of the spectrum coming from bundles for which all ovals are real. Let the set $X(\Bbb R)$ of fixed points of $\tau$ be the union of ovals $C_1,...,C_n\subset X$.
Assume first that these ovals cut $X$ into two pieces $X_+,X_-$ swapped by $\tau$. Then the behavior of the Hecke eigenvalue 
$\bbe(x,\overline x)$ gives us conditions which allow us to pinpoint opers $L$ (with real coefficients) that can occur in the spectrum of Hecke operators. 

Namely, first of all, the eigenvalue $\beta(x),x\in X(\Bbb R)$ of the Hecke operator $H_x$ is a solution of the oper equation $L\beta=0$ periodic along $C_j$. So (assuming this eigenvalue is not identically zero), the local system $\rho_L$ must satisfy

\vskip .05in

{\bf  Condition 1.} The monodromies of $\rho_L$ around $C_j$ are unipotent. 

\vskip .05in

Indeed, this is necessary for the existence of the periodic solution $\beta(x)$, since the monodromy  lies in $SL_2$. 

Also, since the eigenvalue $\bbe(x,\overline x)$ of $H_{x,\overline x}$
is single-valued on $X_+$ and satisfies the oper equations $L\bbe=\overline L\bbe=0$, 
we see that $\rho_L$ must also satisfy 

\vskip .05in

{\bf  Condition 2.} The monodromy representation $\rho_L: \pi_1(X_+)\to SL_2(\Bbb C)$ lands in $SL_2(\Bbb R)\cong SU(1,1)$, up to conjugation. 

\vskip .05in

To write this condition more explicitly, fix a base point $x_0\in X_+$ and paths $p_j$ from $x_0$ to $c_j\in C_j$. This defines elements $\delta_j:=p_j^{-1}C_jp_j\in \pi_1(X,x_0)$, where we agree that $C_j$ begins and ends at $c_j$ and is oriented so that when we travel around it, $X_+$ remains on the left. Let ${\rm g}_+$ be the genus of $X_+$ and $A_k,B_k,1\le k\le {\rm g}_+$ 
be the $A$-cycles and $B$-cycles of $X_+$. Then $\delta_j,A_k,B_k$ generate $\pi_1(X_+,x_0)$ with defining relation 
$$
\prod_{k=1}^{{\rm g}_+}[A_k,B_k]\prod_{j=1}^n\delta_j=1
$$ 
(for a suitable choice of the paths $p_j$). Then Condition 2 is equivalent to the condition that 
there exists a basis in which 
the matrices $\rho_L(\delta_j),\rho_L(A_k),\rho_L(B_k)$ are real for all $j,k$. 

\begin{definition} 
We will say that a local system $\rho$ on $X_+$ (not necessarily an oper) satisfying Conditions 1, 2 is {\bf balanced}. 
\end{definition} 

Let us now reformulate Condition 2 on the local system $\rho_L$ attached to the oper $L$ in a more analytic language, assuming that Condition 1 holds for $\rho_L$.  By Condition 1, for each $j$ we have a nonzero real periodic solution of the equation $Lf=0$ on $C_j$, call it $f_j(x)$. We also have the other real solution $g_j(x)$ which changes by a multiple of $f_j$ when we go around $C_j$ such that the Wronskian $W(f_j,g_j)=\pi$.\footnote{Note that if $\rho(C_j)\ne 1$ then $f_j$ is uniquely defined up to scaling, and once it is chosen, $g_j$ is uniquely defined up to adding a real multiple of $f_j$.} We will use the notation $f_j$, $g_j$ also for the analytic continuations 
of $f_i,g_j$ to a neighborhood of $I_j$. Motivated by Proposition \ref{l2}, introduce

\vskip .05in

{\bf  Condition 2a.} There is a single-valued solution $\bbe$ of the system 
$$
L\bbe=0,\ \overline L\bbe=0
$$ 
on $X\setminus \cup_j C_j$ such that near each $C_j$, we have 
$$
\bbe(x,\overline x)=\varepsilon_{j1} |f_j(x)|^2\pm \varepsilon_{j2} {\rm Im}(\overline{f_j(x)}g_j(x))
$$
for a suitable choice of $f_j,g_j$ with $W(f_j,g_j)=\pi$ (unique up to sign) and $\varepsilon_{j1},\varepsilon_{j2}=\pm 1$, where the sign in front of the second summand is $+$ if $x$ is above $C_j$ and $-$ if $x$ is below $C_j$.\footnote{Note that the function ${\rm Im}(\overline{f_j(x)}g_j(x))$ does not depend on the choice of $g_j$, is single-valued in the neighborhood of $C_j$, and vanishes on $C_j$. Hence $\bbe(x,\overline x)$ is continuous but only one-sided differentiable on $C_j$.}  

\begin{proposition} If Condition 1 holds and $\rho_L(C_j)\ne 1$ for all $j$ 
then Condition 2a is equivalent to Condition 2. 
\end{proposition} 

\begin{proof} Since $\bbe$ is a single-valued solution of the oper equations $L\bbe=0,\overline L\bbe=0$, Condition 2a implies Condition 2. To prove the converse, note that Condition 1 implies  that there exist bases $\lbrace f_j,g_j\rbrace$ of the fibers of the local system $\rho_L$ at the points $c_j$ in which $\rho_L(C_j)=\begin{pmatrix}1 & \lambda_j\\  0 & 1\end{pmatrix}$ for $\lambda_j\in \Bbb R$, and Condition 2 implies that on these fibers 
there are nondegenerate Hermitian forms invariant under $\rho_L(C_j)$ 
and compatible with the operators $\rho_L(p_kp_j^{-1})$, and $\det \rho_L(p_kp_j^{-1})=1$. 

Now, nondegenerate Hermitian forms in two variables $X,Y$ invariant under the matrix $\begin{pmatrix}1 & \lambda\\  0 & 1\end{pmatrix}$ for nonzero $\lambda\in \Bbb R$ are of the form 
$p|X|^2+q{\rm Im}(\overline XY)$, where $p,q\in \Bbb R, q\ne 0$. 
This implies that the Hermitian form at $c_j$ in the basis $f_j,g_j$ has the form 
$p_j|X|^2+q_j{\rm Im}(\overline XY)$, where $p_j,q_j\in \Bbb R$, $q_j\ne 0$. We can now renormalize $f_j,g_j$ by reciprocal positive constants  
to make sure that $p_j=\pm 1$. 

It remains to show that all $q_j$ are the same up to sign, then we can renormalize $\bbe$ to 
make sure that $q_j=\pm 1$, and construct the solution $\bbe$ satisfying Condition 2a
from the invariant Hermitian form on the fibers of $\rho_L$. But this follows from the equality $\det \rho_L(p_kp_j^{-1})=1$. 
\end{proof} 

\begin{remark} 1. The continuation of $g_j$ around $C_j$ gives $g_j+\lambda_j f_j$, 
and the real numbers $\lambda_j$ attached to the components $C_j$ do not depend on any choices and are invariants of a balanced oper.  

2. If the representation $\rho_L$ of $\pi_1(X_+)$ is irreducible then the solution $\bbe$ satisfying Conditions 1 and 2a is unique up to sign if exists. 
\end{remark}

The above discussion implies 

\begin{proposition} Every oper appearing in the spectrum of Hecke operators is balanced. 
\end{proposition} 

There is, however, another condition satisfied by spectral opers. 
Namely, let us say that a balanced oper $L$ is {\bf positive} 
if there is a solution $\bbe$ satisfying Condition 2a 
with $\varepsilon_{j1}=\varepsilon_{j2}=1$ for all $j$. 

\begin{proposition} Every oper appearing in the spectrum of Hecke operators is positive.
\end{proposition} 

\begin{proof} Since $\bbe_k(x,x)=\beta_k(x)^2$ for $x\in X(\Bbb R)$,  
$\varepsilon_{j1}$ are all the same, so can be assumed to be $1$.  
Then it follows from Remark \ref{anycurve} that 
we also have $\varepsilon_{j2}=1$ for all $j$. 
\end{proof} 

\begin{conjecture} The spectrum of Hecke operators is labeled by positive balanced opers, possibly with finitely many eigenvalues corresponding to the same oper. 
\end{conjecture} 

\begin{example} Let $X$ be of genus $2$ with $X(\Bbb R)$ having 3 components, $C_1,C_2,C_3$ that cut $X$ into two trinions $X_+,X_-$. Thus ${\rm g}_+=0$, so $\pi_1(X_+)$ is generated by $\delta_j$, $j=1,2,3$, with defining relation  $\delta_1\delta_2\delta_3=1$. So Condition 1, saying that all $A_i:=\rho(\delta_i)$ are unipotent, implies that they all commute, as $A_1A_2A_3=1$. 
In spite of having three equations (${\rm Tr}\rho(\delta_i)=2, i=1,2,3$), 
one can show that the space of such (real) local systems on $X$ is of codimension $2$ (i.e., 4-dimensional over $\Bbb R$) so it is {\bf  not} a complete intersection. However, 
Condition 2 provides one more equation to get a 3-dimensional real manifold. 
Therefore if appropriate transversality holds, then Conditions 1, 2 and  the oper condition imply the 
discreteness of the spectrum. Namely, isomorphism classes of (nontrivial) unipotent representations of the group $\pi_1(X_+,x_0)=F_2$ are labeled by one complex parameter $\kappa$, which lives in $\Bbb C\Bbb P^1$ (namely, $\rho(\delta_2)=\rho(\delta_1)^\kappa$). So Condition 2
just tells us that $\kappa\in \Bbb R\Bbb P^1$.  

Let us write down a formula for $\bbe(x,\overline x)$ in this case, for $\rho=\rho_L$. 
If Condition 1 holds then we have a global holomorphic solution 
$f(x)$ of the equation $Lf=0$ on $X_+$, generically unique up to scaling. We can normalize it to be real 
on $C_1$. We also have another holomorphic solution
$g(x)$ on $X_+$ which is changed by a multiple of $f(x)$ when we go around cycles, which we can normalize so that $W(f,g)=\pi$. 
Then $g$ is uniquely determined up to adding a multiple of $f$. 
We can make sure that $g$ is real on $C_1$, then the remaining freedom is adding to $g$ a real multiple of $f$. Now consider the function 
$$
\bbe(x,\overline x):=|f(x)|^2+{\rm Im}(\overline{f(x)}g(x)). 
$$
Note that this function is independent on the choice of $g$, and is single-valued since $\kappa$ is real. This function is the eigenvalue of $H_{x,\overline x}$ 
when $f$ is normalized so that $f(x)|_{C_1}=\beta(x)$ (then this will also hold on $C_2$ and $C_3$, up to sign). Thus we see that in this case every balanced oper is automatically positive. 
\end{example} 

\begin{example} Suppose $X$ has genus ${\rm g}\ge 3$ with $X(\Bbb R)$ 
having ${\rm g}+1$ components $C_1,...,C_{{\rm g}+1}$; thus ${\rm g}_+=0$ so $\pi_1(X_+)$ 
is generated by $\delta_j$ with $\prod_{j=1}^{{\rm g}+1}\delta_j=1$. Then Condition 1 imposes ${\rm g}+1$ constraints on the local system: we have that $A_j:=\rho(\delta_j)$ are unipotent for $j=1,...,{\rm g}+1$. Moreover, now this actually defines a complete intersection (unlike the previous example, which is a degenerate case). Once this condition is imposed, our representation $\rho$ of $\pi_1(X_+,x_0)=F_{\rm g}$ is a point of the {\bf  unipotent $SL_2$-character variety} $\M_{0,{\rm g}+1}^{\rm unip}$ for the sphere with ${\rm g}+1$ holes, which has (complex) dimension $2({\rm g}-2)$. Thus the condition that this point is real is $2({\rm g}-2)$ real equations. So altogether we get 
${\rm g}+1+2({\rm g}-2)=3{\rm g}-3$ real equations, i.e. if appropriate transversality holds then we should get a real submanifold of middle dimension $3{\rm g}-3$ in the $6{\rm g}-6$-dimensional manifold of local systems, as needed for discrete spectrum.  
\end{example}  

\begin{example} More generally, suppose $X(\Bbb R)$ is the union of $C_1,...,C_n$ where $n\le {\rm g}+1$. Then $X_+$ has genus ${\rm g}_+=\frac{{\rm g}+1-n}{2}$ (so ${\rm g}+1-n$ must be even). So Condition 1 gives us $n$ real equations, and then we end up in the unipotent character variety $\M_{ \frac{{\rm g}+1-n}{2},n}^{\rm unip}$, which has (complex) dimension 
$d=3({\rm g}+1-n)-6+2n=3{\rm g}-3-n$. So altogether we again get $3{\rm g}-3-n+n=3{\rm g}-3$ real equations, as needed for discrete spectrum. 
\end{example} 

\subsection{Real ovals, non-separating real locus}

Now suppose that $X(\Bbb R)$ still consists of $n$ circles $C_1,...,C_n$ of real type but now is non-separating. 
To handle this case, let $\Sigma$ be a connected non-orientable surface with $n$ holes of Euler characteristic $\chi$, and let $\M_{\chi,n}^{{\rm unip},-}$ be the corresponding unipotent character variety. 
Note that $\Sigma$ can be obtained by gluing $s=1$ or $2$ M\"obius strips into an orientable surface $\Sigma_+$ with $n+s$ holes and the same Euler characteristic $\chi$. We have $2-2{\rm g}(\Sigma_+)-n-s=\chi$, so ${\rm g}(\Sigma_+)=1-\frac{n+s+\chi}{2}$, and $s$ is such that $n+s+\chi$ is even. 
Thus $\M_{{\rm g}(\Sigma_+),n+s}^{\rm unip}$ has dimension $-3(n+s+\chi)+2(n+s)=-n-s+3\chi$. 
So $\dim \M_{\chi,n}^{{\rm unip},-}=-n-3\chi$ (since we no longer have unipotency condition at $s$ of the $n+s$ holes, where we glue in a M\"obius strip). 

Now let $\Sigma=X/\tau$. This is a non-orientable surface of Euler characteristic $\chi=1-{\rm g}$ and $n$ holes. So $\dim \M_{\chi,n}^{{\rm unip},-}=-n-3\chi=3{\rm g}-3-n$. Thus 
Condition 1 gives $n$ equations, and the real locus in $\M_{\chi,n}^{{\rm unip},-}$ another 
$3{\rm g}-3-n$ equations, so altogether we get $3{\rm g}-3$ equations, again as needed. 

\subsection{Quaternionic ovals}\label{quatov}

For a quaternionic oval $C_j$ the Hecke operator $H_x$ for $x\in C_j$ 
is not defined, so the function $\beta(x)$ is not defined either. As a result, 
it is not hard to show that 
$$
\lim_{x\to C_j}H_{x,\overline x}=0
$$ 
and thus $\bbe(x,\overline x)=0$ 
for $x\in C_j$. More specifically, the function $\bbe(x,\overline x)$
near $C_j$ has the form 
\begin{equation}\label{quateq} 
\bbe(x,\overline x)=\pm {\rm Im}(\overline{f_j(x)}g_j(x)),
\end{equation} 
where as before $f_j,g_j$ are real solutions of the oper on $C_j$ such that $W(f_j,g_j)=\pi$, so 
we have 
$$
\bbe(x,\overline x)\sim \pi |b|
$$
where $b$ is the distance from 
$x$ to the oval (this makes sense because $\bbe$ is not a function but actually
a $-1/2$-density, i.e., $|b|$ is really $|db/b|^{-1}$). However, 
the normalization of $f_j$ is now not fixed, so we are free to multiply $f_j$ by a nonzero real scalar and divide $g_j$ by the same scalar. 

Also, we no longer have a condition that monodromy around $C_j$ is unipotent. It only has to have real eigenvalues $\mu_j^{\pm 1}$, so it can preserve the indefinite Hermitian form defined by $\bbe$ (see \eqref{quateq}), and $f_j,g_j$ are the corresponding eigenvectors.  

\subsection{Conditions with ramification points}

In presence of ramification points on $C_j$ (in which case in absence of twists $C_j$ is necessarily of real type), the story should be the same, except that by monodromy along $C_j$ we should mean monodromy ``in the sense of principal value", as explained in \cite{EFK3}, Subsection 4.7; i.e., when we continue through a ramification point, we take minus the half-sum of upper and lower analytic continuation. So for genus $0$ we recover exactly the answer from \cite{EFK3}, Subsection 4.7. In this case the real locus divides $\Bbb C\Bbb P^1$ into two disks, which are simply connected, so we don't have any conditions similar to Subsection \ref{nofix}, and the balancing conditions on $\rho_L$ can be formulated solely in terms of the neighborhood of the real locus. This is exactly what happens in \cite{EFK3}, Subsection 4.7.

\section{Analytic Langlands correspondence and Gaudin model} 

In this section we discuss the Gaudin model and its
  generalizations and relate these models to various settings of the
  analytic Langlands correspondence on $\pone$ with ramification
  points over $\Bbb R$ and $\Bbb C$.

Initially, the Gaudin model associated to a simple Lie algebra $\g$
was defined for the tensor products of finite-dimensional
representations of $\g$. But in fact the Gaudin Hamiltonians (which we
have already encountered in Subsection \ref{Differential equations for
  Hecke operators} in the case $\g=\sw_2$) give rise to
well-defined
commuting operators on the tensor product of any representations of
$\g$. If this tensor product has a weight space decomposition with
finite-dimensional weight spaces, then since these subspaces are
preserved by the Gaudin Hamiltonians, the corresponding spectral
problem is clearly well-defined. This is the case, for example, when
all representations are of highest weight or lowest weight.

If this is not the case, the spectral problem may still be well-defined if there is a natural Hilbert space structure on a completion of this tensor product and the Gaudin Hamiltonians can be extended to self-adjoint strongly commuting operators on it. This conjecturally happens when the representations of $\g$ come from tempered unitary representations of a connected real Lie group $G(\Bbb R)$
whose complexified Lie algebra is $\g$; for example, this happens for
representations of the unitary principal series of $SL_2(\Bbb R)$. The
traditional methods of Bethe Ansatz can no longer be used in this
case. But here we get into the setting of the analytic Langlands
correspondence for $G=SL_2$, $X=\Bbb P^1$, and $F=\Bbb R$ with real
ramification points discussed in Section \ref{R} (with the Gaudin
Hamiltonians being the Hitchin Hamiltonians). Hence we can use the
results of Section \ref{R} to describe the spectrum of the Gaudin Hamiltonians.

In fact, we will show that even the original case of the tensor
product of finite-dimensional representations of $\g$ can be
interpreted in the framework of the analytic Langlands correspondence
(namely, it appears in the quaternionic case discussed in Subsection
\ref{quatov}). Applying our results, we obtain a new interpretation of
the description of the spectrum of the Gaudin Hamiltonians in this
case in terms of monodromy-free opers \cite{F1,R}. This description is
closely related to the Bethe Ansatz in the Gaudin model. We discuss
all this in Subsections \ref{monfre}--\ref{Gcase} and
\ref{exbethe}--\ref{baxop}.

In Subsection \ref{comwei} we consider the case of the tensor product
of the infinite-dimensional contragredient Verma modules with
arbitrary highest weights. In this case the weight subspaces of the
tensor product are finite-dimensional and the spectral problem for the
Gaudin Hamiltonians is well-defined. It is natural to expect that the
spectrum is given by opers whose monodromy is contained in a Borel
subgroup of $G^\vee$, the Lie group of adjoint type associated to the
Lie algebra $\g^\vee$. In a follow-up paper \cite{EF} we
  intend to prove this result using the tools of the present paper.

In Subsection \ref{disser} we interpret this result, in the case when
the highest weights satisfy a certain reality condition, as a
description of the spectrum for the tensor product of the holomorphic
discrete series representations of $G(\Bbb R)$. In
  Subsection \ref{chirlan} we introduce chiral versions of the Hecke
  operators acting on the tensor product of contragredient Verma
  modules.

Next, in Subsection \ref{infdi} we discuss the infinite-dimensional
case. First, we describe the spectrum for the tensor product of
unitary principal series representations in terms of balanced
opers. This is essentially the statement of Theorem \ref{twii}. This
case is very interesting because we cannot use the ordinary Bethe
Ansatz method (since these representations don't have highest weight
vectors). We also comment on the case of a tensor product of discrete
series representations involving both holomorphic and anti-holomorphic
ones.

In Subsection \ref{double} we interpret the analytic Langlands
correspondence for $\pone$ with parabolic structures over $\Bbb C$
(discussed in Subsection \ref{Differential equations for Hecke
  operators}) as a ``double'' of the Gaudin model.

In all of these settings, our description of the spectrum relies on
the existence of the Hecke operators, which commute with the Gaudin
Hamiltonians and satisfy differential equations (the universal oper
equations). These equations can be used to describe the analytic properties of the opers encoding the possible eigenvalues of the Gaudin Hamiltonians.

Interestingly, the role of the Hecke operators is played by the Gaudin
model analogues of Baxter's $Q$-operators or closely related operators
(see Subsection \ref{baxop}). This allows us to regard the
  generalized Bethe Ansatz method and Baxter's $Q$-operators as
  arising from special cases of the tamely ramified analytic Langlands
  correspondence in genus $0$. We also discuss a $q$-deformation of
  this story, which has to do with the quantum integrable
    models of XXZ type associated to the quantum affine algebra
    $U_q(\ghat)$, in Subsections \ref{qdefo} and \ref{qdefo1}.

\subsection{The Gaudin model and monodromy-free opers}\label{monfre}
Let $\g$ be a simple Lie algebra over $\Bbb C$ and $G$ the connected
simply-connected algebraic group with the Lie algebra $\g$. Let
$\g=\n_+\oplus \h\oplus \n_-$ be a triangular decomposition of $\g$
with Weyl group $W\subset {\rm Aut}(\h)$. Let $\Lambda_G^+\subset
\h^*$ be the set of dominant integral weights of $\g$ (and $G$), and
let $V_{\la}$ be the finite-dimensional irreducible representation of
$G$ with highest weight $\la\in \Lambda_G^+$.

Let
$$
\{ t_i \} := \{ t_0,t_1,\ldots,t_{m+1}=\infty \}
$$
be a collection of $(m+2)$ distinct points on $\pone$, with the last
point $t_{m+1}$ identified with the point $\infty$ (with respect to a
once and for all chosen global coordinate $x$ on $\pone$). Fix
$\lambda_i\in \Lambda_G^+$, $0\le i\le m+1$, and let
\begin{equation}\label{spaceH} 
\mathcal H := (V_{\la_0}\otimes...\otimes V_{\la_{m+1}})^{\g} \simeq
(V_{\la_0}\otimes...\otimes V_{\la_{m}})^{\n_+}[\la_{m+1}^*].
\end{equation} 
Thus, $\mathcal H$ is the subspace of singular vectors of weight
$\lambda_{m+1}^*=-w_0(\la_{m+1})$ in $V_{\la_0}\otimes...\otimes
V_{\la_{m}}$, where $w_0\in W$ is the maximal element.

The space $\mathcal H$ carries an action of a commutative subalgebra ${\mc G}\subset (U(\g)^{\otimes m+1})^\g$ of the {\bf generalized Gaudin Hamiltonians} introduced in \cite{FFR}
(see also \cite{F:icmp,F1}).\footnote{This algebra is called the {\bf Bethe algebra} or the {\bf Gaudin algebra}.} 
The algebra ${\mc G}$ includes the original (quadratic) Gaudin
Hamiltonians
\begin{equation}    \label{Hi}
G_i := \sum_{j \neq i} \sum_a \frac{J^{a(i)}J_a^{(j)}}{t_j-t_i},
\end{equation}
where $\{ J^a \}$ and $\{ J_a \}$ are two bases of the Lie algebra
$\g$ dual to each other with respect to the normalized non-degenerate
invariant bilinear form. In the case $\g=\sw_2$, the algebra ${\mc G}$ is
generated by the $G_i$'s. For groups of rank greater than 1, there are
also higher Gaudin Hamiltonians.

Joint eigenvectors and eigenvalues of the algebra ${\mc G}$ in ${\mc
  H}$ have been constructed explicitly in \cite{FFR} generalizing the
classical Bethe Ansatz method in the case $\g=\sw_2$ (in \cite{RV} an
alternative proof was given that these vectors are eigenvectors of the
$G_i$'s). It has been proved in \cite{SV} (see also \cite{F:icmp})
that for $\g=\sw_2$ these eigenvectors form a basis in ${\mc H}$ if the
collection $\{ t_i \}$ is generic. But for other groups this is not
always the case. The reasons for this are explained in \cite{F1},
Section 5.5. An explicit counterexample in the case $\g=\sw_3$ has
been given in \cite{MV}.

However, an alternative description of the spectrum of the algebra
${\mc G}$ on the space ${\mc H}$, which does not rely on explicit
formulas for the eigenvectors, was conjectured in \cite{F1} (following
\cite{F:icmp}) and proved in \cite{R}. Namely, let $G^\vee$ be the
Langlands dual group of $G$ (thus, $G^\vee$ is the Lie group of
adjoint type with associated with the Lie algebra $\g^\vee$ which is
Langlands dual to $\g$). There is a bijection between the joint
spectrum of the algebra ${\mc G}$ of generalized Gaudin Hamiltonians
on ${\mc H}$ given by formula \eqref{spaceH} (without multiplicities)
and the set of {\bf monodromy-free} $G^\vee$-opers on $\pone$ with
regular singularities at the points $t_i$, with {\bf residues}
$\varpi(-\lambda_i-\rho) \in \h^*/W$, where $\varpi$ is the projection
$\h^*\to \h^*/W$.\footnote{The notion of residue was introduced in
\cite{BD}, Section 3.8.11; see also \cite{F:loop}, Subsection 9.1.}
(In the case of $\g=\sw_n$, a similar result for the spectrum of
a certain deformation of ${\mc G}$ follows from
\cite{MTV}.)  Moreover, it is shown in \cite{R} that for a generic
collection $\{ t_i \}$ the algebra ${\mc G}$ is diagonalizable on
${\mc H}$ and its spectrum is simple. The precise statement is given
in Theorem \ref{BAM} below.

\subsection{The Gaudin model for $\g=\sw_2$}\label{sl2case}

Consider first the Gaudin model for $\g=\sw_2$. Note that
$w_0(\la)=-\la$, so $\lambda^*=\la$ for all weights $\la$. We identify
the dominant integral weights $\la_i$ with non-negative integers and
define $n$ by the formula
\begin{equation}    \label{weight}
2n = \sum_{i=0}^m \lambda_i - \la_{m+1}.
\end{equation}
 For the space
\begin{equation}    \label{spst}
{\mc H} = \left( \bigotimes_{i=0}^m V_{\la_i}
\right)^{\n_+}[\la_{m+1}]
\end{equation}
to be non-zero, $n$ must be a non-negative integer.

The algebra ${\mc G}$ of Gaudin Hamiltonians is in this case generated
by the $G_i$'s given by formula \eqref{Hi}. The original formulation
of Bethe Ansatz for diagonalization of these operators is the
following. Given a collection $\bold w = \{ w_1,\ldots,w_n \}$ of
distinct complex numbers such that $w_j \neq t_i$ for all $i$ and $j$,
define the {\bf Bethe vector} by the formula
\begin{equation}    \label{bvec}
  v_{\bold w} := f(w_1) \ldots f(w_m) v,
\end{equation}
where
\begin{equation}    \label{fw}
f(w) := \sum_{i=0}^m \frac{f_i}{w-t_i}
\end{equation}
and $v$ is the tensor product of the highest weight vectors of the
representations $V_{\la_i}$.

The following is the system of {\bf Bethe Ansatz equations} on the
numbers $w_j, j=1,\ldots,n$:
\begin{equation}    \label{BAE}
\sum_{i=0}^m \frac{\la_i}{w_j-t_i} = \sum_{s \neq j} \frac{2}{w_j-w_s}, \qquad
  j=1,\ldots,n.
\end{equation}

\begin{theorem}[\cite{SV}]    \label{SV}
  For a generic collection $\{ t_i \}$, the vectors $v_{\bold w}$ with
  ${\bold w} = \{ w_1,\ldots,w_n \}$ satisfying the system \eqref{BAE}
  form an eigenbasis of ${\mc H}$. The eigenvalue $\mu_i$ of $G_i$ on
  $v_{\bold w}$ is given by the formula
\begin{equation}    \label{mui}
  \mu_i = \la_i\left(\sum_{k\ne i}\frac{\la_k}{2(t_i-t_k)}-\sum_{j=1}^n
  \frac{1}{t_i-w_j}\right).
\end{equation}
\end{theorem}

This result is referenced as {\bf completeness} of Bethe Ansatz for
$\g=\sw_2$.

Analogs of Bethe vectors have been constructed for an arbitrary Lie
algebra $\g$ \cite{FFR,RV}. Unfortunately, they do not give an
eigenbasis for a general $\g$ even for a generic collection $\{ t_i
\}$, so Bethe Ansatz is incomplete; a counterexample has been found
already for $\g=\sw_3$ \cite{MV}.

Luckily, there is an alternative approach to describing the joint
spectrum of the Gaudin Hamiltonians on the space ${\mc H}$ given by
\eqref{spaceH}. It uses a realization \cite{FFR,F:icmp,F1} of the
algebra of Gaudin Hamiltonians as the quotient of the center of the
completed enveloping algebra of the affine Kac-Moody algebra at the
critical level and its isomorphism with the algebra of functions on
the space of $G^\vee$-opers on the punctured disc \cite{FF,F:loop}. We
will now explain this approach in the case of $\g=\sw_2$ and connect it to
the Bethe Ansatz discussed above.

For $\g=\sw_2$, we have $G=SL_2$ and $G^\vee = PGL_2$. The Bethe
Ansatz equations can be interpreted in terms of monodromy-free
$PGL_2$-opers on $\pone$ as follows.

Recall that a $PGL_2$-oper is a second order-differential operator
acting from $K^{-{1\over 2}}$ to $K^{3\over 2}$ of the form $\pa_x^2 -
v(x)$. Such an oper is said to have a regular singularity at the point
$x=t$ with residue $\varpi(\la+1)=\frac{1}{2}(\lambda+1)^2$ if its
expansion near this point has the form
$$
\pa_x^2 - \tfrac{\la(\la+2)}{4(x-t)^2} + O(\tfrac{1}{x-t}),\ x\to t
$$
(where for $t=\infty$ we take the expansion in $1/x$). 

%Here $\varpi$ denotes the natural projection $\h^* \to \h^*/W$, where
%$\h$ is the Cartan subalgebra of $\g = {\mathfrak s}{\mathfrak l}_2$
%which we identify with the Cartan subalgebra $\h^\vee$ of $\g^\vee =
%{\mathfrak s}{\mathfrak l}_2$ (note that $\varpi(\la+1) =
%\varpi(\la+\rho) = \varpi(-\la-\rho)$ in this case). 

\begin{lemma}    \label{nomon}
If
\begin{equation}    \label{miura}
  \pa_x^2 - v(x) = (\pa_x - u(x))(\pa_x + u(x)),
\end{equation}
where
\begin{equation}    \label{uz}
  u(x) = \sum_{i=0}^m \frac{\la_i}{2(x-t_i)} - \sum_{j=1}^n
  \frac{1}{x-w_j}
\end{equation}
for some distinct $w_j, 1 \leq j \leq n$, with $w_j \neq t_i$,
satisfying the system \eqref{BAE}, then the $PGL_2$-oper $\pa_x^2 -
v(x)$ on $\pone$ has trivial monodromy representation
\begin{equation}    \label{mon}
  \pi_1(\pone \bs \{t_0,\ldots,t_m,t_{m+1} = \infty \}) \to PGL_2
\end{equation}
and regular singularities at the points $t_i$, with residues
$\varpi(\la_i+1)$, $\lambda_i\in \Z_{\geq
  0}$, $0\le i\le m+1$.

Moreover, the converse is also true for a generic collection
$\{t_0,\ldots,t_m,t_{m+1} = \infty \}$.
\end{lemma}

\begin{proof}
If the $PGL_2$-oper $\pa_x^2 - v(x)$ has the form \eqref{miura} with
$u(x)$ given by formula \eqref{uz}, then the section 
\begin{equation}    \label{Phi}
\Phi:=\Phi(x)dx^{-{1\over 2}},\ \Phi(x) := \prod_{i=0}^m (x-t_i)^{-{\la_i\over 2}}
  \prod_{j=1}^n (x-w_j) 
\end{equation}
of $K^{-{1\over 2}}$ is a solution of the equation
\begin{equation}    \label{diffeq}
  (\pa_x^2 - v(x)) \Phi(x) = 0.
\end{equation}
Moreover, 
$$
\Phi_*(x):=\Phi(x) \int \Phi^{-2}(x)dx
$$ 
is then another, linearly
independent local solution of the same equation. This solution is a
single-valued global solution
on $\pone$ with singularities only at the points $\{ t_i \}$ if and
only if equations \eqref{BAE} are satisfied (in fact, the $j$-th
equation in \eqref{BAE} is equivalent to it having no monodromy at
$x=w_j$).

Conversely, suppose that $\pa_x^2 - v(x)$ is a $PGL_2$-oper on $\pone$
which has regular singularities at $t_i$'s with residues
$\varpi(\la_i+1)$, $\lambda_i \in \Z_{\geq 0}$ and trivial monodromy
representation \eqref{mon}. Equation \eqref{diffeq} then
  must have a solution of the form
\begin{equation}\label{Phix}
\Phi:=\Phi(x)dx^{-{1\over 2}},\ \Phi(x):= \prod_{i=0}^m
(x-t_i)^{-{\la_i\over 2}} Q(x),
\end{equation}
where $Q(x)$ is a polynomial of degree $n$. According to Theorem 13 of
\cite{SV}, the roots of this polynomial do not belong to the set $\{
t_i \}$ for generic collections $\{ t_i \}$. Writing
\begin{equation}    \label{Qzz}
Q(x) = \prod_{j=1}^n (x-w_j),
\end{equation}
we obtain formula \eqref{miura}. Moreover, formula \eqref{diffeq} then
implies that the $w_j$'s are pairwise distinct. Using the argument of
the preceding paragraph, we find that the $w_j$'s must satisfy
Bethe Ansatz equations \eqref{BAE}. This completes the proof.
\end{proof}

Lemma \ref{nomon} links Bethe Ansatz equations to monodromy-free
$PGL_2$-opers. In fact, the $j$-th equation in \eqref{BAE} is
equivalent to the property that oper \eqref{miura} has no singularity
at $x=w_j$.

Theorem \ref{SV} describes the case of generic parameters $\{ t_i
\}$. For special collections $\{ t_i \}$ the Gaudin Hamiltonians may
not be diagonalizable and/or the Bethe vectors may not give a basis of
${\mc H}$ (see the examples in Subsection 
\ref{exbethe}). Nevertheless, it turns out that the joint spectrum of
the Gaudin Hamiltonians is in bijection with the monodromy-free
$PGL_2$-opers satisfying the conditions of Lemma \ref{nomon} {\bf for
  all} collections $\{ t_i \}$, as stated in the following
theorem. Its part (1) was proved in \cite{F:icmp}; part (2) was
conjectured in \cite{F1} and proved in \cite{R} (without explicit
construction of eigenvectors).

\begin{theorem}\label{BAMsl2}
  (1) The joint eigenvalues $\{ \mu_i
  \}_{i=0,\ldots,m}$ of the Gaudin Hamiltonians $\{ G_i
  \}_{i=0,\ldots,m}$ acting on the space \eqref{spst} are such that
  the $PGL_2$-oper on $\pone$
  \begin{equation}    \label{PGL2op}
L(\bmu)=\pa_x^2 - \sum_{i=0}^m \frac{\la_i(\la_i+2)}{4(x-t_i)^2} -
\sum_{i=0}^m \frac{\mu_i}{x-t_i}
  \end{equation}
has regular singularity with residue $\varpi(\la_{m+1}+1)$ at $t_{m+1}=\infty$
and trivial monodromy representation \eqref{mon}.

(2) For any collection $\{ t_0,t_1,\ldots,t_{m+1}=\infty \}$ this
defines a one-to-one correspondence between the set of joint
eigenvalues (without multiplicity) of the Gaudin Hamiltonians and the
set of all such $PGL_2$-opers.
\end{theorem}

\begin{remark}    \label{Qop}
1. Under this correspondence, the eigenvalues $\mu_i$ of the Gaudin operators
$G_i$ are given by the formula
$$
\mu_i={\rm Res}_{t_i}\left(\frac{\Phi''(x)}{\Phi(x)}\right)=\la_i\left(\sum_{k\ne i}\frac{\la_k}{2(t_i-t_k)}-\sum_{j=1}^n \frac{1}{t_i-w_j}\right),
$$
where $\Phi(x)dx^{-{1\over 2}}$ is the solution \eqref{Phix} (compare
with formula \eqref{mui}).

The conditions that $L(\bmu)$ has a regular singularity and
residue $\varpi(\la_{m+1}+1)$ at $t_{m+1}=\infty$ are equivalent to
the conditions on the eigenvalues $\mu_i$'s given by the first and
second equations in \eqref{muik}, respectively.

2. The equation $L(\bmu)\beta=0$ can we written in the form 
of connection with first order poles in two different ways. Namely, setting 
$$
\bold b:=
\binom{\partial_x\beta+(\sum_{j=0}^m\frac{\lambda_j}{2(x-t_j)})\beta}{\beta},
$$
for the first way we get
\begin{equation}\label{firs}
\pa_x \bold b=\begin{pmatrix} \sum_{j=0}^m\frac{\lambda_j}{2(x-t_j)}& \sum_{j=0}^m \frac{\widehat \mu_j}{x-t_j}\\ 1 &  -\sum_{j=0}^m\frac{\lambda_j}{2(x-t_j)}\end{pmatrix}\bold b.
\end{equation}
The second way is the same but replacing $\lambda_j$ with $-\lambda_j-2$. Note that near $x=t_j$ 
equation \eqref{firs} in the variable $z:=x-t_j$ looks like
$$
\pa_z \bold b=\begin{pmatrix} \frac{\lambda_j}{2z}+a_j(z)& \frac{\widehat \mu_j}{z}+c_j(z)\\ 1 & -\frac{\lambda_j}{2z}-a_j(z) \end{pmatrix}\bold b
$$
where $a_j,c_j$ are regular at $z=0$. So we can make the residue of the matrix on the right hand side independent of $\widehat\mu_j$ by setting $\widetilde{\bold b}:={\rm diag}(z^{\frac{1}{2}},z^{-\frac{1}{2}})\bold b$. Then we get 
$$
\pa_z \widetilde{\bold b}=\begin{pmatrix} \frac{\lambda_j+1}{2z}+a_j(z)& \widehat \mu_j+zc_j(z)\\ \frac{1}{z} & -\frac{\lambda_j+1}{2z}-a_j(z) \end{pmatrix}\widetilde{\bold b}.
$$
Now the residue  is the regular element 
$$
\begin{pmatrix} \frac{\lambda_j+1}{2}& 0\\ 1 & -\frac{\lambda_j+1}{2} \end{pmatrix}
$$
which maps to $\varpi(\lambda_j+1)=\frac{1}{2}(\lambda_j+1)^2$ under the map $A\mapsto {\rm Tr}A^2$. 
\end{remark}

%According to Lemma \ref{nomon}, for generic $\{ t_i \}$ the
%monodromy-free $PGL_2$-opers \eqref{PGL2op} can be written in the form
%\eqref{miura} with $u(x)$ given by \eqref{uz} with the $w_i$'s
%satisfying \eqref{BAE}.

%According to part (3) of the above theorem, the spectrum of Gaudin
%Hamiltonians is simple for generic $\{ t_j \}$ \cite{R}. Moreover, one
%can write an explicit formula for the eigenvectors (this is not so for
%a general group, as we mentioned earlier). These vectors are known as
%{\bf Bethe vectors}. They correspond to solutions 

\subsection{The case of a general Lie algebra $\g$}    \label{Gcase}

Theorem \ref{BAMsl2} generalizes to an arbitrary simple Lie algebra
$\g$ as follows.

\begin{theorem}\label{BAM}
(1) For any collection $\{ t_0,\ldots,t_{m+1}=\infty \}$, there is a
  one-to-one correspondence between the set of joint eigenvalues
  (without multiplicity) of the algebra ${\mc G}$ of generalized
  Gaudin Hamiltonians on the space ${\mc H}$ given by \eqref{spaceH}
  and the set of all $G^\vee$-opers on $\pone$ with regular
  singularities and residues $\varpi(-\la_i-\rho)$ at $t_i$ and with
  trivial monodromy representation
  $$
  \pi_1(\pone \bs \{t_0,\ldots,t_m,t_{m+1}=\infty \}) \to G^\vee.
  $$

(2) For generic $\{ t_0,\ldots,t_m \}$, the algebra ${\mc G}$ of
  generalized Gaudin Hamiltonians is diagonalizable and has simple
  spectrum on the space ${\mc H}$ given by \eqref{spaceH}.
\end{theorem}

In \cite{F1}, part (1) of this theorem was conjectured and a map in
one direction (from the set of joint eigenvalues to the set of
monodromy-free opers) was constructed. The bijectivity of this map
(and hence the statement of part (1)) was proved in \cite{R}. Part (2)
was also proved in \cite{R}.

Analytic Langlands correspondence provides a novel conceptual
framework for (and in many cases, solution of) the problem of
diagonalization of the Gaudin Hamiltonians. In particular, Theorem
\ref{BAMsl2} will be derived using this framework in Subsection
\ref{comgr}. We will then extend this framework to
infinite-dimensional representations.

\subsection{The Gaudin model with complex weights} \label{comwei}
The theory of Subsections \ref{monfre}, \ref{sl2case}, and \ref{Gcase}
can be ``analytically continued'' to complex weights $\lambda_j$ such
that $\alpha:=\sum_{i=0}^m \la_i-\la_{j+1}^*\in Q_+$ (nonnegative
integer linear combination of simple roots).  Namely, we may replace
the finite dimensional modules $V_{\lambda_j}$, $0\le j\le m$, by
contragredient Verma modules $\nabla(\la_j)$ over $\g$ with highest
weights $\la_j$, and consider the action of Gaudin Hamiltonians in
\begin{equation}    \label{HH}
\mathcal H:=(\nabla({\la_0})\otimes...\otimes \nabla(\la_{m}))^{\n_+}[\la_{m+1}^*].
\end{equation}

First consider the case of $\sw_2$, so $\la_{m+1}^* =
  \la_{m+1}$ and equation \eqref{weight} is satisfied for a
  non-negative integer $n$. Analytically continuing the explicit formulas for the Bethe vectors
$v_{\bold w}$ and using Theorem \ref{SV}, we obtain the following
result.

\begin{theorem}    \label{sl2verma}
For generic collections $\{ t_0,t_1,\ldots,t_{m+1} =
  \infty \}$ and $\{ \la_0,\la_1,\ldots,\la_{m+1} \}$ satisfying
  formula \eqref{weight} with a non-negative integer $n$:

  (1) The Bethe vectors $v_{\bold w}$ given by formula \eqref{bvec}
  with ${\bold w} = \{ w_1,\ldots,w_n \}$ satisfying the system
  \eqref{BAE} form an eigenbasis of the space ${\mc H}$ given by
  \eqref{HH}. The eigenvalue $\mu_i$ of $G_i$ on $v_{\bold w}$ is
  given by the formula \eqref{mui}.

  (2) The spectrum of the Gaudin Hamiltonians is simple and the set of
  their joint eigenvalues is in one-to-one correspondence with the
  set of $PGL_2$-opers on $\pone$ with regular singularities and
  residues $\varpi(\lambda_i+1)$ at $t_i, i=0,\ldots,m+1$, and with
  {\bf solvable monodromy}, i.e. contained in a Borel subgroup $B^\vee
  \subset PGL_2$. Under this correspondence, the
    collection $\{ \mu_i \}$ of joint eigenvalues of the Gaudin
    Hamiltonians $G_i, i=0,1,\ldots,m$, maps to the $PGL_2$-oper given
    by formula \eqref{PGL2op}.
\end{theorem}

\begin{proof}
Suppose that $\lambda_i$ and $\lambda_{m+1}^* = \lambda_{m+1}$ are
dominant integral weights. Then we have a canonical inclusion (up to a
scalar)
\begin{equation}    \label{inclusion}
(V_{\la_0}\otimes...\otimes V_{\la_m})^{\n_+}[\la_{m+1}]\hookrightarrow
(\nabla(\la_0)\otimes...\otimes \nabla(\la_m))^{\n_+}[\la_{m+1}] =
  {\mc H},
\end{equation}
which commutes with the Gaudin Hamiltonians.

For any fixed $n$ given by formula \eqref{weight}, this map is an
isomorphism for sufficiently large $\{ \la_i \}$. Suppose that this is
the case. Theorem \ref{SV} then implies that for generic $\{ t_i
\}$ the Bethe vectors $v_{\bold w}$ given by formula \eqref{bvec},
with ${\bold w} = \{ w_1,\ldots,w_n \}$ satisfying the system
\eqref{BAE}, form an eigenbasis of ${\mc H}$, and moreover, the
spectrum of the Gaudin Hamiltonians on ${\mc H}$ is simple. Since this
is an open condition, the same is true for generic $\{ \la_i \}$ and
$\{ t_i \}$. Moreover, explicit calculation shows that the eigenvalues
$\mu_i$ of the Gaudin Hamiltonians $G_i$ are still given by formula
\eqref{mui}. This proves part (1).

Consider the corresponding oper $L(\bmu)$ on $\pone$ given by formula
\eqref{PGL2op}. By construction, it has regular singularities and
residues $\varpi(\lambda_i+1)$ at $t_i, i=0,\ldots,m+1$. Since
equations \eqref{BAE} are satisfied, this oper can be written as the
Miura transformation \eqref{miura}, where $u(z)$ is given by formula
\eqref{uz}. Therefore $\Phi$ given by formula \eqref{Phix} is a
solution of the equation $L(\bmu) \Phi = 0$. This implies that the
monodromy of $L(\bmu)$ is contained in a Borel subgroup of $PGL_2$.

Conversely, suppose that the numbers $\bmu = \{ \mu_i \}$ are such
that the $PGL_2$-oper $L(\bmu)$ satisfies the conditions of the
theorem. According to Lemma \ref{soln} below, if
  $\lambda_{m+1}\notin \lbrace -2,-3,...,-n-1\rbrace$, then $L(\bmu)$
  is equal to the Miura transformation \eqref{miura} of $u(x)$ given
  by formula \eqref{uz}. The set of numbers ${\bold w} = \{ w_j \}$
  appearing in $u(x)$ then must satisfy equations \eqref{BAE}. But
then the corresponding Bethe vector $v_{\bold w}$ is an eigenvector of
the $G_i$'s with the eigenvalues $\mu_i$'s. Since we know that these
vectors form an eigenbasis for generic $\{ t_i \}$ and 
  $\{ \la_i \}$, we obtain the statement of part (2).
\end{proof}

  \begin{lemma}    \label{soln}
Let $L(\bmu)$ be a $PGL_2$-oper of the form \eqref{PGL2op} with $\la_i
\in \C$ that has a regular singularity at $\infty$ with residue
$\varpi(\la_{m+1}+1)$, where $\lambda_{m+1}\notin \lbrace
-2,-3,...,-n-1\rbrace$ and satisfies equation \eqref{weight} with a
non-negative integer $n$. Then for any collection $\{
  \la_i \}$ and generic $\{ t_i \}$ the equation $L(\bmu) \Phi = 0$
  has a unique solution of the form
\begin{equation}\label{Phix1}
\Phi:=\Phi(x)dx^{-{1\over 2}},\ \Phi(x):= \prod_{i=0}^m
(x-t_i)^{-{\la_i\over 2}} Q(x),
\end{equation}
where $Q(x)$ is a polynomial of degree $n$ with $n$ distinct roots
$w_1,\ldots,w_n$ that are also distinct from $\{ t_i \}$
  and satisfy the Bethe Ansatz equations \eqref{BAE}.  Moreover, if
$\la_i \notin \{ 0,1,\ldots,n-1 \}$ for all $i=0,\ldots,m$, then this
is so for all $\{ t_i \}$.

Equivalently, under the above the above conditions, $L(\bmu)$ is equal
to the Miura transformation \eqref{miura} of $u(x)$ given by formula
\eqref{uz}.
  \end{lemma}

  \begin{proof}
Substituting \eqref{Phix1} into the equation $L(\bmu) \Phi = 0$, we
obtain the equation
    \begin{equation}\label{opereq1} 
\left(\partial_x^2-\sum_{i=0}^m\frac{\lambda_i}{x-t_i}\partial_x-
\sum_{i=0}^m \frac{\wh\mu_i}{x-t_i} \right) Q(x) = 0,
    \end{equation}
        where
        $$
          \wh\mu_i := \mu_i -\sum_{j\ne
          i}\frac{\lambda_i\lambda_j}{2(t_i-t_j)}.
        $$
        
        One of the characteristic exponents of this equation at $\infty$ is
equal to $n$.  Hence for any $\lambda_i\in \Bbb C, i=0,\ldots,m+1$,
satisfying \eqref{weight} there is a solution of this differential
equation of the form $x^n+\sum_{j=1}^\infty Q_jx^{n-j}$, as long as
the coefficients $Q_1,...,Q_n$ are uniquely determined by this
condition. If this is the case, we set $Q_j=0$ for $j>n$, and this
will give us a solution since it does so when all $\lambda_i$'s
are positive integers. The standard theory of ODE implies that
$Q_1,...,Q_n$ are indeed uniquely determined if and only if the second
characteristic exponent of equation \eqref{opereq1} is not in $\lbrace
0,1,...,n-1\rbrace$, which translates into the condition
$\lambda_{m+1}\notin \lbrace -2,-3,...,-n-1\rbrace$.

Denote by $w_1,\ldots,w_n$ the roots of $Q(x)$ counted
  with multiplicity. If $w_j \neq t_i$ for all $i=0,\ldots,m$,
  equation \eqref{opereq1} implies that $w_j$ is a simple root.

Suppose now that $w_j = t_i$ for a some $i$ and $j \in J_i \subset \{
1,\ldots,n\}$. Then the expansion of the solution \eqref{Phix1} of the
equation $L(\bmu) \Phi = 0$ near $x=t_i$ is equal to
  $$
  (x-t_i)^{-{\la_i\over 2}+|J_i|}(1 +
  O(x-t_i))
  $$
  up to a non-zero scalar factor. But since the leading
  term of $L(\bmu)$ at $x=t_i$ is $\la_i(\la_i+2)/4(x-t_i)$, this is
  only possible if $|J_i| = \la_i+1$ which means that $\la_i \in \{
  0,1,\ldots,n-1 \}$.

 If $\la_i \notin \{ 0,1,\ldots,n-1 \}$ for all
    $i=0,\ldots,m$, we find that $w_j \neq t_i$ for all $i$ and any
    collection $\{ t_i \}$. If $\la_i \in \{ 0,1,\ldots,n-1 \}$ for
    some $i$, then we find that $w_j \neq t_i$ for all $i$ provided
    that the collection $\{ t_i \}$ is sufficiently generic. That's
    because we know from \cite{SV} that this is so when all $\la_i$'s
    are non-negative integers (see the last paragraph of the proof of
    Lemma \ref{nomon}).
\end{proof}

%Note that one of the characteristic exponents of this equation at
%$\infty$ is equal to $n$.  Hence for {\it arbitrary} parameters
%$\lambda_i\in \Bbb C$ satisfying \eqref{we%ight} we may {\bf define}
%$\bold Q(x)$ as a solution of this differential equat%ion of the form
%$x^n+\sum_{j=1}^\infty Q_jx^{n-j}$, as long as the coefficients%
%$Q_1,...,Q_n$ are uniquely determined by this condition; in this
%case, we set %$Q_j=0$ for $j>n$, and this will give a solution since
%it does so for positive integer $\lambda_j$. The standard theory of
%ODE implies that $Q_1,...,Q_n$ are uniquely determined iff the second
%characteristic exponent of equation \eqref{opereqq} is not in
%$\lbrace 0,1,...,n-1\rbrace$, which translates into the condition
%$\lambda_{m+1}\notin \lbrace -2,-3,...,-n-1\rbrace$.
%It follows that the operator $\bold Q(x)$ is well defined 
%for any $t_j$ and $\la_j$ satisfying \eqref{weight} as long as $\lambda_{m+1}\n%otin \lbrace -2,-3,...,-n-1\rbrace$.

Motivated by Theorem \ref{BAM}, it is natural to conjecture the
following statements for a general Lie algebra $\g$.

%This statement, however, has not been proved, as far as we know (see
%the discussion in Subsection \ref{genv} below).

%  It is natural to conjecture, by
%analogy with the finite-dimensional case, that this parametrization of
%eigenvalues is also valid for all $\lambda_i$ and all $t_i$ (though
%the Gaudin Hamiltonians may not be diagonalizable or may not have
%simple spectrum).

\begin{conjecture}    \label{conjsl2}
(1) For all $\{ t_i \}$ and $\{ \la_i \}$, there is a one-to-one
  correspondence between the set of joint eigenvalues of 
    the algebra ${\mc G}$ of generalized Gaudin Hamiltonians
   on ${\mc H}$ given by formula
  \eqref{HH} (without multiplicity) and the set of $G^\vee$-opers on $\pone$ with regular
  singularities and residues $\varpi(-\lambda_i-\rho)$ at $t_i,
  i=0,\ldots,m+1$, and solvable monodromy.

(2) Suppose that $\lambda_i$ and $\lambda_{m+1}^*$ are 
dominant integral weights. Then the monodromy of the corresponding $G^\vee$-opers is unipotent, and the inclusion  \eqref{inclusion} corresponds to the inclusion of the set of opers
  with trivial monodromy into the set of opers with unipotent
  monodromy.

(3) Let $I\subset \lbrace 1,...,m\rbrace$ and 
$$
\mathcal H:=(\bigotimes_{i\in I} V_{\la_i}\otimes \bigotimes_{i\notin I} \nabla(\la_i))^{\n_+}[\lambda_{m+1}^*],
$$
where $\la_i$ are dominant integral for $i\in I$. Then the set of joint eigenvalues of
 the algebra ${\mc G}$ on $\mathcal H$ is in one-to-one correspondence with the set
  of $G^\vee$-opers as above with solvable monodromy and trivial local monodromy
  around the points $t_i$, $i\in I$. 

(4) More generally, suppose that  for $i=0,...,m$, $\mathfrak{p}_i$ are parabolic subalgebras of $\g$ containing the positive Borel subalgebra $\b_+\subset \g$, $\nabla(\la_i,\mathfrak{p}_i)$ are parabolic contragredient Verma modules for 
$\mathfrak{p}_i$, and 
$$
\mathcal H:=(\bigotimes_{i=0}^m \nabla(\la_i,\mathfrak{p}_i))^{\n_+}[\lambda_{m+1}^*].
$$
Then the set of joint eigenvalues of the algebra ${\mc G}$ on $\mathcal H$ is in one-to-one correspondence with the set of $G^\vee$-opers as above with solvable monodromy 
and local monodromy at $t_i$ belonging to $Z(L_i^\vee)U_i^\vee\subset G^\vee$, where $L_i^\vee, U_i^\vee\subset P_i^\vee$ are the Levi factor and unipotent radical of the positive parabolic $P_i^\vee\subset G^\vee$ dual to $\mathfrak{p}_i$, and $Z(L_i^\vee)$ is the center of $L_i^\vee$. 
\end{conjecture}

Note that part (3) is a special case of (4) where $\mathfrak{p}_i=\g$ if $i\in I$ and $\mathfrak{p}_i=\b_+$
if $i\notin I$. 

In a follow-up paper \cite{EF}, we plan to prove one direction of part
(1); namely, construct an injective map from the set of joint
eigenvalues of ${\mc G}$ to the corresponding set of $G^\vee$-opers
with solvable monodromy. It follows from the results
  \cite{MTV1} that if $\g=\sw_n$ and the set $\{ \la_i \}$ is generic,
  then the cardinality of the subset of the set of $PGL_n$-opers from
  part (1) corresponding (via the Miura transformation) to solutions
  of the Bethe Ansatz equations is not greater than the dimension
  of ${\mc H}$ given by \eqref{HH}. If for genetic
  $\{ \la_i \}$ all $PGL_n$-opers from part (1) have this form and
  the spectrum of ${\mc G}$ is simple, this would imply part (1) for
  $\g=\sw_n$ and generic $\{ \la_i \}$.

%Also, in Subsection \ref{chirlan} below we will consider the case when
%$\lambda_{m+1} = -r-1$, $r \in \Bbb Z_{\geq 0}$.

\begin{remark} Still more generally, one may consider 
the action of the algebra ${\mc G}$ on
$$(M_0\otimes...\otimes M_m)^{\mathfrak n_+}[\lambda_{m+1}^*],$$
where $M_i$ are arbitrary objects from category $\mathcal O$. 
It would be very interesting to parametrize the spectrum of this action, 
but here we don't even have a conjectural picture yet. 
\end{remark} 

\subsection{Examples of Bethe vectors}    \label{exbethe}

In this subsection for readers convenience we give some (well known)
examples of Bethe vectors in the space ${\mc H}$ defined
  by formula \eqref{HH} and various phenomena related to them.

\begin{example}\label{simexam} Let $\g=\sw_2$ and assume
    that $\la_j$ is generic for all $0\le j\le m$, and let $n=1$. 
Then $\dim \mathcal H=m$ and the Bethe Ansatz equation  has the form 
$$
\sum_{i=0}^m \frac{\la_i}{w-t_i}=0,
$$
which reduces to a polynomial equation of degree $m$ in $w$, so its solution set $S\subset \Bbb C$ has cardinality $m$. For every $w\in S$ the Bethe vector $v_w$ 
has the form:
\begin{equation}\label{bethvec}
v_w=\left(\sum_{i=0}^m \frac{f_i}{w-t_i}\right)v,
\end{equation} 
where $v$ is the tensor product of the highest weight vectors of
$\nabla(\la_i)$ and $f_i$ is $f\in \mathfrak{sl}_2$ acting in the
$i$-th factor (this is a special case of formula \eqref{bvec}).  These vectors form a basis in $\mathcal H$.
\end{example} 

For special weights, however, the Gaudin Hamiltonians may fail to be semisimple and 
the Bethe Ansatz equations may have fewer solutions than $\dim
\mathcal H$ (possibly none at all), or infinitely many solutions. Thus
Bethe vectors may fail to form a basis of $\mathcal H$. Such things
happen, for instance, when $\lambda_j$ are dominant integral for $0\le
j\le m$ while $\lambda_{m+1}^*$ is integral, but not dominant (see \cite{F1}).  

\begin{example} Let $m=2$, so we have four ramification points $t_0,t_1,t_2,\infty$. 
Let $E_k$ be the space of $\mathfrak{n_+}$-invariant vectors in 
$\nabla(\la_0)\otimes \nabla(\la_1)\otimes \nabla(\la_2)$ of weight 
$\sum_{i=0}^2 \la_i-2k$. 
Thus $\dim E_k=k+1$. 

In $E_0$ we have a unique up to scaling Bethe vector, which is merely the tensor product $v$ of highest weight vectors of $\nabla(\la_i)$. The eigenvalues of $G_i$ on $v$ are $\mu_i^0:=\sum_{j\ne i}\frac{\la_i\la_j}{2(t_i-t_j)}$. 

Now consider Bethe vectors in $E_1$. The Bethe Ansatz equation has the form 
$$
\sum_{i=0}^2\frac{1}{w-t_i}=0, 
$$
i.e., 
\begin{equation}\label{BAE11}
P(\bla,\bold t,w):=\la_0(w-t_1)(w-t_2)+\la_1(w-t_0)(w-t_2)+\la_2(w-t_0)(w-t_1)=0,
\end{equation} 
so it is quadratic if $\sum_{i=0}^2\la_i\ne 0$, 
and for generic $t_j$ has two solutions $w_\pm$ 
giving rise to two Bethe vectors $v_+,v_-$ which form a basis of 
$E_1$. The eigenvalues of $G_i$ on the vectors $v_\pm$ are $\mu_i^\pm:=\mu_i^0-\frac{\la_i}{t_i-w_\pm}$. 

Now consider the case $\sum_{i=0}^2\la_i=0$. In this case 
$$
E_k={\rm Hom}(\Delta(-2k),\nabla(\la_0)\otimes \nabla(\la_1)\otimes \nabla(\la_2)),
$$ 
so we have an injective map $R=\Delta_3(f): E_0\to E_1$ which is defined by restricting 
of a homomorphism $\Delta(0)\to \nabla(\la_1)\otimes \nabla(\la_2)\otimes \nabla(\la_3)$ 
to $\Delta(-2)\subset \Delta(0)$. Thus $Rv\in E_1$ is an eigenvector of $G_i$ with eigenvalues 
$\mu_i^0$ (as $[R,G_i]=0$). 

The vector $Rv$ is {\bf not}, however, a Bethe vector of the form \eqref{bvec}. Namely, 
if $\sum_{i=0}^2\la_i=0$ then the quadratic term in the Bethe Ansatz equation \eqref{BAE11} 
drops out and it becomes linear, so has only one finite solution 
$$
w_+=-\frac{\la_0t_1t_2+\la_1t_0t_2+\la_2t_0t_1}{\la_0t_0+\la_1t_1+\la_2t_2}
$$
(provided that the denominator $\sum_{i=0}^2\la_it_i$ is nonzero). The second solution $w_-$ escapes to $\infty$ as we approach the hyperplane $\sum_{i=0}^2\la_i=0$. Thus we obtain only one Bethe vector $v_+$; the second vector $v_-=Rv$ is a limit of Bethe vectors from generic $\lambda_i$, but is not itself a Bethe vector, as it does not correspond to a finite solution of the Bethe Ansatz equations. Nevertheless, the vectors $v_+,v_-$ are still an eigenbasis 
in $E_1$ for the operators $G_i$, which for generic $t_i$ still act regularly and semisimply on this space. 

It is instructive to consider the resonance case when $\sum_{i=0}^2 \la_i=\sum_{i=0}^2 \la_it_i=0$ but $\lambda_i$ don't vanish simultaneously. 
Then the Bethe Ansatz equations have no solutions, so there are no 
Bethe vectors at all. The operators $G_i$ are not semisimple on $E_1$ in this case, 
so their only joint eigenvector in $E_1$ up to scaling is $v_-$. 

Finally, in the most degenerate case $\la_0=\la_1=\la_2=0$, 
the Bethe Ansatz equation is vacuous, so any $w\in \Bbb C$ is a solution. 
In this case, $G_i$ act by scalars on $E_1$ and there are infinitely many Bethe vectors 
given by \eqref{bethvec}. 

A slightly more interesting example is $\sum_{i=0}^2\lambda_i=2$ (assuming that otherwise $\lambda_i$ are generic). In this case we have the injective restriction map $R=\Delta_3(f): E_1\to E_2$. Generically we have a basis of Bethe vectors $v_+,v_-$ of $E_1$ as above, so we have eigenvectors $Rv_+,Rv_-$ of $G_i$ with the same eigenvalues. 

These vectors are not Bethe vectors, however. Indeed, 
the Bethe Ansatz equations for $E_2$ have the form 
$$
\sum_{i=0}^2 \frac{\lambda_i}{w_1-t_i}=\frac{2}{w_1-w_2},\quad \sum_{i=0}^2 \frac{\lambda_i}{w_2-t_i}=\frac{2}{w_2-w_1},
$$
which implies that  
$$
2(w_1-t_0)(w_1-t_1)(w_1-t_2)=(w_1-w_2)P(\bla,\bold t,w_1),
$$
$$
2(w_2-t_0)(w_2-t_1)(w_2-t_2)=(w_2-w_1)P(\bla,\bold t,w_2).
$$
This is a system of two cubic equations in $w_1,w_2$, so for generic $\lambda_i$ by Bezout's theorem it has $3\cdot 3=9$ solutions. However, three of these solutions are $w_1=w_2=t_i$, $i=0,1,2$, which are not solutions of the Bethe Ansatz equations. This leaves us with $6$ solutions, or $3$ modulo the symmetry exchanging $w_1$ and $w_2$, each defining a Bethe vector. But on the hyperplane $\sum_{i=0}^2 \la_i=2$ two of these solutions run away to infinity, 
tending to $(\infty,w_\pm)$. This leaves us with just one solution $\bold w$ 
which gives rise to a single Bethe vector $v_{\bold w}$. The vectors $v_{\bold w},Rv_+,Rv_-$ 
form a basis of $E_2$. 
\end{example} 

\subsection{Analytic Langlands correspondence over $\Bbb R$ for
  compact groups}\label{comgr}

In this subsection and the next, we incorporate the Gaudin model for
finite-dimensional representations into the framework of the analytic
Langlands correspondence.

Let $F=\Bbb R$ and 
$X=\Bbb P^1$ with the usual real structure and distinct real ramification
points $t_0,...,t_{m+1}$ (in this cyclic order). As before,
we will set $t_{m+1} = \infty$, so $t_0<...<t_m$. Let the group $G^{\sigma_i}$, $0\le i\le m+1$,
be the compact form $G_c$ of $G$ for all $i$. Place the
finite-dimensional irreducible representation $V_{\la_i}$ of $G_c$ with
highest weight $\la_i$ at the point $t_i$. Then, as we explained in Subsection \ref{ram2}, the Hilbert space of the analytic Langlands correspondence is 
defined precisely by formula \eqref{spaceH}. Also the quantum Hitchin Hamiltonians 
are exactly the Gaudin Hamiltonians in this case. 

It remains to explain why in our setting the monodromy-free condition
is precisely the topological reality condition on spectral opers
arising from the analytic Langlands correspondence. Let us do so for
$G=SL_2$. We expect that this argument can be generalized
  to all simple Lie groups $G$.

In view of Remark \ref{moregen}, we can define the Hecke operator $H_{x,\overline x}$ 
for $x\in \Bbb C$ (with coweight $1$ of $SL_2=(G/(\pm 1))^\vee$ attached to both $x$ and $\overline x$). Recall that since the real locus $X(\Bbb R)$ is a quaternionic oval, 
the eigenvalue of the Hecke operator $H_{x,\overline x}$ is given by formula \eqref{quateq}: 
$$
\bbe(x,\overline x)={\rm Im}(\overline{f_j^+(x)}g_j^+(x)),\ {\rm Im}(x)>0
$$
near the interval $(t_j,t_{j+1})$, where $f_j^+,g_j^+$ are the basic local solutions of the oper equation $L\beta=0$ near $t_j$ 
such that
$$
g_j^+(t_j+u)=\pi u^{\frac{\lambda_j}{2}+1}(1+ug_j^0(u)),
$$
$$
f_j^+(t_j+u)=u^{-\frac{\lambda_j}{2}}(1+uf_j^0(u))+\gamma_ju^{\frac{\lambda_j}{2}+1}(1+ug_j^0(u))\log u,
$$
where $f_j^0(u),g_j^0(u)\in \Bbb R[[u]]$, $\gamma_j\in \Bbb R$. 
Consider also the basic local solutions $f_j^-,g_j^-$ of the oper equation near $t_j$ such that for small $u>0$ 
$$
g_j^-(t_j-u)=\pi u^{\frac{\lambda_j}{2}+1}(1-ug_j^0(-u)),
$$
$$
f_j^-(t_j-u)=u^{-\frac{\lambda_j}{2}}(1-uf_j^0(-u))-\gamma_ju^{\frac{\lambda_j}{2}+1}(1-ug_j^0(-u))\log u.
$$
Then the half-monodromy matrix along the upper (respectively, lower) half-circle around $t_j$ 
between the bases $f_j^+,g_j^+$ and $f_j^-,g_j^-$ is  
$$
J_j^{\pm}=\begin{pmatrix} i^{\mp\lambda_j} & 0\\ i^{\mp (1+\lambda_j)}\gamma_j & -i^{\mp\lambda_j}\end{pmatrix} 
$$
Thus the real analytic continuation along the upper half-circle transforms 
$\bbe={\rm Im}(\overline{f_j^+}g_j^+)$ to $\pm{\rm Im}(\overline{f_j^-}g_j^-)\mp \gamma_j|g_j^-|^2$. So to preserve the vanishing condition for $\bbe$ on the real locus, we must have $\gamma_j=0$. 
This means that the monodromy of the oper $L$ around $t_j$, which is given by the matrix $(J_j^-)^{-1}J_j^+=(-1)^{\lambda_j}$ is trivial in $PGL_2$ for all $j$, i.e. the oper is monodromy-free. 

\subsection{The connection between Hecke operators and Baxter's
  $Q$-operator}\label{baxop}

Let us go back to the case of the Gaudin model with the space of
states ${\mc H}$ given by formula \eqref{HH}. The following result is
proved in the same way as Lemma \ref{soln}.

\begin{theorem}    \label{Q-op}
For arbitrary collections $\{ t_i \}$
and $\{ \la_i \}$ satisfying
\eqref{weight} with
$$
\lambda_{m+1}\notin \lbrace
-2,-3,...,-n-1\rbrace,
$$
there is a unique linear operator $\bold Q(x)$
acting on $\mathcal H$ given by formula \eqref{HH} commuting with the
Gaudin Hamiltonians
$G_i$, which is a monic polynomial in $x$ of degree
$n$ satisfying the universal oper equation (compare with Proposition
\ref{opereq}(i)):
\begin{equation}\label{opereqq} 
\left(\partial_x^2-\sum_{i=0}^m\frac{\lambda_i}{x-t_i}\partial_x\right)\bold
Q(x)-\bold Q(x)\sum_{i=0}^m\frac{\widehat G_i}{x-t_i}=0,
\end{equation}
     where
      $$
      \widehat G_i:=G_i-\sum_{j\ne
        i}\frac{\lambda_i\lambda_j}{2(t_i-t_j)}.
      $$
\end{theorem}

In particular, if $v \in {\mc H}$ is an eigenvector of the $G_i$'s
with eigenvalues $\bmu = \{ \mu_i \}$, we have
$$
\bold Q(x)v=Q_v(x)v,
$$
where $Q_v(x)$ is the polynomial appearing in the corresponding
solution \eqref{Phix1} of the equation $L(\bmu) \Phi = 0$.

As we will see in Subsection \ref{qdefo}, this
operator can be obtained as the $q\to 1$ limit of the celebrated
Baxter $Q$-operator introduced by R. Baxter in the study of integrable
quantum spin chains. For this reason we call $\bold Q(x)$ the {\bf
  Baxter Q-operator of the Gaudin model} (or Baxter's $Q$-operator for
short).

On the other hand, as we explain presently, this operator $\bold Q(x)$
may be viewed as an algebraic version of the Hecke operators of the
analytic Langlands correspondence.

Let us switch to the setting of the previous subsection (for
$G=SL_2$). Thus, ${\mc H}$ is now given by formula \eqref{spst}. The
proof of Theorem \ref{Q-op} (which follows the proof of Lemma
\ref{soln}) carries over without changes to this case. Thus, we obtain
a $Q$-operator in this case as well, for which we will use the same
notation. Moreover, the condition $\lambda_{m+1}\notin \lbrace
-2,-3,...,-n-1\rbrace$ becomes vacuous in this case because all
$\la_i$'s are assumed to be dominant integral.

It is clear that the resulting operator is the restriction of the
$Q$-operator acting on the tensor product of contragredient Verma
modules with dominant integral weights (which is described in Theorem
\ref{Q-op}) to the tensor product of the corresponding
finite-dimensional representations under the embedding
\eqref{inclusion}.

Consider now the finite-dimensional case with ${\mc H}$ given by formula
\eqref{spst}. Let $L(\bmu)$ be a monodromy-free $PGL_2$-oper corresponding
via Theorem \ref{BAMsl2} to a set of joint eigenvalues $\bmu = \{ \mu_i \}$
of the Gaudin Hamiltonians acting on ${\mc H}$. Formula \eqref{quateq} implies
that the eigenvalue $\bbe(x,\overline x)$ of the Hecke operator
$H_{x,\overline x}$ corresponding to the $PGL_2$-oper $L(\bmu)$ is,
up to scaling by an $x$-independent constant, given by
$$
\bbe(x,\overline x)\sim |\Phi(x)|^2{\rm Im}\int_{x_0}^x \Phi^{-2}(z)dz,\ {\rm Im}(x)\ge 0,
$$ 
where $\Phi$ is defined by \eqref{Phix}, and the same expression with 
a minus sign if ${\rm Im}(x)<0$. Here the lower limit $x_0$ of integration can be any point of $X(\Bbb R)$; recall from Subsection \ref{comgr} that if the oper is monodromy-free
then the imaginary part of the integral is independent on the choice of this point. 

Thus, again up to a constant, 
$$
H_{x,\overline x}\sim \prod_{i=0}^m |x-t_i|^{-\la_i}\bold Q(x)^\dagger\bold Q(x){\rm Im}\int_{x_0}^x \bold Q^{-2}(z)\prod_{i=0}^m (z-t_i)^{\la_i}dz,\ {\rm Im}(x)>0,
$$
$$
\Bbb H_{x,\overline x}\sim \bold Q(x)^\dagger\bold Q(x){\rm Im}\int_{x_0}^x \bold Q^{-2}(z)\prod_{i=0}^m (z-t_i)^{\la_i}dz,\ {\rm Im}(x)>0.
$$
Thus we see that the Hecke operator can be expressed in a
  rather direct way in terms of Baxter's $Q$-operator of the Gaudin model.

\subsection{Analytic Langlands correspondence for discrete series representations}\label{disser} Consider now another
example of the analytic Langlands correspondence in the setting
of the Gaudin model, which involves
discrete series representations of $SL_2(\Bbb R)$. As in Subsection
\ref{comgr},
we take $F=\Bbb R$, $X=\Bbb P^1$ with the standard real structure and
{\bf real} ramification points 
$t_0,...,t_{m+1} \in {\mathbb R}$, but now set $G^\sigma$ to be the
split real group $SL_2(\Bbb R)$.
So this setting is somewhat different from the one of Subsection \ref{comgr}, where we deal with compact groups; they are related by analytic continuation in highest weights. 

Let $V_i=\widehat \Delta(-r_i)$ be the completion of the Verma module 
for $0\le i\le m$, and $V_{m+1}=\widehat\Delta(-r_{m+1})^*$, where $r_i\in \Bbb Z_{\ge 1}$. So $V_i$ is a holomorphic discrete series representation for all $0\le i\le m$, while $V_{m+1}$ is an antiholomorphic  discrete series representation.\footnote{More precisely, the representations 
$\widehat{\Delta}(-1)$, $\widehat{\Delta}(-1)^*$ are limit of discrete series representations.}
Then
$$
{\mathcal H}_{V_0,...,V_{m+1}}=\Hom(\Delta(-r_{m+1}),\Delta(-r_0)\otimes...\otimes \Delta(-r_{m}))
$$
$$
=(\Delta(-r_0)\otimes...\otimes \Delta(-r_m))^{{\mathfrak n}_+}_{-r_{m+1}}. 
$$
So the Hilbert space is finite dimensional in this
case. It is non-zero if
    and only if
\begin{equation}    \label{ri}
  r_{m+1} - \sum_{i=0}^m r_i = 2n,
\end{equation}
where $n$ is a non-negative integer.
It is easy to see that the Gaudin operators $G_i$ are
self-adjoint, and hence diagonalizable, in this case.

%As before, the Quantum Hitchin Hamiltonians are the Gaudin operators.

\begin{theorem}\label{disseri}
For all real $\{ t_i \}$ and all $\{ r_i\in \Bbb Z_{\ge 1} \}$ the
Gaudin Hamiltonians are diagonalizable on ${\mathcal
  H}_{V_0,...,V_{m+1}}$ with simple joint
  spectrum and there is a one-to-one correspondence between
the set of their joint eigenvalues and the set of $PGL_2$-opers on
$\pone$ with regular singularities and residues $\varpi(-r_i+1)$ at
$t_i, i=0,\ldots,m+1$, and solvable monodromy.
\end{theorem}

\begin{proof}
 According to Theorem \ref{sl2verma}, this statement holds for generic
 collections $\{ t_i \}$ and $\{ \lambda_i = -r_i \}$. However,
 Corollary 2.4.6 of \cite{V} implies that in the case that all $r_i$'s
 are positive (so that all $\lambda_i$'s are negative) and all $t_i$'s
 are real, the Bethe vectors $v_{\bold w}$
   corresponding to the solutions ${\bold w} = \{ w_1,\ldots,w_n \}$
   (with $n$ defined by formula \eqref{ri}) of the Bethe Ansatz
   equations \eqref{BAE} form a basis of eigenvectors of the Gaudin
 Hamiltonians. As before, for each such solution ${\bold w}$, denote by
 $\bmu = \{ \mu_i \}$ the corresponding set of joint eigenvalues of
 the Gaudin Hamiltonians $\{ G_i \}$. Then the
 $PGL_2$-oper $L(\bmu) = \pa^2_x-v(x)$ given by formula \eqref{PGL2op}
 (with $\la_i = -r_i$) is equal to the Miura
 transformation \eqref{miura} of $u(x)$ given by formula \eqref{uz}.
 Equivalently, equation $L(\bmu) \Phi = 0$ has a solution $\Phi$ given by
 formula \eqref{Phix1} with
 \begin{equation}    \label{Qzz1}
Q(x) = \prod_{j=1}^n (x-w_j).
 \end{equation}
This shows that all opers $L(\bmu)$ corresponding to the joint
eigenvalues $\bmu$ of the Gaudin Hamiltonians have solvable monodromy
representation.

Suppose that the joint spectrum of the Gaudin Hamiltonians is not
simple. Then there are two different solutions ${\bold w}_1$ and
${\bold w}_2$ of the Bethe Ansatz equations \eqref{BAE} such that the
corresponding joint eigenvalues coincide; that is, $\bmu_1 = \bmu_2$
and so $L(\bmu_1) = L(\bmu_2)$. But then equation $L(\bmu_1) \Phi = 0$
has two different solutions $\Phi_1$ and $\Phi_2$ of the form
\eqref{Phix1} with the monic polynomials $Q_1(x)$ and $Q_2(x)$ of
degree $n$ of the form \eqref{Qzz1} corresponding to ${\bold w}_1$ and
${\bold w}_2$. Therefore, $\Phi = \Phi_1-\Phi_2$ is also a solution of
$L(\bmu_1) \Phi = 0$, such that the corresponding polynomial $Q(x) =
Q_1(x)-Q_2(x)$ is non-zero and has degree $0\le k \le n-1$.

It follows from our condition \eqref{ri} that near $t_{m+1}=\infty$,
any solution of the equation $L(\bmu_1) \Phi = 0$ has the expansion
(up to a non-zero scalar) either $y^{r_m/2}(1 + O(y))$ or
$y^{(-r_{m+1}+2)/2}(1 + O(y))$, where $y = x^{-1}$. Therefore, if
$L(\bmu_1) \Phi = 0$ has a solution of the form \eqref{Phix1} with
$Q(x)$ of degree $0\le k\le n-1$, then $r_{m+1} = n-k+1$. But this is
impossible because \eqref{ri} also implies that $r_{m+1} > 2n$. This
shows that the joint spectrum of the Gaudin Hamiltonians is simple.

Conversely, suppose that a $PGL_2$-oper $L(\bmu) = \pa_x^2-v(x)$ has
regular singularities with residues $\varpi(-r_i+1)$ at $t_i,
i=0,\ldots,m+1$, where $r_i \in \Bbb Z_{\ge 1},
i=0,\ldots,m+1$. Then the $r_i$'s must satisfy equation
  \eqref{ri}, which implies that $r_{m+1} \notin \lbrace
  2,3,...,n+1\rbrace$. Therefore, by Lemma \ref{soln}, for all $\{
t_i \}$, the $PGL_2$-oper $L(\bmu)$ is equal to the Miura
transformation \eqref{miura} of $u(x)$ given by formula \eqref{uz}
with $\la_i = -r_i$, with all $w_j$'s being distinct and different from the
$t_i$'s. But then the set of numbers ${\bold w} = \{ w_j \}$ appearing
in $u(x)$ must satisfy Bethe Ansatz equations \eqref{BAE}, so the
corresponding Bethe vector $v_{\bold w}$ is an eigenvector of the
$G_i$'s with the eigenvalues $\mu_i$'s. Since we know from above that these
vectors form an eigenbasis, this completes the proof.
\end{proof}

%Thus the eigenvectors and eigenvalues can be found by the Bethe
%Ansatz method a%s explained above.

\begin{remark}\label{projre1} 1. We hope that this statement can be
    proved within the framework of the analytic Langlands
    correspondence, using the results of of Subsection \ref{fixpo}.

2. The above argument showing the simplicity of the joint
  spectrum of the Gaudin Hamiltonians can also be used to prove the
  simplicity of the joint spectrum of the Gaudin Hamiltonians in the setting of
  Subsection \ref{comgr}.

3. The above proof shows that the statement of Theorem \ref{disseri} remains true if the $r_i$'s are arbitrary positive real numbers constrained by equation \eqref{ri}. If $r_i$ is not an integer, one cannot define an action of $SL(2,{\Bbb R})$ on a completion of $\Delta(-r_i)$. But it can be interpreted as a unitary representation of the universal cover of $SL(2,{\Bbb R})$, and according to Remark \ref{projre}, such representations can also included into the framework of the analytic Langlands correspondence.
\end{remark} 

\subsection{Chiral analytic Langlands correspondence} \label{chirlan}

In this subsection we consider another setting in which the analytic
Langlands correspondence can be linked to a particular Gaudin model. Let $F=\Bbb C$, $X=\Bbb P^1$ with ramification points $t_0,...,t_m,t_{m+1}=\infty$, and suppose that the representations $V_j$ of $G$ are holomorphic. Then in the definition of the Hecke operator $H_{x,\lambda}$ we may replace the integral $\int_{{\mathcal Z}_{\lambda,x}(P)}\psi(Q)dQd\overline Q$ over the complex variety ${\mathcal Z}_{\lambda,x}(P)$ by the ``contour" integral $\int_{C}\psi(Q)dQ$, where $C\subset {\mathcal Z}_{\lambda,x}(P)$ is a real cycle, and by Cauchy's theorem the result is stable under deformations of $C$. Since $V_j$ are not Hermitian, the corresponding space $\mathcal H$ will not carry a Hermitian form, but we may still consider eigenvectors and eigenvalues of Hecke operators (in the spirit of Remark \ref{schwar}).

Admittedly, there are very few holomorphic irreducible representations of $G$ (only the finite dimensional ones), but we may in fact take $V_j$ to be certain representations of the Lie algebra $\g={\rm Lie}G$ which do not necessarily integrate to $G$. We call this setting the {\bf  chiral analytic Langlands correspondence}. Let us show how it works in an example with $G=SL_2$. 

For $\lambda\in \Bbb C$ let $\Delta(\lambda),\nabla(\lambda)$ be the Verma and contragredient Verma modules over $\sw_2$ with highest weight $\la$. 
For $0\le i\le m$ take $V_i:=\nabla(\lambda_i)$, and let $V_{m+1}=\Delta(\lambda_{m+1})^*$ be the graded dual of $\Delta(\lambda_{m+1})$, i.e., the contragredient Verma module with lowest weight $-\lambda_{m+1}$. Assume that $\sum_{i=0}^m\lambda_i-\lambda_{m+1}=2n$ and that $\lambda_{m+1}=r-1$, where $n,r\in \Bbb Z_{\ge 0}$. Then the space 
$$
\mathcal H:={\rm Hom}(V_{m+1}^*,V_0\otimes...\otimes V_m)=
{\rm Hom}(\Delta(r-1),\nabla(\lambda_0)\otimes...\otimes \nabla(\lambda_m))
$$
has dimension $\binom{n+m}{m}$; it is the space of singular vectors 
in $V_0\otimes...\otimes V_m$ of weight $r-1$. We realize $\nabla(\lambda)$ as 
$\Bbb C[y]$ with the action of $\g$ given by \eqref{sl2for}. Then the space 
$\mathcal H$ is realized as the space of homogeneous polynomials of degree 
$n$ in $y_0,...,y_m$ which are invariant under simultaneous translations. 
Consider also the space 
$$
\mathcal H':={\rm Hom}(\Delta(-r-1),\nabla(\lambda_0)\otimes...\otimes \nabla(\lambda_m))
$$
of dimension $\binom{n+r+m}{m}$; it is the space of homogeneous polynomials of $y_0,...,y_m$ of degree $n+r$. Then we can define the (modified) {\bf  chiral Hecke operator} $\Bbb H_x^{\rm chir}: \mathcal H\to \mathcal H'$ by analogy with 
\eqref{heckoper}, replacing integration over $\Bbb C$ by contour integration: 
\begin{equation}\label{polyno}
(\Bbb H_x^{\rm chir}\psi)(y_0,...,y_m)=\oint \psi\left(\frac{t_0-x}{s-y_0},...,\frac{t_m-x}{s-y_m}\right)\prod_{j=0}^m (s-y_j)^{\lambda_j}ds,
\end{equation} 
where $\oint f(s)ds$ denotes the residue of $f(s)ds$ at $\infty$. Note that this residue is well defined since $\sum_{i=0}^m \lambda_i=n+r-1\in \Bbb Z$, hence the integrand is single-valued near $\infty$. So this formula makes sense even though the $\g$-modules $V_i$ do not integrate to $G$.  

Note that we have an inclusion $\iota: \Delta(-r-1)\hookrightarrow \Delta(r-1)$, so we have the restriction 
operator $R: \mathcal H\to \mathcal H'$. 

\begin{lemma}\label{restr} Under suitable normalization of $\iota$, we have 
$$
(R\psi)(y_0,...,y_m)=\oint \psi\left(\frac{1}{s-y_0},...,\frac{1}{s-y_m}\right)\prod_{j=0}^m (s-y_j)^{\lambda_j}ds.
$$
\end{lemma} 

\begin{proof} The proof is similar to the proof of Lemma \ref{Rpsi}. We have 
$$
R\psi=\Delta_{m+1}(\tfrac{f^r}{r!})\psi=\oint \Delta_{m+1}(\exp(s^{-1}f)\psi) s^{r-1}ds.
$$
The 1-parameter group generated by $f\in \mathfrak{sl}_2$ consists of fractional linear transformations $z\mapsto \frac{z}{tz+1}$. Therefore  
$$
R\psi=\oint \psi\left(\frac{y_0s}{s-y_0},...,\frac{y_ms}{s-y_m}\right)s^{r-1-\sum_j \lambda_j}\prod_{j=0}^m (s-y_j)^{\lambda_j}ds=
$$
$$
\oint \psi\left(\frac{1}{s-y_0},...,\frac{1}{s-y_m}\right)\prod_{j=0}^m (s-y_j)^{\lambda_j}ds.
$$
\end{proof} 

By Lemma \ref{restr}, $\Bbb H_x^{\rm chir}\sim x^nR$, $x\to \infty$. 
Also it is easy to show that $\Bbb H_x^{\rm chir}$ satisfies the universal oper equation of Proposition \ref{opereq}(i) (the proof is the same as for $\Bbb H_x$), hence it commutes with Gaudin Hamiltonians. So the discussion of Section \ref{sl2case} implies 

\begin{proposition} The chiral Hecke operator is the composition 
of the restriction operator and the Baxter $Q$-operator:
$$
\Bbb H_x^{\rm chir}=R\bold Q(x).
$$
\end{proposition} 

%In particular, we see that for generic parameters the space $\mathcal H$ has a %basis of {\bf  Bethe vectors} $v_{\bold w}$ which satisfy the equation 
%$$
%\Bbb H_x^{\rm chir}v_{\bold w}=RQ_{\bold w}(x)v_{\bold w}.
%$$

%Thus the spectral opers in this case are opers of the form \eqref{PGL2op} 
%with {\bf  upper triangular monodromy} in some basis, and upper left entry $e^{%-\pi i\lambda_j}$
%at $t_j$ and $e^{\pi i\lambda_{m+1}}$ at $t_{m+1}=\infty$. (The second solution%of the oper equation need not be rational here, so the monodromy does not have %to be trivial).\footnote{Note, however, that the eigenvectors $Rv_{\bold w}$ of% $G_i$ are not Bethe vectors, as explained in Example \ref{simexam}.}

\subsection{Infinite dimensional generalizations, including principal series} \label{infdi}

The above discussion shows that tamely ramified analytic Langlands correspondence in genus $0$ may be viewed as a generalization of the Gaudin model in which we can consider the action of Gaudin Hamiltonians on tensor products of any admissible representations of a real reductive group $G$, so that  the space of states can be infinite dimensional. For instance, in the case $G=SL_2$ and $F=\Bbb R$ (Subsection \ref{sl2case}), we may replace finite dimensional representations $V_i$ of $SU(2)$ by principal series representations $\mathcal V_i$ of $SL_2(\Bbb R)$ (taking the real locus 
$X(\Bbb R)=\Bbb R\Bbb P^1\subset X(\Bbb C)=\Bbb C \Bbb P^1$ to be real rather than quaternionic). 
Then the Gaudin operators $\lbrace G_i\rbrace$ act as self-adjoint
unbounded strongly commuting operators on the Hilbert space
$$
\mathcal H:={\rm Mult}_{SL_2(\Bbb R)}(\mathcal V_{m+1}^*,\mathcal
V_0\otimes...\otimes \mathcal V_m).
$$
So the spectral problem for $\lbrace G_i\rbrace$ becomes analytic and can no longer be solved by algebraic Bethe Ansatz. Nevertheless, 
the description of
the spectrum of the operators $\lbrace G_i\rbrace$ in terms of
monodromy-free opers (see Theorem \ref{BAM}) can be generalized: the
spectrum is parametrized (at least conjecturally) by the set of {\bf
  balanced opers} with prescribed residues, as explained in Subsection
\ref{twica}. Namely, we have the following result.

\begin{theorem} The set of joint joint eigenvalues $\mu_i$ of the
  Gaudin operators $G_i$ on the Hilbert space ${\rm Mult}_{SL_2(\Bbb
    R)}(\mathcal V_{m+1}^*,\mathcal V_0\otimes...\otimes \mathcal
  V_m)$ is in one-to-one correspondence with the set of $\bLa$-balanced opers
on $\pone$ of the form
$$
L(\bmu)=\pa_x^2 - \sum_{i=0}^m \frac{\la_i(\la_i+2)}{4(x-t_i)^2} -
\sum_{i=0}^m \frac{\mu_i}{x-t_i},
$$
subject to the equations 
\begin{equation}    \label{eqs}
\sum_{i=0}^{m} \mu_{i}=0,\ \sum_{i=0}^{m} t_i\mu_{i}=\frac{\lambda_{m+1}(\lambda_{m+1}+2)}{4}-\sum_{i=0}^m \frac{\lambda_i(\lambda_i+2)}{4}, 
\end{equation}
where $\lambda_j$ are the parameters of the principal series representations $\mathcal V_j:=L^2(\Bbb R\Bbb P^1,|K|^{-\frac{\lambda_j}{2}})$ and $\bLa=(\Lambda_0,...,\Lambda_{m+1})$, $\Lambda_j=-ie^{\frac{\pi i\lambda_j}{2}}$. 
\end{theorem} 

Indeed, this follows from the fact that the Gaudin operators $\lbrace G_i\rbrace$ strongly commute with Hecke operators, hence have the same spectral decomposition, and the latter is described in 
Subsection \ref{twica}. 

Similarly, one can generalize to the infinite dimensional case the setting of Subsection \ref{disser}. Namely, we may consider the situation when $V_i=\widehat \Delta(-r_i)$, 
$0\le i\le p$, and $V_i=\widehat \Delta(-r_i)^*$ for $p+1\le i\le m+1$, i.e., 
we have $p$ holomorphic discrete series representations and $m-p+1$ antiholomorphic ones. Then the Hilbert space $\mathcal H$ is infinite dimensional if $1\le p\le m-1$, so the problem is again no longer algebraic. Since discrete series representations are composition factors of (non-unitary) principal series representations, one can presumably 
generalize the analysis of Subsection \ref{twica} to describe the spectrum in this case, but 
we will not discuss this here. 

\subsection{Double Gaudin model}    \label{double}

The discussion of the analytic Langlands correspondence for a group
$G$ on $\pone$ with parabolic points over $\C$ in Subsection
\ref{Differential equations for Hecke operators} can be framed as a
``double'' of the Gaudin model for the Lie algebra $\g$ in which we
combine holomorphic and anti-holomorphic degrees of freedom.

For concreteness, consider the case of $SL_2$. Then we associate to
the points $t_i, i=0,1,\ldots,m+1$, with $t_{m+1} = \infty$ as before
(these are the parabolic points) spherical unitary principal series
representations of $SL_2(\Bbb C)$, which are completions of {\bf
  spherical Harish-Chandra bimodules}, i.e., spaces of finite type
linear maps between Verma modules for $\mathfrak{sl}_2$. We have two
commuting actions of the Lie algebra $\sw_2$ on such
representations (by left and right multiplications), which we can view
as holomorphic and anti-holomorphic. Therefore, on such
representations act not only the above Gaudin Hamiltonians $G_i$ but
also their complex conjugates $\overline G_i$. Hence it is natural to
form the generating series $S(x)$ given by formula \eqref{Sz} and its
complex conjugate $\ol{S}(\ol{x})$.

We have proved in \cite{EFK2,EFK3} that the Hecke operator $H_x$
(which is a section of $K^{-1/2} \otimes \ol{K}^{-1/2}$ on $\pone$)
satisfies a system of two second-order differential
equations \eqref{diffeq1}, see Corollary \ref{spec2}.

The corresponding eigenvalues of the Hecke operators are single-valued
$C^\infty$ solutions of this system with $S(x)$ and $\ol{S}(\ol{x})$
replaced by the corresponding eigenvalues $v(x)$ and $\ol{v}(x)$. As
explained in Subsection \ref{ALcomp}, such single-valued solutions
exist if and only if the oper $\pa_x^2 - v(x)$, where $v(z)$ is given
by formula \eqref{PGL2op}, has real monodromy, and such opers are
called {\bf real opers}. Theorem \ref{mutatis} implies the following
result.

\begin{theorem}    \label{dG}
If $\{ \mu_i \}$ is the set of joint eigenvalues of the
Gaudin operators $\{ G_i \}$ on the Hilbert space 
\eqref{hilspace}
then the $PGL_2$-oper
$$
L(\bmu)=\pa_x^2 - \sum_{i=0}^m \frac{\la_i(\la_i+2)}{4(x-t_i)^2} -
\sum_{i=0}^m \frac{\mu_i}{x-t_i}
$$
on $\pone$ satisfies equations \eqref{eqs} and has real monodromy. The
corresponding eigenvalues of $\ol{G}_i$ are equal to $\ol\mu_i$.
\end{theorem}

Moreover, we conjecture that there is a one-to-one correspondence
between the spectrum of Gaudin operators $\{ G_i \}$ on this Hilbert
space and the set of such opers. This is an extension of a conjecture
from \cite{EFK3} from the case when all $\la_i=-1$ to the case when
$\la_i \in -1 + {\Bbb R}\sqrt{-1}$. In the case when $m=3$
or $4$ it is proved in the same way as in \cite{EFK3}.

\subsection{$q$-deformation of the Gaudin model}\label{qdefo}

The $q$-deformation of this Gaudin model is known as the {\bf XXZ model}. The
role of the affine Kac-Moody algebra $\wh{\sw}_2$ is now played by the
corresponding quantum affine algebra $U_q(\wh{\sw}_2)$.\footnote{There is also an intermediate deformation called the XXX model, which preserves the classical $\mathfrak{sl}_2$ symmetry; its full symmetry algebra is the (doubled) Yangian $Y(\mathfrak{sl}_2)$, which is a degeneration of $U_q(\widehat{\mathfrak{sl}}_2)$. The material of this subsection applies mutatis mutandis to the XXX model.} 
  In particular,
the role of the representations $V_{\la_k}$ attached to the points
$t_k$ is now played by the corresponding evaluation representations of
$U_q(\wh{\sw}_2)$, and the role of the space ${\mc H}$ given by
\eqref{spst} is played by the space ${\mc H}_q$ in which we replace
each $V_{\la_k}$ by the corresponding evaluation representation (see
\cite{FH} for more details).

The role of the generating function $S(x)$ of the Gaudin Hamiltonians
given by formula \eqref{Sz} is now played by the {\bf  transfer-matrix}
${\mb T}(z)$ associated to the two-dimensional evaluation
representation of $U_q(\wh{\sw}_2)$ corresponding to $z$ (see \cite{FH}).

The $q$-deformation of the Hecke operator $H_z$ is now closely related to
the transfer-matrix ${\mb H}(z)$ corresponding to an
infinite-dimensional representation of a Borel subalgebra of
$U_q(\wh{\sw}_2)$ called prefundamental, see \cite{FH}. Just as in the
$q=1$ case, it is equal to the product of a scalar factor that depends
only on the values of $\{ t_k \}$ and $\{ \la_k \}$ and an operator
${\mb Q}(z)$ acting on ${\mc H}_q$ whose eigenvalues are polynomials
in $z$. The latter is is known as {\bf Baxter's $Q$-operator}.

Moreover, there is an analogue of the second-order differential
equation \eqref{diffeq1}. It is a second-order {\bf  difference}
equation
\begin{equation}    \label{differ}
  (D^2 - D{\mb T}(z)  + 1) {\mb H}(z) = 0,
\end{equation}
where $D$ is a shift operator, $(D \cdot f)(z) := f(zq^2)$, i.e.,
$D:=q^{2z\partial}$. If we write ${\mb T}(z) = 2 + h^2 z^2 S(z) + \ldots$, where $q=e^\hbar$,
expand equation \eqref{differ} in a power series in $\hbar$, divide by
$\hbar^2$ and set $\hbar=0$ (i.e. $q=1$), we obtain equation \eqref{diffeq1}.
In this sense, equation \eqref{differ} is indeed a $q$-deformation of
equation \eqref{diffeq1}. 

The second-order difference operator in the
RHS of \eqref{differ} is  known as a $q$-{\bf  oper} for the group
$SL_2$. The notion of a $q$-oper has been defined for an arbitrary
simple Lie group in \cite{FKSZ}.

Relation \eqref{differ} can be interpreted as a relation in the
Grothendieck ring of a certain category of representations of a Borel
subalgebra of $U_q(\wh{\sw}_2)$, see e.g. \cite{FH}. It is often
written in the form
\begin{equation}   {\mb T}(z) = \frac{{\mb H}(zq^2)}{{\mb H}(z)} + \frac{{\mb
    H}(zq^{-2})}{{\mb H}(z)},
\end{equation}
and it can be further rewritten as a $q$-difference equation on ${\mb
  Q}(z)$ known as the  {\bf Baxter $TQ$-relation}:
\begin{equation}    \label{differ2}
{\mb T}(z) = A(z)\frac{{\mb Q}(zq^2)}{{\mb Q}(z)} + D(z)\frac{{\mb
    Q}(zq^{-2})}{{\mb Q}(z)}
\end{equation}
(here $A(z)$ and $D(z)$ are functions that only depend on the
parameters $\{ t_k \}$ and $\{ \la_k \}$ of ${\mc H}_q$, see e.g.
\cite{FH}, formula (1.1)). It has been conjectured by Baxter and
 proved in \cite{FH} that the eigenvalues of ${\mb Q}(z)$ on ${\mc
  H}_q$ are polynomials whose roots satisfy a $q$-deformation of the
Bethe Ansatz equations \eqref{BAE} (see \cite{FH}, Section  5.6). 

At the level of the eigenvalues, we obtain  that the Baxter
$TQ$-relation \eqref{differ2} expressing the eigenvalues of ${\mb
  T}(z)$ on ${\mc H}_q$ is a $q$-deformation of the Miura
transformation \eqref{miura} expressing the eigenvalues of $S(z)$ on
${\mc H}$. This was one of the key observations of 
\cite{FR} (Section 1.3 and Remark 6), see also \cite{F:icmp} (Section
6.8). (Moreover,
both maps preserve natural Poisson structures 
\cite{FR}.)

An insight of the present paper is that Baxter's  $Q$-operator
${\mb Q}(z)$ may be viewed as a $q$-deformation of 
an operator closely related to 
the Hecke operator $H_z$ of the tamely ramified analytic Langlands
correspondence for $SL_2$ on $\pone$. There is a similar statement for an arbitrary simple
algebraic group $G$. This opens the door to the investigation of
$q$-deformation of the analytic Langlands correspondence
(which we expect to exist in the ramified setting in genus $0$ and $1$). 

\subsection{$q$-deformation of the double Gaudin model}\label{qdefo1}
In conclusion, let us speculate what this $q$-deformation of the
analytic Langlands
correspondence should look like for $F=\Bbb C$ and $G=SL_2$ in the case of 
$X=\pone$ with parabolic structures. As we explained in Subsection
\ref{double}, for $q=1$ this is the double Gaudin model.

In the undeformed situation (with twists) we place at the 
parabolic points the spherical unitary principal series representations of $SL_2(\Bbb C)$, which 
are completions of spherical Harish-Chandra bimodules, on which we
have an action of the Gaudin Hamiltonians
$G_i$ and their complex conjugates $\overline G_i$. Hence it is
natural to form the generating series $S(z)$ given by formula
\eqref{Sz} and its complex conjugate $\ol{S}(\ol{z})$.

As shown in Subsection \ref{Differential equations for Hecke
  operators} (following \cite{EFK2,EFK3}), the Hecke operator $H_z$
satisfies the following system of two second-order differential
equations:
\begin{equation}    \label{system}
(\pa_z^2 - S(z)) H_z= 0, \qquad (\ol\pa_z^2 - \ol{S}(\ol{z}))
H_z = 0.
\end{equation}
Such a solution exists if and only if the corresponding oper  has real
monodromy, and such opers are called real opers. Locally, the Hecke
eigenvalue can be written in the form
\begin{equation}    \label{bilin}
  \Psi_0(z) \ol\Psi_1(\ol{z}) + \Psi_1(z) \ol\Psi_0(\ol{z}),
\end{equation}
where $\Psi_0(z)$ and $\Psi_1(\ol{z})$ are two linearly independent
local solutions of equation \eqref{diffeq} (see \cite{EFK3}).

In light of the discussion of the previous subsection, we expect that the 
$q$-deformation of this picture for $0<q<1$ should look as follows. We
also expect that a similar picture exists for $F=\Bbb R$, and also for
elliptic curves.

1. The
spherical unitary principal series representations of $SL_2(\Bbb C)$ 
should be replaced by their $q$-analogs defined in \cite{Pu}; they are 
completions of spherical Harish-Chandra $U_q(\mathfrak{sl}_2)$-bimodules, which are 
spaces of finite maps between Verma modules for $U_q(\sw_2)$.

2. We should view these as bimodules over 
$U_q(\wh{\sw}_2)$ via the evaluation homomorphism. 
Then we can define the transfer-matrices ${\mb
  T}(z)$ and $\ol{\mb T}(\ol{z})$ which are the $q$-analogues of
$S(z)$ and $\ol{S}(\ol{z})$.

3. The Hecke operator ${\mb H}_z$ should be given by a suitable 
$q$-deformation of the formulas of Subsection \ref{heopp} satisfying 
the following universal $q$-oper equations: 
\begin{equation}    \label{differ3}
  (D^2 - D{\mb T}(z)  + 1) {\mb H}_z = 0, \qquad (\ol{D}^2 -
  \ol{D}\ol{\mb T}(\ol{z})  + 1) {\mb H}_z= 0,
\end{equation}
where $D:=q^{2z\partial}$, $\overline D:=\overline
  q^{2\overline{z\partial}}$.

The challenge is to make sense of the system \eqref{differ3} which at
the outset is only well-defined if the eigenvalues of ${\mb H}_z$
extend to analytic functions in $z$ and $\ol{z}$ (viewed as
independent variables). We may also consider the setting where $q=e^\hbar$ and $\hbar$ 
is a formal variable, in which this analytic subtlety disappears. Finally, we need to identify the property of the solutions of \eqref{differ3} that is a $q$-analogue of the existence of a single-valued $C^\infty$ solution of the oper equations.

%The solution ${\mb H}_z$ of this system should be expressible as a
%bilinear combination of solutions of each of the two equations,
%analogously to formula \eqref{bilin} in the differential case.

Answering this question should lead us to the notion of a {\bf real
$q$-oper}. We leave this interesting topic for a future paper.

\end{document}